\documentclass[10pt]{article}
\usepackage[margin=1.15in]{geometry} 
\usepackage{amsmath,amsthm,amssymb}
\usepackage{float}
 \usepackage{hyperref}
 \usepackage{authblk}

\newcommand{\E}{\mathbb{E}}
\newcommand{\ignore}[1]{}
\usepackage[dvipsnames]{xcolor}
\usepackage{multirow}
\usepackage{stmaryrd}
\usepackage{bm}
\usepackage{caption}
\usepackage{cleveref}
\usepackage{tcolorbox}
\usepackage{MnSymbol}
\newtheorem{theorem}{Theorem}

\usepackage{mathtools}
\newtheorem{lemma}{Lemma}

\newtheorem{definition}{Definition}

\usepackage{algorithm}
\usepackage{algpseudocode}
\numberwithin{equation}{section}

\graphicspath{{..}}

    \makeatletter
\def\@fnsymbol#1{\ensuremath{\ifcase#1\or \dagger\or \ddagger \or 
   \mathsection\or \mathparagraph\or \|\or **\or \dagger\dagger
   \or \ddagger\ddagger \else\@ctrerr\fi}}
    \makeatother

\begin{document}
 
 
\title{Energy Stable and Structure-Preserving Algorithms for the Stochastic Galerkin System of 2D Shallow Water Equations}

\author[1]{Yekaterina Epshteyn \thanks{\href{epshteyn@math.utah.edu}{epshteyn@math.utah.edu}}}

\author[1,2]{Akil Narayan \thanks{\href{akil@sci.utah.edu}{akil@sci.utah.edu}}}

\author[1]{Yinqian Yu \thanks{ \href{yinqian.yu@utah.edu}{yinqian.yu@utah.edu}}}

\affil[1]{\small Department of Mathematics, University of Utah, 155 S 1400 E, Salt Lake City, 84112, UT, USA}
\affil[2]{\small Scientific Computing and Imaging Institute, University of Utah, 72 S Central Campus Dr, Salt Lake City, 84112, UT, USA}

\date{ }
\maketitle
\begin{abstract}
  Shallow water equations (SWE) are fundamental nonlinear hyperbolic PDE-based models in fluid dynamics that are essential for studying a wide range of geophysical and engineering phenomena. Therefore, stable and accurate numerical methods for SWE are needed. Although some algorithms are well studied for deterministic SWE, more effort should be devoted to handling the SWE with uncertainty. In this paper, we incorporate uncertainty through a stochastic Galerkin (SG) framework, and building on an existing hyperbolicity-preserving SG formulation for 2D SWE, we construct the corresponding entropy flux pair, and develop structure-preserving, well-balanced, second-order energy conservative and energy stable finite volume schemes for the SG formulation of the two-dimensional shallow water system. We demonstrate the efficacy, applicability, and robustness of these structure-preserving algorithms through several challenging numerical experiments.

\end{abstract}



\section{Introduction}\label{intro}

The Saint-Venant system of shallow water equations is widely used in mathematical and physical modeling of geophysical fluid flows, where the horizontal length scale is much greater than the vertical length scale, such as in rivers, lakes, and coastal regions \cite{de1871theorie}. The shallow water equations can capture the essential dynamics of wave propagation, currents, and other hydrodynamic behaviors under the assumption of a shallow fluid layer, making them suitable for modeling a variety of practical problems involving water flows. Despite the wide applicability of shallow water systems, shallow water models with random component are even more relevant for practical real-world situations since precise knowledge of an environment or operating conditions is frequently absent. Such random component/uncertainty can improve the predictive capabilities of the models, motivating the development of advanced methods to incorporate and manage the uncertainty in the governing shallow water systems. In this context, we will consider the \textit{parameterized} stochastic shallow water equations (SWE), with the parameter as a random variable modeling the uncertainty.

In order to discretize the parameterized SWE, we consider the polynomial chaos expansion (PCE), which is a technique for modeling uncertainty in parameterized systems. Initially introduced by Wiener \cite{wiener1938homogeneous} for Gaussian processes using Hermite polynomials, the method has since been extended to more general orthogonal polynomial bases to handle random inputs with various distributions \cite{ghanem2003stochastic,le2004uncertainty,wan2005adaptive,xiu2002wiener}. The PCE approach effectively encodes uncertainty in parameterized partial differential equations (PDEs), resulting in a system of equations that can be solved using two major classes of methods: \textit{intrusive} methods and \textit{non-intrusive} methods. Based on sampling \cite{mishra2012multilevel,nobile2008sparse,xiu2005high}, non-intrusive methods construct the polynomial solution to the parameterized system by collecting an ensemble of solutions to the deterministic form at a collection of fixed values of the parameters corresponding to random variables. This approach can utilize existing and trusted legacy solvers for the deterministic SWE, for example, see, e.g. \cite{zhou2001surface,kurganov2002central,rogers2003mathematical,vcrnjaric2004balanced,xing2005high,xing2006high,xing2006new,kurganov2007second,BEKP,epshteyn2023adaptive,kurganov2018finite,liu2018well,xing2016high,XING2017361,zhong2022entropy}. However, the solutions generated by the non-intrusive approaches may be less accurate than those from intrusive methods. Moreover, ensuring associated structure-preserving properties, including the entropy conditions, can be challenging or impossible.

In the context of PCE methods, intrusive methods typically refer to the stochastic Galerkin (SG) approach, where the truncated PCE \cite{wiener1938homogeneous,xiu2002wiener} representations of the state variables are projected using a Galerkin method in the stochastic space. This process results in a new system of deterministic PDEs, with the variables being the truncated PCE coefficients. To solve the new system, a substantial rewrite of legacy codes and PDEs solvers is required, which is one of the drawbacks of intrusive methods. However, the SG method is generally expected to be more accurate than some alternative non-intrusive methods. Since the SG method is based on projection, it leads to near-optimal accuracy in the $L^2$ sense for static problems \cite{babuska2004galerkin,le2010spectral}. Additionally, the SG approach provides opportunities to rigorously prove the desired properties of the SG system and the subsequent numerical schemes. Therefore, SG methods have been effectively employed to model uncertainty in diffusion equations \cite{xiu2009efficient,eigel2014adaptive}, kinetic equations \cite{hu2016stochastic,shu2017stochastic}, and conservation and balance laws with symmetric Jacobian matrices \cite{tryoen2010intrusive}.

However, for a general hyperbolic system of conservation and balance laws, such as the SWE system, the associated SG system may not be hyperbolic \cite{despres2013robust,gerster2019hyperbolic,jin2019study}. Thus, the SG formulation may result in a system of partial differential equations that belongs to a different class from the original deterministic system, potentially yielding unphysical solutions and compromising the robustness of subsequent numerical schemes. Numerous recent efforts have explored numerical methods for the SG formulation of various types of hyperbolic conservation laws. Some advances include SG-type analysis and algorithms for scalar conservation laws \cite{zhou2012galerkin} including well-balanced methods \cite{jin2016well}, Haar wavelet-based approaches \cite{gerster2022haar}, hyperbolicity preservation through a non-equivalent Roe variables formulation \cite{gerster2020entropies}, filtering strategies \cite{kusch2020filtered}, limiter-type methods \cite{schlachter2018hyperbolicity}, hyperbolicity formulations for linear problems \cite{pulch2012generalised} or using linearization techniques \cite{wu2017stochastic}, operator splitting \cite{chertock2015operator,chertock2015well}, non-conservative formulations of SG SWE systems \cite{chen2023cross}, and entropic variable representations \cite{poette2019,poette2009uncertainty}.


In this paper, we build upon the recent work on a hyperbolicity-preserving SG formulation for two-dimensional SWE \cite{dai2022hyperbolicity}. Despite the hyperbolicity-preserving property, if the corresponding source term vanishes, the SG SWE is a nonlinear hyperbolic conservation and balance law and hence it inherits standard challenges in developing numerical methods for such a model. For example, solutions can develop shock discontinuities in finite time with generic initial conditions, potentially producing non-unique weak solutions. Therefore, an additional entropy condition should be imposed, either implicitly or explicitly, to identify the desired physical solution. Note that implementing implicit time-integration solvers is challenging because of the nonlinearity involved in the system \cite{dafermos2016hyperbolic,leveque1992numerical,leveque2002finite}. In addition, depending on its structure, the SWE system is supposed to satisfy the \textit{well-balanced} property \cite{bermudez1994upwind}, ensuring that numerical schemes accurately capture the steady-state solution of the PDE. Moreover, uncertainty in the stochastic model complicates all challenges when considering the SG SWE, making them more ambiguous and difficult to address.

\subsection{Contributions of this paper}
In this paper, we extend the ideas developed in \cite{dai2024energy} to stochastic shallow water equations in two-dimensional physical space. The main contributions of this paper are as follows:
\begin{itemize}\label{contributions}
\item We derive an entropy flux pair for the hyperbolicity-preserving, positivity-preserving stochastic Galerkin (SG) formulation of two-spatial-dimensional shallow water equations (SWE) from \cite{dai2022hyperbolicity}. Entropy-entropy flux pairs are the theoretical
starting point towards the entropy admissibility criteria to resolve non-uniqueness of weak solutions.
\item \ignore{We develop second-order energy conservative, and first- and second-order energy stable semi-discrete finite volume methods for the SG formulation of the 2D SWE, using the specially designed entropy flux pair.} Using the entropy-entropy flux pair, we devise second-order energy conservative, and first- and second-order
energy stable finite volume schemes for the 2D SG SWE, all of which are also well-balanced. The designed energy conservative and energy stable schemes are
stochastic extensions of the schemes developed in \cite{fjordholm2011well,fjordholm2012arbitrarily} and are also the two-dimensional extensions of schemes developed in \cite{dai2024energy} for the stochastic SWE in 1D physical space.
\item We present several challenging numerical experiments to evaluate the performance of our schemes, including accuracy, well-balanced property, energy decay, and numerical robustness.
\end{itemize}
An outline of this paper is as follows: In Section \ref{Sec2}, we present preliminaries of polynomial chaos expansion (PCE) and the SG formulation of two-dimensional SWE system from \cite{dai2022hyperbolicity}. In Section \ref{Sec3}, we introduce entropy flux pairs for the deterministic SWE discussed in \cite{fjordholm2011well} and construct an entropy flux pair for the SG formulation of the two-dimensional SWE system. In Section \ref{Sec4}, we develop energy conservative and energy stable schemes, accompanied by theoretical proofs and algorithmic details. In Section \ref{Sec5}, we present several challenging numerical examples to demonstrate the performance of our schemes and verify their theoretical properties. In Section \ref{Sec6} we provide a brief summary of the main results of this paper and some potential topics for future research. Finally, we include some technical details for several proofs of lemmas and theorems in \cref{sec:entropy-tuple-proof,app:ec-proofs}.

\section{Stochastic Galerkin formulation of the two-dimensional shallow water equations}\label{Sec2}
In this section, we review the stochastic Galerkin (SG) formulation of the two-dimensional (2D) shallow water equations (SWE), which possess a hyperbolicity-preserving property, as developed in \cite{dai2022hyperbolicity}. We begin with the deterministic form of the 2D SWE:
\begin{align}\label{SWE}
  U_t + F(U)_x + G(U)_y &= S(U), &  U&=(h,q^x,q^y)^\top,
\end{align}
where $F$ and $G$ denote the fluxes in the $x$- and $y$-directions, respectively, and $S$ represents the source term,
\begin{align}\label{deterministic_SWE_flux}
  F(U) &= \begin{pmatrix}
    q^x\\ \frac{(q^x)^2}{h}+\frac{gh^2}{2}\\ \frac{q^x q^y}{h}
  \end{pmatrix}, & 
  G(U) &= \begin{pmatrix}
    q^y\\ \frac{q^x q^y}{h} \\ \frac{(q^y)^2}{h}+\frac{gh^2}{2}
  \end{pmatrix}, &
  S(U) &= \begin{pmatrix}
    0\\-gh\frac{\partial B}{\partial x}\\ -gh\frac{\partial B}{\partial y}
\end{pmatrix},
\end{align}
where $U = U(x,y,t)$ is the vector of conservative variables, $h = h(x,y,t)$ is the water height, $q^x(x,y,t)$ and $ q^y(x,y,t)$ are the discharges in the $x$- and $y$-directions, respectively, and $B(x,y)$ is the time-independent bottom topography. 
For the \textit{stochastic} shallow water equations, we consider introduction of a random field $\xi$, which could result from uncertainty or ignorance of the inputs, for example, bottom topography and initial data. We begin with some preliminaries for the model formulation of the stochastic Galerkin approach to the stochastic shallow water equations.


\subsection{Polynomial chaos expansion}\label{Sec2_1}
In this subsection, we provide a brief review of polynomial chaos expansion (PCE). For further details, see \cite{debusschere2004numerical, sullivan2015introduction, xiu2010numerical}.
Let $\xi \in \mathbb{R}^d$ be a $d$-dimensional random variable with a Lebesgue density function $\rho$. Define the $L^2$-integrable function space associated with $\rho$ as follows:
\begin{equation}
    L_{\rho}^2 (\mathbb{R}^d) \coloneqq \left\{ f: \mathbb{R}^d \to \mathbb{R} \;\;\bigg|\;\; \left( \int_{\mathbb{R}^d} f^2(s)\rho(s) ds \right)^{1/2} < +\infty \right\}.
\end{equation}
Assuming that $\rho$ has finite polynomial moments of all orders, there exists a $d$-variate orthonormal polynomial basis $\{ \phi_k\}_{k=1}^{\infty}$ such that,
\begin{equation}
  \E [ \phi_k(\xi) \phi_l(\xi)] = \langle \phi_k, \phi_l \rangle_{\rho} \coloneqq \int_{\mathbb{R}^d} \phi_k(s)\phi_l(s)\rho(s)ds = \delta_{k,l}, \quad \forall k,l \in \mathbb{N}, \quad \phi_1(\xi) \equiv 1, 
\end{equation}
where $\delta_{k,l}$ is the Kronecker delta. Further, under mild conditions \cite{ernst2012convergence}, then these basis functions span $L^2_{\rho}$: For any $z \in L_{\rho}^2$, then 
\begin{align}
  z(x,y,t,\xi) &\overset{L_{\rho}^2}{=} \sum_{k=1}^{\infty} \widehat{z}_k(x,y,t)\phi_k(\xi), & 
  \widehat{z}_k &= \left\langle z, \phi_k \right\rangle_{\rho}
\end{align}
where $x$, $y$, and $t$ are deterministic spatial and temporal variables, and $\widehat{z}_k(x,y,t)$ are the deterministic Fourier-type coefficients corresponding to the orthonormal basis $\{ \phi_k\}_{k=1}^{\infty}$. Numerical computations require finite truncations of these expansions. Let $P = \text{span}\{ \phi_k, k=1,2,\dots,K\} $ be a $K$-dimensional polynomial subspace of $ L_{\rho}^2$. We then define the $K$-term PCE approximation of a random field $z$ on this subspace as follows:
\begin{equation}\label{PCE-truncated}
    \Pi_P[z](x,y,t,\xi) \coloneqq \sum_{k=1}^{K} \widehat{z}_k(x,y,t)\phi_k(\xi).
\end{equation}
The statistics of $\Pi_P[z]$ can be derived from its expansion coefficients. Specifically, the mean and variance of the random field $z$ can be expressed in terms of these coefficients as follows:
\begin{align}
  \mathbb{E}[\Pi_P[z](x,y,t,\xi)  ] &= \widehat{z}_1(x,y,t), & \text{Var}[\Pi_P[z](x,y,t,\xi) ] &= \sum_{k=2}^K \widehat{z}_k^2(x,y,t).
\end{align}
Our numerical schemes involve specific manipulations of truncated expansion coefficients. With $\widehat{z} = (\widehat{z}_1,\dots,\widehat{z}_k)^\top \in \mathbb{R}^K$ the size-$K$ vector of truncated PCE coefficients of $z$, define the linear operator $\mathcal{P}:\mathbb{R}^K \to \mathbb{R}^{K\times K}$ as follows:
\begin{align}
  \mathcal{P}(\widehat{z}) &\coloneqq \sum_{k=1}^K \widehat{z}_k \mathcal{M}_k, & 
  \mathcal{M}_k &\in \mathbb{R}^{K\times K}, & 
  (\mathcal{M}_k)_{l,m} &\coloneqq \langle \phi_k,\phi_l\phi_m  \rangle_{\rho}.
\end{align}
Due to the symmetry and commutativity properties of the operator $\mathcal{P}(\cdot)$, the following identities hold,
\begin{align}\label{commutative_prop}
  \mathcal{P}(\widehat{z}) &= (\mathcal{M}_1 \widehat{z},\mathcal{M}_2 \widehat{z},\dots, \mathcal{M}_K \widehat{z} ), & 
  \mathcal{P}(\widehat{a})\widehat{b} &= \mathcal{P}(\widehat{b})\widehat{a}, &
  \widehat{b}^\top \mathcal{P}(\widehat{a}) &= \widehat{a}^\top \mathcal{P}(\widehat{b}),
\end{align}
The last two properties are proved in \cite[Lemma 2.1]{dai2024energy}, using the definition and symmetry of $\mathcal{P}$.
A stochastic Galerkin (SG) formulation of a $\xi$-parameterized partial differential equation (PDE) assumes that the state variable lies in the polynomial space $P$ and forms a scheme corresponding to projecting the PDE residual onto $P$. Note that the hyperbolicity of the straightforward SG formulation for nonlinear hyperbolic PDEs, such as the shallow water equations, is not automatically guaranteed. Therefore, special designs are required for the SG formulation of these nonlinear hyperbolic PDEs to preserve such an important property.

\subsection{Hyperbolic-preserving stochastic Galerkin formulation for 2D shallow water equation}\label{Sec2_2}
In this subsection, we review existing results on the hyperbolicity-preserving stochastic Galerkin formulation of the two-spatial-dimensional shallow water equations (2D SG SWE) \cite{dai2022hyperbolicity}.  We follow a standard Galerkin procedure in the stochastic space. It begins with reducing the problem to an alternative finite-dimensional form by replacing the solutions $(h,q^x,q^y)$ by the ansatz,
\begin{equation}\label{K-term-PCE}
\begin{split}
  h \simeq h_P &\coloneqq \sum_{k \in [K]} \widehat{h}_k(x,y,t)\phi_k(\xi),\\
   q^x \simeq q^x_P &\coloneqq \sum_{k \in [K]} (\widehat{q^x})_k(x,y,t)\phi_k(\xi),\\
   q^y \simeq q^y_P &\coloneqq \sum_{k \in [K]} (\widehat{q^y})_k(x,y,t)\phi_k(\xi),
\end{split}
\end{equation}
respectively, and the bottom $B$ by $\Pi_P[B]$, where we use the notation $[K] \coloneqq \{1, 2, \ldots, K\}$. With a special choice of how the Galerkin projection is applied to the nonlinear, non-polynomial terms $(q^x)^2/h, (q^y)^2/h, q^xq^y/h$ introduced in \cite{dai2022hyperbolicity}, a SG system of balance laws was derived, whose state variables are the coefficients in \eqref{K-term-PCE}
\begin{equation}\label{SGSWE}
    \widehat{U}_t + \widehat{F}(\widehat{U})_x + \widehat{G}(\widehat{U})_y = \widehat{S}(\widehat{U},\widehat{B}).
\end{equation}
Here, $\widehat{U} = \big(  \widehat{h}^\top , (\widehat{q^x})^\top, (\widehat{q^y})^\top  \big)^\top \in \mathbb{R}^{3K}$, where $\widehat{h} , \widehat{q^x}, \widehat{q^y}$ are length-$K$ vectors whose entries are the coefficients in \eqref{K-term-PCE}. The flux terms are defined by
\begin{equation}\label{exact_flux_term}
    \widehat{F}(\widehat{U}) = \begin{pmatrix}
    \widehat{q^x} \\ \mathcal{P}(\widehat{q^x})\mathcal{P}^{-1}(\widehat{h})\widehat{q^x} + \frac{1}{2}g \mathcal{P}(\widehat{h})\widehat{h}\\
    \mathcal{P}(\widehat{q^x})\mathcal{P}^{-1}(\widehat{h})\widehat{q^y}
\end{pmatrix}, \quad \widehat{G}(\widehat{U}) = \begin{pmatrix}
    \widehat{q^y} \\ 
    \mathcal{P}(\widehat{q^y})\mathcal{P}^{-1}(\widehat{h})\widehat{q^x} \\
    \mathcal{P}(\widehat{q^y})\mathcal{P}^{-1}(\widehat{h})\widehat{q^y}+ \frac{1}{2}g \mathcal{P}(\widehat{h})\widehat{h}
\end{pmatrix}. 
\end{equation}
The source term is given by
\begin{equation}\label{exact_source_term}
    \widehat{S}(\widehat{U}) = \begin{pmatrix}
    0 \\ -g\mathcal{P}(\widehat{h})\widehat{B_x} \\ -g\mathcal{P}(\widehat{h})\widehat{B_y}
\end{pmatrix}.
\end{equation}
In the deterministic case, the fluxes \eqref{exact_flux_term} and the source term \eqref{exact_source_term} reduce to their deterministic forms as shown in \eqref{deterministic_SWE_flux}. A notable observation from the above is that the deterministic term $\frac{q^x q^y}{h}$ have two different stochastic representations in \eqref{exact_flux_term}, i.e., $\mathcal{P}(\widehat{q^x})\mathcal{P}^{-1}(\widehat{h})\widehat{q^y} \neq \mathcal{P}(\widehat{q^y})\mathcal{P}^{-1}(\widehat{h})\widehat{q^x}$. This different treatment of these terms is essential to retaining hyperbolicity; a more detailed discussion of this fact and a brief empirical investigation of the difference between these two stochastic representations is provided in \cite{dai2022hyperbolicity}. 

Recall the deterministic SWE is hyperbolic under the condition that the water height $h>0$. There is a natural extension of this property to the SG formulation of the SWE.
\begin{theorem}[Theorem 3.1 in \cite{dai2022hyperbolicity}]
    If the matrix $\mathcal{P}(\widehat{h})$ is strictly positive definite at every point $(x,y,t)$ in the computational spatial-temporal domain, then the SG formulation \eqref{SGSWE} is hyperbolic.
\end{theorem}
This result is proven by identifying a symmetrizing similarity transform of the SG SWE flux Jacobians $\frac{\partial \widehat{F}}{\partial \widehat{U}}$ and $\frac{\partial \widehat{G}}{\partial \widehat{U}}$. These flux Jacobians will be useful for us later, so we explicitly provide them below: When $\mathcal{P}(\widehat{h})$ is invertible, the well-defined terms
\begin{align}\label{velocity_coeff}
  \widehat{u} &= \mathcal{P}^{-1}(\widehat{h})\widehat{q^x}, & 
  \widehat{v} &= \mathcal{P}^{-1}(\widehat{h})\widehat{q^y},
\end{align}
are stochastic representations of the $x$- and $y$-directional water velocity variables. 
I.e., these terms can be interpreted as the vectors of the PCE coefficients of the $x$-velocity $u \coloneqq q^x/h$ and the $y$-velocity $v \coloneqq q^y/h$. The flux Jacobians of the SG SWE system \eqref{SGSWE} can then expressed in terms of $K \times K$ blocks as follows:
\begin{equation}\label{SG_Jacobians}
    \begin{split}
        \frac{\partial \widehat{F}}{\partial \widehat{U}} = &\begin{pmatrix}
            0& I & 0\\
            g\mathcal{P}(\widehat{h}) - \mathcal{P}(\widehat{q^x})\mathcal{P}^{-1}(\widehat{h})\mathcal{P}(\widehat{u}) & \mathcal{P}(\widehat{q^x})\mathcal{P}^{-1}(\widehat{h}) + \mathcal{P}(\widehat{u})  & 0 \\
            -\mathcal{P}(\widehat{q^x})\mathcal{P}^{-1}(\widehat{h})\mathcal{P}(\widehat{v}) & \mathcal{P}(\widehat{v}) & \mathcal{P}(\widehat{q^x})\mathcal{P}^{-1}(\widehat{h})
        \end{pmatrix},\\
        \frac{\partial \widehat{G}}{\partial \widehat{U}} = & \begin{pmatrix}
            0 & 0 & I \\
             -\mathcal{P}(\widehat{q^y})\mathcal{P}^{-1}(\widehat{h})\mathcal{P}(\widehat{u}) &  \mathcal{P}(\widehat{q^y})\mathcal{P}^{-1}(\widehat{h}) & \mathcal{P}(\widehat{u}) \\
             g\mathcal{P}(\widehat{h}) - \mathcal{P}(\widehat{q^y})\mathcal{P}^{-1}(\widehat{h})\mathcal{P}(\widehat{v}) & 0 &  \mathcal{P}(\widehat{q^y})\mathcal{P}^{-1}(\widehat{h}) + \mathcal{P}(\widehat{v})
        \end{pmatrix}.
    \end{split}
\end{equation}

\section{An entropy flux pair for the 2D SG SWE }\label{Sec3}
To construct numerical schemes with desired energy stability properties, it is essential to derive entropy flux pairs for the SG formulation \eqref{SGSWE} of the SWE. We first review entropy flux pairs for the \textit{deterministic} SWE discussed in \cite{fjordholm2011well}.

\subsection{Entropy flux pairs for the deterministic SWE}\label{Sec3_1}
Solutions of systems of conservation and balance laws can develop shock discontinuities in finite time, even from smooth initial data, potentially leading to non-unique weak solutions defined in the distributional sense. To identify the desired physically relevant solution, an additional entropy admissibility criterion is imposed \cite{benzoni2007multi,dafermos2016hyperbolic}. For a general balance law in two spatial dimensions,
\begin{equation}\label{2D_balancelaw}
    U_t + F(U)_x + G(U)_y = S(U),
\end{equation}
an entropy flux tuple is a tuple $(E(U), H(U), K(U))$ satisfying the \textit{companion balance law},
\begin{equation}\label{2D_companion_balancelaw}
    E(U)_t + H(U)_x + K(U)_y = 0,
\end{equation}
where the entropy $E(U)$ is a scalar function that is convex in $U$, and $H$ and $K$ represent the corresponding entropy flux functions. For the consistency with the original balance law, the entropy flux pair $(E,H,K)$ is supposed to satisfy the \textit{compatibility condition} as follows:
\begin{equation}\label{compatibility_condition1}
    \frac{\partial E}{\partial U}(F_x + G_y - S) = H_x + K_y.
\end{equation}
This condition ensures that multiplying the equation \eqref{2D_balancelaw} by $\frac{\partial E}{\partial U}$ recovers the equation \eqref{2D_companion_balancelaw} for smooth solutions. When the source term vanishes and $(E, H, K) = (E(U), H(U), K(U))$, the compatibility condition \eqref{compatibility_condition1} simplifies to the usual entropy condition for conservation laws. Even though an entropy flux pair may not exist for a general system of balance laws, the companion balance law \eqref{2D_companion_balancelaw} for a hyperbolic system of balance laws, derived from continuum physics, is usually related to the Second Law of thermodynamics, with the total energy of the system often serving as the entropy function. Several related examples can be found in Section 3.3 of \cite{dafermos2016hyperbolic}. For the deterministic SWE in \eqref{SWE}, the total energy \cite{fjordholm2011well} is given by,
\begin{align}\label{Deterministic_energy}
    E^d(U) &= \frac{1}{2}(q^x u + q^y v) + \frac{1}{2}gh^2 + ghB,&   u & \coloneqq q^x/h,&   v &\coloneqq q^y/h,
\end{align}
where $\frac{1}{2}(q^x u + q^y v)$ is the kinetic energy, and $\frac{1}{2}gh^2 + ghB$ is the potential energy. A direct computation of the Hessian confirms that $E^d(U)$ is convex in $U = (h, q^x, q^y)$. If we choose the fluxes $H^d$ and $K^d$ as,
\begin{align}\label{Deterministic_fluxes}
    H^d(U) &\coloneqq \frac{1}{2}(hu^3 + huv^2) + gq^x(h+B),& \quad K^d(U) &\coloneqq \frac{1}{2}(hu^2v + hv^3) + gq^y(h+B),
\end{align}
then one can verify directly that $E^d_t + H^d_x + K^d_y = 0$, i.e., $(E^d, H^d, K^d)$ satisfies the companion balance law \eqref{2D_companion_balancelaw}. Hence, $(E^d, H^d, K^d)$ is an entropy flux tuple for \eqref{SWE}.
In the case of weak solutions with shocks, the entropy admissibility criterion requires that energy dissipates according to a vanishing viscosity principle as follows:
\begin{equation}
     E^d(U)_t + H^d(U)_x + K^d(U)_y \le  0.
\end{equation}
By extending the results of deterministic SWE, we construct an entropy flux pair for the SG formulation of the 2D SWE system \eqref{SGSWE} in the next section. This involves identifying a tuple, consisting of an entropy function that is convex in the state variable, that satisfies the companion balance law.

\subsection{An entropy flux pair for the 2D SG SWE}\label{Sec3_2}
In this subsection, we focus on constructing an entropy flux pair for the SG system \eqref{SGSWE}. We recall the discussion surrounding \eqref{SGSWE} for the definition of $\widehat{U}$. The main result of this section is stated in the following theorem.
\begin{theorem}\label{Thm_entropy_flux_pair}
    Define the entropy function
    \begin{equation}\label{entropy}
        E(\widehat{U}) = \frac{1}{2}\Big((\widehat{q^x})^\top \widehat{u} + (\widehat{q^y})^\top \widehat{v}\Big) + \frac{1}{2}g\|\widehat{h}\|^2 + g\widehat{h}^\top \widehat{B},
    \end{equation}
    and the flux functions
    \begin{equation}\label{flux_functions}
        \begin{split}
            H(\widehat{U}) = & \frac{1}{2} \Big(\widehat{u}^\top \mathcal{P}(\widehat{q^x})\widehat{u} + \widehat{v}^\top \mathcal{P}(\widehat{q^x})\widehat{v} \Big) + g (\widehat{q^x})^\top(\widehat{h}+\widehat{B}),\\
            K(\widehat{U}) = & \frac{1}{2}\Big(\widehat{v}^\top \mathcal{P}(\widehat{q^y})\widehat{v} + \widehat{u}^\top \mathcal{P}(\widehat{q^y})\widehat{u} \Big) + g(\widehat{q^y})^\top(\widehat{h}+\widehat{B}),
        \end{split}
    \end{equation}
    If $\mathcal{P}(\widehat{h})>0$, i.e., strictly positive definite, then $(E,H,K)$ is an entropy flux pair for the SG system of the two-dimensional SWE \eqref{SGSWE}.
\end{theorem}
Note that the stochastic variants of entropy function \eqref{entropy} and flux functions \eqref{flux_functions} reduce to the deterministic energy \eqref{Deterministic_energy} and fluxes \eqref{Deterministic_fluxes}, respectively. The proof of \cref{Thm_entropy_flux_pair} is similar to that of \cite[Theorem 3.1]{dai2024energy}, which proves a corresponding result for an SG formulation of the SWE in a single spatial dimension. Hence, we relegate most details to \cref{sec:entropy-tuple-proof}, and provide only an outline of the major results needed. 
We start from a lemma that computes the gradient of $\widehat{u}$ and $\widehat{v}$. This technical result is also useful later when we construct energy conservative and stable schemes.
\begin{lemma}[Gradients of $\widehat{u}, \widehat{v}$]\label{lemma_gradient}
    Let $\widehat{q^x}, \widehat{q^y}$ be arbitrary, and let $\widehat{h} \in \mathbb{R}^K$ be such that $P(\widehat{h})$ is invertible. Define $\widehat{u}, \widehat{v}$ by \eqref{velocity_coeff}, then the gradients of the velocities are
\begin{align}\label{gradient_velocity}
        \frac{\partial \widehat{u}}{\partial \widehat{U}} &= [-\mathcal{P}^{-1}(\widehat{h})\mathcal{P}(\widehat{u}),\mathcal{P}^{-1}(\widehat{h}), 0],&  \frac{\partial \widehat{v}}{\partial \widehat{U}}& = [-\mathcal{P}^{-1}(\widehat{h})\mathcal{P}(\widehat{v}),0,\mathcal{P}^{-1}(\widehat{h})].
    \end{align}
\end{lemma}
For the proof, see \cref{sec:entropy-tuple-proof}. This result is a crucial ingredient in the following two lemmas.
\begin{lemma}[Convexity of $E(\widehat{U})$]\label{lemma_convexity}
    If $\mathcal{P}(\widehat{h})$ is positive definite, then the function $E(\widehat{U})$ defined \eqref{entropy} is convex in $\widehat{U}$.
\end{lemma}
\begin{lemma}[Companion balance law]\label{lemma_companion_bl}
If $\widehat{U}$ is smooth, the entropy flux pair $(E,H,K)$ defined in \eqref{entropy} and \eqref{flux_functions} satisfies the two-dimensional companion balance law:
\begin{equation}\label{stochastic_companion_bl}
    E(\widehat{U})_t + H(\widehat{U})_x + K(\widehat{U})_y = 0.
\end{equation}
\end{lemma}
The proofs of \cref{lemma_convexity,lemma_companion_bl} are provided in \cref{sec:entropy-tuple-proof}. The proof of Theorem \ref{Thm_entropy_flux_pair} follows from \Cref{lemma_convexity,lemma_companion_bl}. I.e., we construct an entropy flux pair $(E,H,K)$, as defined in \eqref{entropy} and \eqref{flux_functions}, for the SG SWE \eqref{SGSWE}. This forms the foundation for developing energy conservative and energy stable semi-discrete finite volume schemes.

For our next goal of constructing energy conservative and energy stable schemes 
we define the following quantities,
\begin{equation}\label{entropy_variables}
    \begin{split}
        \widehat{V} \coloneqq & \Big(\frac{\partial E}{\partial \widehat{U}}\Big)^\top = \left( -\frac{1}{2}\widehat{u}^\top\mathcal{P}(\widehat{u}) -  \frac{1}{2}\widehat{v}^\top\mathcal{P}(\widehat{v}) + g(\widehat{h}+\widehat{B})^\top ,\widehat{u}^\top, \widehat{v}^\top  \right)^\top, \\
        \Psi \coloneqq & \widehat{V}^\top\widehat{F} - H = \frac{1}{2}g \widehat{u}^\top \mathcal{P}(\widehat{h})\widehat{h}, \\
        \Phi \coloneqq & \widehat{V}^\top\widehat{G} - K = \frac{1}{2}g \widehat{v}^\top \mathcal{P}(\widehat{h})\widehat{h},
    \end{split}
\end{equation}
where $\widehat{V}$ is called the entropy variable, and $\Psi$ and $ \Phi$ are called stochastic energy potentials.
    

\section{Well-balanced energy conservative and energy stable schemes for the 2D SG SWE}\label{Sec4}

In this section, we develop a well-balanced, second-order energy conservative (EC) scheme, as well as first-order and second-order energy stable (ES) schemes for the stochastic Galerkin (SG) formulation for the two-dimensional shallow water equations (SWE) \eqref{SGSWE}. These schemes are constructed using the entropy-flux pairs provided in \cref{Thm_entropy_flux_pair} and are designed to provide energy conservation and decay properties. They are stochastic extensions of the methods in \cite{fjordholm2011well} for deterministic balance laws, and are two-dimensional extensions of methods in \cite{dai2024energy} for the SWE SG system in one spatial dimension. We start by defining the well-balanced property of a numerical scheme, which ensures that stochastic ``lake-at-rest'' steady states are equilibrium states at the discrete level.

\begin{definition}[Well-Balanced SG SWE Property \cite{dai2022hyperbolicity}]\label{def:wb}
The solution $(h_P, q^x_P, q^y_P)$ of \eqref{SGSWE} is said to be well-balanced if it satisfies the stochastic ``lake-at-rest'' solution,
\begin{align}\label{well-balance}
    q^x_P(x,y,t,\xi) = q^y_P(x,y,t,\xi) &\equiv 0,& h_P(x,y,t,\xi) + \Pi_P[B](x,y,t,\xi)&\equiv C(\xi),
\end{align}
where $C(\xi)$ is a random scalar depending only on $\xi$.
  Such a well-balanced solution describes a still water surface with a flat but stochastic water surface. In terms of the system of PCE coefficients, \eqref{well-balance} is equivalent to the following vector equations
\begin{align}\label{well-balance-vector}
    \widehat{q^x} = \widehat{q^y}& \equiv {\bf 0}, & \widehat{h}+\widehat{B} &\equiv \widehat{C},& \forall (x,y,t) &\in \mathcal{D} \times [0,T],
\end{align}
with spatial domain $\mathcal{D}$ and terminal time $T$. Note that the vector equations \eqref{well-balance-vector} represent a steady state of the two-spatial-dimensional SG SWE \eqref{SGSWE}.
\end{definition}

\subsection{Energy conservative schemes}\label{Sec4_1}
The semi-discrete form for finite volume (FV) schemes for \eqref{SGSWE} on a uniform rectangular mesh reads,
\begin{equation}\label{FV-SGSWE}
    \frac{d}{dt}\bm{U}_{i,j} = - \frac{\mathcal{F}_{i+\frac{1}{2},j}-\mathcal{F}_{i-\frac{1}{2},j}}{\Delta x} - \frac{\mathcal{G}_{i,j+\frac{1}{2}}-\mathcal{G}_{i,j-\frac{1}{2}}}{\Delta y} + \bm{S}_{i,j}.
\end{equation}
The domain $\mathcal{D}$ is tessellated with rectangular cells $\mathcal{I}_{i,j} \coloneqq [x_{i-\frac{1}{2}},x_{i+\frac{1}{2}}] \times  [y_{j-\frac{1}{2}},y_{j+\frac{1}{2}}] $, with a uniform mesh size $\Delta x = x_{i+\frac{1}{2}}-x_{i-\frac{1}{2}}$ and $\Delta y = y_{j+\frac{1}{2}}-y_{j-\frac{1}{2}}$.  Hence, $\Delta x \Delta y$ is the size (area) of cell $\mathcal{I}_{i,j}$. We consider $M$ cells in each direction so that the rectangular domain so that $i, j \in [M]$. The quantity $\bm{U}_{i,j}(t)$ denotes the numerical approximation to the cell average of the vector $\widehat{U}$ over cell $\mathcal{I}_{i,j}$ at time $t$, i.e., $\bm{U}_{i,j}(t) \approx \frac{1}{\Delta x \Delta y} \int_{\mathcal{I}_{i,j}} \widehat{U}(x,y,t)dxdy$. The numerical fluxes $\mathcal{F}_{i \pm \frac{1}{2},j}$ and $\mathcal{G}_{i,j \pm \frac{1}{2}}$ depend on the neighboring states, e.g., $\bm{U}_{i,j}$ and $\bm{U}_{i+1,j}$ for $\mathcal{F}_{i+\frac{1}{2},j}$. Additionally, $\bm{S}_{i,j} \approx \frac{1}{\Delta x \Delta y}\int_{\mathcal{I}_{i,j}} \widehat{S}(\widehat{U},\widehat{B})$ represents the discretization of the source term, which should be designed to ensure a well-balanced numerical scheme in the sense of \cref{def:wb}. Other boldface notations, such as $(\bm{h}_{i,j},\bm{q^x}_{i,j},\bm{q^y}_{i,j},\bm{B}_{i,j})$, are defined in a similar manner. The discrete velocities $\bm{u}_{i,j}$ and $\bm{v}_{i,j}$ are defined as stochastic variants of \eqref{velocity_coeff}, as follows:
\begin{align}\label{velocity_discrete}
    \bm{u}_{i,j} &\coloneqq \mathcal{P}^{-1}(\bm{h}_{i,j})\bm{q^x}_{i,j}, &  \bm{v}_{i,j} &\coloneqq \mathcal{P}^{-1}(\bm{h}_{i,j})\bm{q^y}_{i,j}.
\end{align}
The discrete entropic quantities, stochastic variants of \eqref{entropy_variables}, are defined in terms of the discrete conservative variable $\bm{U}_{i,j}$ and the discrete velocities $\bm{u}_{i,j},\bm{v}_{i,j}$ as follows:
\begin{equation}\label{entropy_quantities_discrete}
    \begin{split}
        \bm{E}_{i,j} \coloneqq & \frac{1}{2}\Big( (\bm{q^x}_{i,j})^\top \bm{u}_{i,j}+ (\bm{q^y}_{i,j})^\top \bm{v}_{i,j} \Big) + \frac{1}{2}g\| \bm{h}_{i,j}\|^2 + g\bm h_{i,j}^\top \bm B_{i,j},\\
        \bm V_{i,j} \coloneqq & \left(\frac{\partial \bm E_{i,j}}{\partial \bm U_{i,j}}\right)^\top = \left( -\frac{1}{2}\bm u_{i,j}^\top \mathcal{P}(\bm u_{i,j}) -\frac{1}{2}\bm v_{i,j}^\top \mathcal{P}(\bm v_{i,j}) + g(\bm h_{i,j} + \bm B_{i,j})^\top, \quad \bm u_{i,j}^\top,\quad  \bm v_{i,j}^\top  \right)^\top,\\
        \bm \Psi_{i,j} \coloneqq & \frac{1}{2}g \bm u_{i,j}^\top \mathcal{P}(\bm h_{i,j})\bm h_{i,j}, \quad  \bm \Phi_{i,j} \coloneqq  \frac{1}{2}g \bm v_{i,j}^\top \mathcal{P}(\bm h_{i,j})\bm h_{i,j}.
    \end{split}
\end{equation}
Our numerical schemes involve the average and jump quantities at cell interfaces \cite{dai2024energy,fjordholm2011well}:
\begin{equation}
    \begin{aligned}\label{average_jump}
    \overline{\bm a}_{i+\frac{1}{2},j} &\coloneqq \frac{\bm a_{i,j} + \bm a_{i+1,j}}{2}, &\overline{\bm a}_{i,j+\frac{1}{2}} &\coloneqq \frac{\bm a_{i,j} + \bm a_{i,j+1}}{2}, \\
    \llbracket \bm a\rrbracket_{i+\frac{1}{2},j} &\coloneqq \bm a_{i+1,j} - \bm a_{i,j}, &\llbracket \bm a\rrbracket_{i,j+\frac{1}{2}} &\coloneqq \bm a_{i,j+1} - \bm a_{i,j},
\end{aligned}
\end{equation}
where $\bm a_{i,j}$ denotes any cell averaged quantity over $\mathcal{I}_{i,j}$, e.g., $\bm{U}_{i,j}$. The expressions \eqref{average_jump} are equivalent to
\begin{equation}\label{identity_ave_jump}
    \begin{split}
       \bm  a_{i,j} &= \overline{\bm a}_{i+\frac{1}{2},j} - \frac{\llbracket \bm a\rrbracket_{i+\frac{1}{2},j}}{2} = \overline{\bm a}_{i-\frac{1}{2},j} + \frac{\llbracket \bm a\rrbracket_{i-\frac{1}{2},j}}{2}  \\
       &= \overline{\bm a}_{i,j+\frac{1}{2}} - \frac{\llbracket \bm a\rrbracket_{i,j+\frac{1}{2}}}{2} = \overline{\bm a}_{i,j-\frac{1}{2}} + \frac{\llbracket \bm a\rrbracket_{i,j-\frac{1}{2}}}{2}.
    \end{split}
\end{equation}
Then we introduce additional equalities for the interfacial averages and jumps associated with the linear operator $\mathcal{P}$ in the following lemma. These are straightforward generalizations of \cite[Lemma 4.1]{dai2024energy}, so we omit the proof.
\begin{lemma}
    Let $\bm a_{i,j}, \bm b_{i,j}$ be any spatially discrete quantities, then
    \begin{subequations}
    \begin{align}\label{identity0}
          \mathcal{P}(\overline{\bm a}_{i+\frac{1}{2},j})\llbracket \bm{a}\rrbracket_{i+\frac{1}{2},j} &= \frac{1}{2}\llbracket \mathcal{P}(\bm a) \bm a\rrbracket_{i+\frac{1}{2},j}, & 
            \llbracket \bm a \rrbracket_{i+\frac{1}{2},j}^\top\overline{\bm b}_{i+\frac{1}{2},j} + \llbracket \bm b \rrbracket_{i+\frac{1}{2},j}^\top\overline{\bm a}_{i+\frac{1}{2},j} &= \llbracket \bm a^\top \bm b \rrbracket_{i+\frac{1}{2},j}, \\
    \mathcal{P}(\overline{\bm a}_{i,j+\frac{1}{2}})\llbracket \bm {a}\rrbracket_{i,j+\frac{1}{2}} &= \frac{1}{2}\llbracket \mathcal{P}(\bm a) \bm a\rrbracket_{i,j+\frac{1}{2}}, &
            \llbracket \bm a \rrbracket_{i,j+\frac{1}{2}}^\top\overline{\bm b}_{i,j+\frac{1}{2}} + \llbracket \bm b \rrbracket_{i,j+\frac{1}{2}}^\top\overline{\bm a}_{i,j+\frac{1}{2}} &= \llbracket \bm a^\top \bm b \rrbracket_{i,j+\frac{1}{2}}.
    \end{align}
    \end{subequations}
\end{lemma}

To provide the specific definitions for energy conservative and energy stable schemes for systems of balance laws in two spatial dimensions, we recall that the semi-discrete FV form \eqref{FV-SGSWE} is called a \textit{conservative scheme} when the source term vanishes i.e., $\bm S_{i,j} = 0$. In this case, and by summing \eqref{FV-SGSWE} over the cells, we obtain,
\begin{equation}\label{total_change_U}
    \frac{d}{dt} \sum_{i,j \in [M]} \bm U_{i,j}(t) = \sum_{j=1}^M \frac{\mathcal{F}_{\frac{1}{2},j} - \mathcal{F}_{M+\frac{1}{2},j}}{\Delta x} + \sum_{i=1}^M \frac{\mathcal{G}_{i,\frac{1}{2}} - \mathcal{G}_{i,M+\frac{1}{2}}}{\Delta y}.
\end{equation}
This implies that, if periodic boundary conditions are imposed so that $\mathcal{F}_{\frac{1}{2},j} = \mathcal{F}_{M+\frac{1}{2},j}$ and similarly for $\mathcal{G}$, then $\widehat{U}$ remains constant over time. Since the entropy-flux pair variables $(E,H,K)$ are explicit functions of the state variable $\widehat{U}$ and the inputs of the source term, i.e., $\widehat{B}$, the semi-discrete form \eqref{FV-SGSWE} of the balance law \eqref{SGSWE} can be transformed into a semi-discrete form of the corresponding companion balance law \eqref{stochastic_companion_bl} for the entropy (energy) of the system. The notions of energy conservative and energy stable schemes can be defined through the evolution of the entropy/energy of the system.

\begin{definition}[Energy conservative and energy stable schemes]\label{def_ECES}
Suppose the system of balance laws \eqref{SGSWE} has an entropy flux pair $(E,H,K)$, where $E(\widehat{U})$ represents the system's energy. Then the semi-discrete finite volume (FV) scheme \eqref{FV-SGSWE} is an \textbf{Energy Conservative (EC)} scheme if it can be rewritten in the following semi-discrete form for the evolution of the numerical cell averages $\bm E_{i,j}$ of $E$:
\begin{equation}\label{EC_def}
    \frac{d}{dt}\bm E_{i,j}(t) = -\frac{1}{\Delta x}(\mathcal{H}_{i+\frac{1}{2},j} -\mathcal{H}_{i-\frac{1}{2},j} ) - \frac{1}{\Delta y}(\mathcal{K}_{i,j+\frac{1}{2}} -\mathcal{K}_{i,j-\frac{1}{2}} ), 
\end{equation}
where $\mathcal{H}_{i+\frac{1}{2},j}$ represents the numerical entropy flux at the interface $(x,y) = (x_{i+\frac{1}{2}},y_j)$, and $\mathcal{K}_{i,j+\frac{1}{2}}$ represents another numerical entropy flux at the interface location $(x,y) = (x_i,y_{j+\frac{1}{2}})$.

Alternatively, the scheme \eqref{FV-SGSWE} is called an \textbf{Energy Stable (ES)} scheme under the weaker condition,
\begin{equation}\label{ES_def}
    \frac{d}{dt}\bm E_{i,j}(t) \le -\frac{1}{\Delta x}(\mathcal{H}_{i+\frac{1}{2},j} -\mathcal{H}_{i-\frac{1}{2},j} ) - \frac{1}{\Delta y}(\mathcal{K}_{i,j+\frac{1}{2}} -\mathcal{K}_{i,j-\frac{1}{2}} ).
\end{equation}
\end{definition}

By summing \eqref{EC_def} over all cells of the computational domain, a form similar to \eqref{total_change_U} can be obtained, but for energy instead of the state variable $\bm U$, showing that the cumulative energy remains constant over time under periodic boundary conditions. For non-periodic boundary conditions, energy can increase due to the boundary terms; we discuss a notion of \textit{augmented} energy in \cref{bdry_effect} that attempts to separate this potential intrinsic energy increase effect from any energy increase due to numerical discretizations.

\subsection{An energy conservative scheme for the 2D SG SWE (EC)}\label{Sec4_2}
For the scheme \eqref{FV-SGSWE}, we make the following choices for the fluxes and balance terms:
\begin{equation}\label{EC_flux}
    \begin{split}
      \mathcal{F}_{i+\frac{1}{2},j} = \mathcal{F}^{EC}_{i+\frac{1}{2},j} &= \begin{pmatrix}
            \mathcal{P}(\overline{\bm h}_{i+\frac{1}{2},j})\overline{\bm u}_{i+\frac{1}{2},j}\\
            \frac{1}{2}g(\overline{\mathcal{P}(\bm h)\bm h})_{i+\frac{1}{2},j} + \mathcal{P}(\overline{\bm u}_{i+\frac{1}{2},j})\mathcal{P}(\overline{\bm h}_{i+\frac{1}{2},j})\overline{\bm u}_{i+\frac{1}{2},j}\\
            \mathcal{P}(\overline{\bm v}_{i+\frac{1}{2},j})\mathcal{P}(\overline{\bm h}_{i+\frac{1}{2},j})\overline{\bm u}_{i+\frac{1}{2},j}
        \end{pmatrix},\\
        \mathcal{G}_{i,j+\frac{1}{2}}=  \mathcal{G}^{EC}_{i,j+\frac{1}{2}} &=\begin{pmatrix}
         \mathcal{P}(\overline{\bm h}_{i,j+\frac{1}{2}})\overline{\bm v}_{i,j+\frac{1}{2}}\\
         \mathcal{P}(\overline{\bm u}_{i,j+\frac{1}{2}})\mathcal{P}(\overline{\bm h}_{i,j+\frac{1}{2}})\overline{\bm v}_{i,j+\frac{1}{2}},\\
         \frac{1}{2}g(\overline{\mathcal{P}(\bm h)\bm h})_{i,j+\frac{1}{2}} + \mathcal{P}(\overline{\bm v }_{i,j+\frac{1}{2}})\mathcal{P}(\overline{\bm h}_{i,j+\frac{1}{2}})\overline{\bm v}_{i,j+\frac{1}{2}}
    \end{pmatrix},\\
   \bm{S}_{i,j} &= \begin{pmatrix}
       0 \\
       -\frac{g}{2\Delta x} (\mathcal{P}(\overline{\bm h}_{i+\frac{1}{2},j}) \llbracket \bm B\rrbracket_{i+\frac{1}{2},j} +  \mathcal{P}(\overline{\bm h}_{i-\frac{1}{2},j}) \llbracket \bm B\rrbracket_{i-\frac{1}{2},j})\\
       -\frac{g}{2\Delta y} (\mathcal{P}(\overline{\bm h}_{i,j+\frac{1}{2}}) \llbracket \bm B\rrbracket_{i,j+\frac{1}{2}} +  \mathcal{P}(\overline{\bm h}_{i,j-\frac{1}{2}}) \llbracket \bm B\rrbracket_{i,j-\frac{1}{2}})
   \end{pmatrix}.
    \end{split}
\end{equation}
These choices ensure that the scheme is energy conservative and well-balanced.
\begin{theorem}[Second-order EC well-balanced scheme]\label{theorem_EC}
    Suppose the bottom topography function $B$ is independent of time. The semi-discrete FV scheme \eqref{FV-SGSWE} for the two-spatial-dimensional SG SWE system \eqref{SGSWE} with fluxes and source term in \eqref{EC_flux} is a well-balanced EC scheme with local truncation error $\mathcal{O}(\Delta x^2 + \Delta y^2)$.
\end{theorem}

The proof consists of several parts, which require nontrivial extensions of the one-dimensional results in \cite{dai2024energy}. However, because the proof components are relatively technical, we place their presentation in the appendix.
First, we demonstrate the second-order truncation error through direct computation.
\begin{lemma}[Second-order truncation error]\label{lemma_LTE_EC}
    By selecting the numerical fluxes and the source term \eqref{EC_flux}, the semi-discrete FV scheme \eqref{FV-SGSWE} has a local truncation error of $\mathcal{O}(\Delta x^2 + \Delta y ^2)$.
\end{lemma}
See \cref{sec:LTE_EC} for the detailed proof. Next, we show that the discretization of the source term $\bm S_{i,j}$ in \eqref{EC_flux} satisfies the well-balanced property, which is the extension of results in \cite{dai2024energy} to two spatial dimensions.
\begin{lemma}[Well-balanced property]\label{lemma_wellbalanced}
    Suppose the source term is chosen as \eqref{EC_flux} and the bottom topography function $B$ is time-independent, then the semi-discrete FV scheme \eqref{FV-SGSWE} is well-balanced.
\end{lemma}
The proof of the above result is contained in \cref{sec:wellbalanced}. The final step in proving the main theorem of this section is to identify a sufficient condition that ensures the numerical fluxes produce an EC scheme. This is a two-spatial-dimensional extension of that in \cite{dai2024energy} and also a stochastic variant of that in the deterministic case in \cite{fjordholm2011well}.

\begin{lemma}[A sufficient condition for EC schemes]\label{lemma_sufficientcond_EC}
    Let $\bm S_{i,j}$ be selected as in \eqref{EC_flux}. Suppose the numerical fluxes $\mathcal{F}_{i+\frac{1}{2},j}, \mathcal{G}_{i,j+\frac{1}{2}}$ satisfy
    \begin{equation}\label{sufficient_condition_EC}
        \begin{split}
            \llbracket \bm V \rrbracket_{i+\frac{1}{2},j}^\top \mathcal{F}_{i+\frac{1}{2},j} = & \llbracket \bm \Psi \rrbracket_{i+\frac{1}{2},j} + g \llbracket \bm B\rrbracket_{i+\frac{1}{2},j}^\top\mathcal{P}(\overline{\bm h}_{i+\frac{1}{2},j})\overline{u}_{i+\frac{1}{2},j},\\
             \llbracket \bm V \rrbracket_{i,j+\frac{1}{2}}^\top \mathcal{G}_{i,j+\frac{1}{2}} = & \llbracket \bm \Phi \rrbracket_{i,j+\frac{1}{2}} + g \llbracket \bm B\rrbracket_{i,j+\frac{1}{2}}^\top\mathcal{P}(\overline{\bm h}_{i,j+\frac{1}{2}})\overline{v}_{i,j+\frac{1}{2}},
        \end{split}
    \end{equation}
where the discrete stochastic energy potentials $\bm \Psi$ and $\bm \Phi$ are defined in \eqref{entropy_quantities_discrete}.
Then the corresponding FV scheme \eqref{FV-SGSWE} is an EC scheme, i.e., it satisfies \eqref{EC_def}, with the numerical energy fluxes given by,
\begin{equation}\label{energy_flux}
    \begin{split}
        \mathcal{H}_{i+\frac{1}{2},j} & \coloneqq  \overline{\bm V}_{i+\frac{1}{2},j}^\top \mathcal{F}_{i+\frac{1}{2},j} - \overline{\bm \Psi}_{i+\frac{1}{2},j} - \frac{g}{4}\llbracket \bm B \rrbracket_{i+\frac{1}{2},j}^\top \mathcal{P}(\overline{\bm h}_{i+\frac{1}{2},j})\llbracket \bm u \rrbracket_{i+\frac{1}{2},j}\\
          \mathcal{K}_{i,j+\frac{1}{2}} & \coloneqq  \overline{\bm V}_{i,j+\frac{1}{2}}^\top \mathcal{G}_{i,j+\frac{1}{2}} - \overline{\bm \Phi}_{i,j+\frac{1}{2}} - \frac{g}{4}\llbracket \bm B \rrbracket_{i,j+\frac{1}{2}}^\top \mathcal{P}(\overline{\bm h}_{i,j+\frac{1}{2}})\llbracket \bm v \rrbracket_{i,j+\frac{1}{2}}.
    \end{split}
\end{equation}
\end{lemma}
For the proof, see \cref{sec:sufficien_condition_EC}. Given all essential lemmas above, we can complete the proof of \Cref{theorem_EC}.

\begin{proof}[Proof of Theorem \ref{theorem_EC}]
    It follows from Lemmas \ref{lemma_LTE_EC} and \ref{lemma_wellbalanced} that the scheme \eqref{FV-SGSWE} is well-balanced and second-order accurate for smooth solutions. To show that it is EC, it suffices to verify the condition in Lemma \ref{lemma_sufficientcond_EC}, which can be accomplished by direct computation:
    \begin{equation}
    \begin{split}
       &  \llbracket \bm V \rrbracket_{i+\frac{1}{2},j}^\top\mathcal{F}_{i+\frac{1}{2},j}^{EC}\\
\overset{\eqref{entropy_quantities_discrete},\eqref{EC_flux}}{=} & \left( g\left(\llbracket \bm h\rrbracket_{i+\frac{1}{2},j} + \llbracket \bm B \rrbracket_{i+\frac{1}{2},j}  \right)  - \frac{1}{2}\llbracket \mathcal{P}(\bm u)\bm u \rrbracket_{i+\frac{1}{2},j} - \frac{1}{2}\llbracket \mathcal{P}(\bm v)\bm v \rrbracket_{i+\frac{1}{2},j}    \right)^\top \mathcal{P}(\overline{\bm h}_{i+\frac{1}{2},j} )\overline{\bm u}_{i+\frac{1}{2},j} \\
& + \llbracket \bm u \rrbracket_{i+\frac{1}{2},j} ^\top\left(  \frac{1}{2}g \left(\overline{\mathcal{P}(\bm h)\bm h} \right)_{i+\frac{1}{2},j} + \mathcal{P}(\overline{\bm u}_{i+\frac{1}{2},j} )\mathcal{P}(\overline{\bm h}_{i+\frac{1}{2},j})\overline{\bm u}_{i+\frac{1}{2},j}   \right)\\
& + \llbracket \bm v \rrbracket_{i+\frac{1}{2},j} ^\top \mathcal{P}(\overline{\bm v}_{i+\frac{1}{2},j} )\mathcal{P}(\overline{\bm h}_{i+\frac{1}{2},j})\overline{\bm u}_{i+\frac{1}{2},j} \\
\overset{\eqref{identity0}}{=}& g\left( \llbracket \bm h \rrbracket_{i+\frac{1}{2},j} + \llbracket \bm B \rrbracket_{i+\frac{1}{2},j} \right)^\top \mathcal{P}(\overline{\bm h}_{i+\frac{1}{2},j})\overline{\bm u}_{i+\frac{1}{2},j} + \frac{1}{2}g\llbracket \bm u\rrbracket_{i+\frac{1}{2},j}^\top\left(\overline{\mathcal{P}(\bm h)\bm h} \right)_{i+\frac{1}{2},j}\\
\overset{\eqref{identity0}}{=}& \frac{1}{2}g \llbracket \mathcal{P}(\bm h)\bm h \rrbracket_{i+\frac{1}{2},j}^\top\overline{\bm u}_{i+\frac{1}{2},j} + g\llbracket \bm B \rrbracket_{i+\frac{1}{2},j}^\top  \mathcal{P}(\overline{\bm h}_{i+\frac{1}{2},j})\overline{\bm u}_{i+\frac{1}{2},j} + \frac{1}{2}g\llbracket \bm u\rrbracket_{i+\frac{1}{2},j}^\top\left(\overline{\mathcal{P}(\bm h)\bm h} \right)_{i+\frac{1}{2},j}\\
\overset{\eqref{identity0}}{=}&  \frac{1}{2}g \llbracket \bm u^\top\mathcal{P}(\bm h)\bm h \rrbracket_{i+\frac{1}{2},j} + g\llbracket \bm B \rrbracket_{i+\frac{1}{2},j}^\top  \mathcal{P}(\overline{\bm h}_{i+\frac{1}{2},j})\overline{\bm u}_{i+\frac{1}{2},j}\\
= & \llbracket \bm \Psi \rrbracket_{i+\frac{1}{2},j} + g \llbracket \bm B \rrbracket_{i+\frac{1}{2},j}^\top\mathcal{P}(\overline{\bm h}_{i+\frac{1}{2},j})\overline{\bm u}_{i+\frac{1}{2},j}.
    \end{split}
    \end{equation}
Also, by using the definition of $\mathcal{G}_{i,j+\frac{1}{2}}^{EC}$ \eqref{EC_flux}, $\bm \Phi$ \eqref{entropy_quantities_discrete}, and equalities \eqref{identity0}, through an analogous computation as above, one can obtain
\begin{equation}
     \llbracket \bm V \rrbracket_{i,j+\frac{1}{2}}^\top \mathcal{G}^{EC}_{i,j+\frac{1}{2}} = \llbracket \bm \Phi \rrbracket_{i,j+\frac{1}{2}} + g \llbracket \bm B\rrbracket_{i,j+\frac{1}{2}}^\top\mathcal{P}(\overline{\bm h}_{i,j+\frac{1}{2}})\overline{v}_{i,j+\frac{1}{2}},
\end{equation}
which verifies \eqref{sufficient_condition_EC}. Hence, it shows that the scheme \eqref{FV-SGSWE}, with the numerical fluxes and source term defined in \eqref{EC_flux}, is energy conservative.
    
\end{proof}

\subsection{A first-order energy stable scheme (ES1)}\label{Sec4_3}
The scheme defined in the previous part preserves the energy of the shallow water equation system, which can lead to non-physical oscillations, since the energy is supposed to dissipate in the presence of shocks. Some existing work can resolve this issue by introducing artificial numerical viscosity \cite{fjordholm2009energy,fjordholm2011well,fjordholm2012arbitrarily,tadmor1987numerical,tadmor2003entropy}. We adopt this approach by extending the energy-stable diffusion operators proposed in \cite{fjordholm2009energy,fjordholm2011well,fjordholm2012arbitrarily} to the stochastic case in two spatial dimensions.

First, we recall the traditional Roe-type diffusion (``$RD$'') applied to a conservation law, which involves the EC fluxes, defined as follows:
\begin{align}
    \mathcal{F}_{i+\frac{1}{2},j}^{RD} &\coloneqq \mathcal{F}_{i+\frac{1}{2},j}^{EC} - \frac{1}{2}\bm Q_{i+\frac{1}{2},j}^{Roe,F}\llbracket \bm U \rrbracket_{i+\frac{1}{2},j}, & \mathcal{G}_{i,j+\frac{1}{2}}^{RD} &\coloneqq \mathcal{G}_{i,j+\frac{1}{2}}^{EC} - \frac{1}{2}\bm Q_{i,j+\frac{1}{2}}^{Roe,G}\llbracket \bm U \rrbracket_{i,j+\frac{1}{2}} 
\end{align}
where $\bm Q^{Roe,F}$ and $\bm Q^{Roe,G}$ are positive semi-definite matrices defined by diagonalizing the Jacobians of the interfacial fluxes at the Roe-averaged state $\overline{\bm{U}}$:
\begin{align}
        \bm Q_{i+\frac{1}{2},j}^{Roe,F} & \coloneqq  \bm T^{Roe}_F  \big| \bm \Lambda_F^{Roe}\big|\left( \bm T^{Roe}_F\right)^{-1}, & \frac{\partial \widehat{F}}{\partial \widehat{ U}}(\overline{\bm U}_{i+\frac{1}{2},j}) &= \bm T^{Roe}_F \bm \Lambda_F^{Roe}\left( \bm T^{Roe}_F\right)^{-1},\\
        \bm Q_{i,j+\frac{1}{2}}^{Roe,G} & \coloneqq  \bm T^{Roe}_G  \big| \bm \Lambda_G^{Roe}\big|\left( \bm T^{Roe}_G\right)^{-1}, & \frac{\partial \widehat{G}}{\partial \widehat{ U}}(\overline{\bm U}_{i,j+\frac{1}{2}}) &= \bm T^{Roe}_G \bm \Lambda_G^{Roe}\left( \bm T^{Roe}_G\right)^{-1}.
\end{align}
The semi-discrete scheme \eqref{FV-SGSWE} with the numerical fluxes $\mathcal{F}_{i+\frac{1}{2},j}^{RD}$ and $\mathcal{G}_{i,j+\frac{1}{2}}^{RD} $ would behave like
\begin{equation}
    \begin{split}
       \frac{d}{dt}\bm{U}_{i,j}(t) = &  - \frac{\mathcal{F}^{EC}_{i+\frac{1}{2},j}-\mathcal{F}^{EC}_{i-\frac{1}{2},j}}{\Delta x} - \frac{\mathcal{G}^{EC}_{i,j+\frac{1}{2}}-\mathcal{G}^{EC}_{i,j-\frac{1}{2}}}{\Delta y} + \bm{S}_{i,j} \\
        & + \frac{1}{2\Delta x} \left( \bm Q_{i+\frac{1}{2},j}^{Roe,F}\llbracket \bm U \rrbracket_{i+\frac{1}{2},j} - \bm Q_{i-\frac{1}{2},j}^{Roe,F}\llbracket \bm U \rrbracket_{i-\frac{1}{2},j} \right) \\
         & + \frac{1}{2\Delta y} \left( \bm Q_{i,j+\frac{1}{2}}^{Roe,G}\llbracket \bm U \rrbracket_{i,j+\frac{1}{2}} - \bm Q_{i,j-\frac{1}{2}}^{Roe,G}\llbracket \bm U \rrbracket_{i,j-\frac{1}{2}} \right)  \\
         \approx & -\frac{1}{\Delta x}\left( \widehat{F}(\widehat{U} )\big|_{(x_{i+\frac{1}{2}} ,y_j)} - \widehat{F}(\widehat{U} )\big|_{(x_{i-\frac{1}{2}} ,y_j)} \right)-\frac{1}{\Delta y}\left( \widehat{G}(\widehat{U} )\big|_{(x_{i} ,y_{j+\frac{1}{2}})} - \widehat{G}(\widehat{U} )\big|_{(x_{i} ,y_{j-\frac{1}{2}})} \right) \\
          & + \Delta x \bm Q_F \widehat{U}_{xx}\big|_{(x_i,y_j)} + \Delta y\bm Q_G \widehat{U}_{yy}\big|_{(x_i,y_j)} + S(\widehat{U})\big|_{(x_i,y_j)},
    \end{split}
\end{equation}
where $\bm Q_F$ and $\bm Q_G$ are positive-definite matrices, and $\widehat{U}_{xx}$ and $ \widehat{U}_{yy}$ represent the second-order spatial derivatives of the state variables in the PDE, thus introducing diffusion into the EC scheme.

It is convenient to ensure energy stability by adding a numerical diffusion term operating on the entropic variables $\bm V$ rather than on the conservative variables $\bm U$. We introduce the following first-order energy stable numerical fluxes:
\begin{align}\label{ES1_flux}
    \mathcal{F}_{i+\frac{1}{2},j}^{ES1} &\coloneqq \mathcal{F}_{i+\frac{1}{2},j}^{EC} - \frac{1}{2}\bm Q_{i+\frac{1}{2},j}^{ES,F}\llbracket \bm V \rrbracket_{i+\frac{1}{2},j}, & \mathcal{G}_{i,j+\frac{1}{2}}^{ES1} &\coloneqq \mathcal{G}_{i,j+\frac{1}{2}}^{EC} - \frac{1}{2}\bm Q_{i,j+\frac{1}{2}}^{ES,G}\llbracket \bm V \rrbracket_{i,j+\frac{1}{2}},
\end{align}
where $\bm Q_{i+\frac{1}{2},j}^{ES,F}$ and $\bm Q_{i,j+\frac{1}{2}}^{ES,G}$ are positive definite matrices identified in a Roe-type manner from the adjacent states $(\bm U_{i,j}, \bm U_{i+1,j})$ at the cell interface $(x_{i+\frac{1}{2}},y_j)$, and $(\bm U_{i,j}, \bm U_{i,j+1})$ at the cell interface $(x_i,y_{j+\frac{1}{2}})$, respectively. The term $\bm V_{i,j}$, defined in \eqref{entropy_quantities_discrete}, serves as a second-order approximation of the cell average of the entropy variable $\widehat{V}.$ The matrices $\bm Q^{ES,F}$ and $\bm Q^{ES,G}$ will be defined in terms of the following Roe-type energy stable operator,
\begin{align}
  \mathcal{Q}_{i+\frac{1}{2},j}(\bm U_{i,j}, \bm U_{i+1,j}) &\coloneqq \bm T_F |\bm \Lambda_F| \bm T_F^\top \ge 0, & 
\mathcal{Q}_{i,j+\frac{1}{2}}(\bm U_{i,j}, \bm U_{i,j+1}) &\coloneqq \bm T_G |\bm \Lambda_G| \bm T_G^\top \ge 0, & 
\end{align}
where the matrices $\bm T_F,\bm \Lambda_F, \bm T_G, \bm\Lambda_G$ are defined from the eigendecomposition of the flux Jacobians $  \frac{\partial \widehat{F}}{\partial \widehat{U}}$ and 
$  \frac{\partial \widehat{G}}{\partial \widehat{U}}$ evaluated at the following Roe-type averaged state,
\begin{subequations}\label{def_T_F}
\begin{align}
    \frac{\partial \widehat{F}}{\partial \widehat{U}}\left( \widetilde{\bm U}_{i+\frac{1}{2},j} \right) &= \bm T_F \bm \Lambda_F \bm T_F^{-1}, & \widetilde{\bm U}_{i+\frac{1}{2},j} &\coloneqq \begin{pmatrix}
    \overline{\bm h}_{i+\frac{1}{2},j} \\ \mathcal{P}(\overline{\bm h}_{i+\frac{1}{2},j})\overline{\bm u}_{i+\frac{1}{2},j}\\
    \mathcal{P}(\overline{\bm h}_{i+\frac{1}{2},j})\overline{\bm v}_{i+\frac{1}{2},j}, 
  \end{pmatrix},
      \\
    \frac{\partial \widehat{G}}{\partial \widehat{U}}\left( \widetilde{\bm U}_{i,j+\frac{1}{2}} \right)& = \bm T_G \bm \Lambda_G \bm T_G^{-1}, & \widetilde{\bm U}_{i,j+\frac{1}{2}} &\coloneqq \begin{pmatrix}
    \overline{\bm h}_{i,j+\frac{1}{2}} \\ \mathcal{P}(\overline{\bm h}_{i,j+\frac{1}{2}})\overline{\bm u}_{i,j+\frac{1}{2}}\\
    \mathcal{P}(\overline{\bm h}_{i,j+\frac{1}{2}})\overline{\bm v}_{i,j+\frac{1}{2}}
\end{pmatrix}.
\end{align}
\end{subequations}
Then we set the diffusion matrices $\bm Q^{ES,F}$ and $\bm Q^{ES,G}$ appearing in the definition \eqref{ES1_flux} of the numerical fluxes $\mathcal{F}^{ES1}_{i+\frac{1}{2},j}$ and $\mathcal{G}^{ES1}_{i,j+\frac{1}{2}}$, respectively, as,  
\begin{subequations}\label{diff_matrix_ES1}
\begin{align}\label{diff_matrix_ES1_F}
  \bm Q_{i+\frac{1}{2},j}^{ES,F} &\coloneqq   \mathcal{Q}_{i+\frac{1}{2},j}(\bm U_{i,j}, \bm U_{i+1,j}) = \bm T_F |\bm \Lambda_F| \bm T_F^\top. \\
  \label{diff_matrix_ES1_G}
    \bm Q_{i,j+\frac{1}{2}}^{ES,G} &\coloneqq   \mathcal{Q}_{i,j+\frac{1}{2}}(\bm U_{i,j}, \bm U_{i,j+1}) = \bm T_G |\bm \Lambda_G| \bm T_G^\top,
\end{align}
\end{subequations}
The energy stable scheme in this subsection is constructed by using the numerical fluxes in \eqref{ES1_flux}, with the diffusion matrices given in \eqref{diff_matrix_ES1}. We codify the properties of this scheme as follows. 

\begin{theorem}[ES1 scheme]
    Consider the semi-discrete FV scheme \eqref{FV-SGSWE} with the source term given in \eqref{EC_flux} and diffusive numerical fluxes $\mathcal{F}^{ES1}_{i+\frac{1}{2},j}, \mathcal{G}^{ES1}_{i,j+\frac{1}{2}}$ defined in \eqref{ES1_flux}, and the diffusion matrices $\bm Q_{i+\frac{1}{2},j}^{ES,F},   \bm Q_{i,j+\frac{1}{2}}^{ES,G}$ as defined in \eqref{diff_matrix_ES1}. The resulting scheme is a first-order, well-balanced, and energy stable scheme.
\end{theorem}
\begin{proof}
  Note that we have already shown that the scheme \eqref{FV-SGSWE} with EC fluxes $\mathcal{F}_{i+\frac{1}{2},j}^{EC}$ and $ \mathcal{G}_{i,j+\frac{1}{2}}^{EC}$ is second-order accurate. Therefore, using the definition of $\bm V_{i,j}$, it is straightforward to show the following approximation
\begin{align}
    \llbracket \bm V \rrbracket_{i+\frac{1}{2},j} &\approx \Delta x \widehat{V}_x(x_{i+\frac{1}{2}},y_j), & \llbracket \bm V \rrbracket_{i,j+\frac{1}{2}} &\approx \Delta y \widehat{V}_y(x_i,y_{j+\frac{1}{2}}),
\end{align}
which implies that the artificial diffusion term in \eqref{ES1_flux} introduces a first-order local truncation error in the scheme.

  To demonstrate that this scheme is well-balanced, we consider the stochastic lake-at-rest initial data described as introduced in \eqref{well_balanced_IC} in \cref{Sec4_1}. This, combined with the definition of $\bm V_{i,j}$ in \eqref{entropy_quantities_discrete}, leads to the conclusion that $\llbracket \bm V \rrbracket_{i+\frac{1}{2},j} = \llbracket \bm V \rrbracket_{i,j+\frac{1}{2}} = \bm 0$. Since the EC fluxes are shown to be well-balanced in Lemma \ref{lemma_wellbalanced}, it follows that the ES1 scheme is also well-balanced.

Finally, we need to show the ES property \eqref{ES_def}. Define the ES1 energy fluxes
\begin{equation}\label{energy_flux_ES1}
\begin{split}
    \mathcal{H}_{i+\frac{1}{2},j}^{ES1} &\coloneqq  \mathcal{H}_{i+\frac{1}{2},j}^{EC} -  \frac{1}{2}\overline{\bm V}_{i+\frac{1}{2},j}^\top\bm Q_{i+\frac{1}{2},j}^{ES,F} \llbracket \bm V \rrbracket_{i+\frac{1}{2},j}\\
     \mathcal{K}_{i,j+\frac{1}{2}}^{ES1} &\coloneqq  \mathcal{K}_{i,j+\frac{1}{2}}^{EC} -  \frac{1}{2}\overline{\bm V}_{i,j+\frac{1}{2}}^\top\bm Q_{i,j+\frac{1}{2}}^{ES,G} \llbracket \bm V \rrbracket_{i,j+\frac{1}{2}},
\end{split}
\end{equation}
  where the EC energy fluxes $\mathcal{H}_{i+\frac{1}{2},j}^{EC}$ and $ \mathcal{K}_{i,j+\frac{1}{2}}^{EC}$ are defined in \eqref{energy_flux}. As in the proof of Lemma \ref{lemma_sufficientcond_EC} (\cref{sec:sufficien_condition_EC}), we multiply \eqref{FV-SGSWE} by $\bm V_{i,j}^\top$, and use a similar estimate with the ES1 energy fluxes defined above. This leads to the following estimate:
\begin{equation}
    \begin{split}
        \frac{d}{dt} \bm E_{i,j}(t) = & -\frac{1}{\Delta x}\left( \mathcal{H}_{i+\frac{1}{2},j}^{ES1} - \mathcal{H}_{i-\frac{1}{2},j}^{ES1} \right) - \frac{1}{\Delta y}\left( \mathcal{K}_{i,j+\frac{1}{2}}^{ES1} - \mathcal{K}_{i,j-\frac{1}{2}}^{ES1} \right) \\
        & - \frac{1}{4\Delta x} \left(   \llbracket \bm V \rrbracket_{i+\frac{1}{2},j}^\top \bm Q_{i+\frac{1}{2},j}^{ES,F} \llbracket \bm V \rrbracket_{i+\frac{1}{2},j}  + \llbracket \bm V \rrbracket_{i-\frac{1}{2},j}^\top \bm Q_{i-\frac{1}{2},j}^{ES,F} \llbracket \bm V \rrbracket_{i-\frac{1}{2},j} \right)\\
        & - \frac{1}{4\Delta y} \left(   \llbracket \bm V \rrbracket_{i,j+\frac{1}{2}}^\top \bm Q_{i,j+\frac{1}{2}}^{ES,G} \llbracket \bm V \rrbracket_{i,j+\frac{1}{2}}  + \llbracket \bm V \rrbracket_{i,j-\frac{1}{2}}^\top \bm Q_{i,j-\frac{1}{2}}^{ES,G} \llbracket \bm V \rrbracket_{i,j-\frac{1}{2}} \right).
    \end{split}
\end{equation}
Since the diffusion matrices $\bm Q_{i \pm \frac{1}{2},j}^{ES,F}, \bm Q_{i,j\pm \frac{1}{2}}^{ES,G}$ are positive semi-definite, this scheme satisfies the energy decay property \eqref{ES_def}, resulting in an ES scheme.
\end{proof}

\subsection{A second-order energy stable scheme (ES2)}\label{Sec4_4}
The ES1 scheme was derived by adding a first-order diffusion term to the second-order EC scheme. It is thus natural to develop a second-order ES scheme by incorporating a suitably constructed, second-order accurate diffusion term. This can be achieved through appropriate non-oscillatory second-order polynomial reconstructions of the entropy variable. We follow the idea of \cite{dai2024energy,fjordholm2012arbitrarily} to recover a non-oscillatory piecewise linear reconstruction. Let $\bm V_{i,j}^{E}$ and $\bm V_{i+1,j}^{W}$ denote the second-order reconstructions from the east and west, respectively, of the entropy variable $\bm V(x_{i+\frac{1}{2}},y_j)$. Similarly, let $\bm V_{i,j}^{N}$ and $\bm V_{i,j+1}^{S}$ represent the second-order reconstructions from the north and south of the entropy variable $\bm V(x_i,y_{j+\frac{1}{2}})$. I.e., if $\widetilde{\bm V}_{i,j}(x,y)$ is the polynomial reconstruction of the entropy variable $\bm V$ in the cell $\mathcal{I}_{i,j}$, then, 
\begin{align}
   \bm V_{i,j}^E &\coloneqq \lim_{ \substack {x \to x_{i+1/2}\\ y \to y_{j}}}\widetilde{\bm V}_{i,j}(x,y),& \bm V_{i,j}^W &\coloneqq \lim_{ \substack {x \to x_{i-1/2}\\ y \to y_{j}}}\widetilde{\bm V}_{i,j}(x,y),& \\
   \bm V_{i,j}^N &\coloneqq \lim_{ \substack {x \to x_{i}\\ y \to y_{j+1/2}}}\widetilde{\bm V}_{i,j}(x,y),& \bm V_{i,j}^S &\coloneqq \lim_{ \substack {x \to x_{i}\\ y \to y_{j-1/2}}}\widetilde{\bm V}_{i,j}(x,y).
\end{align}
Using these reconstructions, we can compute second-order accurate jumps of the entropy variable, 
\begin{align}\label{jump_V}
    \llangle \bm V \rrangle_{i+\frac{1}{2},j}& \coloneqq  \bm V_{i+1,j}^{W} - \bm V_{i,j}^{E}, & \llangle \bm V \rrangle_{i,j+\frac{1}{2}} &\coloneqq \bm V_{i,j+1}^{S} - \bm V_{i,j}^{N}.
\end{align}
Using the second-order jumps in the energy stable fluxes, we obtain the ES2 fluxes,
\begin{align}\label{ES2_flux}
    \mathcal{F}_{i+\frac{1}{2},j}^{ES2} &\coloneqq \mathcal{F}_{i+\frac{1}{2},j}^{EC} - \frac{1}{2}\bm Q_{i+\frac{1}{2},j}^{ES,F}\llangle \bm V \rrangle_{i+\frac{1}{2},j}, & \mathcal{G}_{i,j+\frac{1}{2}}^{ES2} &\coloneqq \mathcal{G}_{i,j+\frac{1}{2}}^{EC} - \frac{1}{2}\bm Q_{i,j+\frac{1}{2}}^{ES,G}\llangle \bm V \rrangle_{i,j+\frac{1}{2}},
\end{align}
where the positive definite matrices $\bm Q^{ES,F}$ and $\bm Q^{ES,G}$ are the same as those in the ES1 scheme \eqref{ES1_flux}. The remaining part of this section describes how to compute $  \bm V_{i+1,j}^{W}, \bm V_{i,j}^{E} , \bm V_{i,j+1}^{S},  \bm V_{i,j}^{N}$ in a manner that ensures entropy stability.

The first step is to use a scaled version of the entropy variables as follows:
\begin{equation}\label{ES2_def_w}
    \bm w_{i,j}^E \coloneqq  (\bm T_F)_{i+\frac{1}{2},j}^\top \bm V_{i,j},\quad w_{i,j}^{W}  \coloneqq  (\bm T_F)_{i-\frac{1}{2},j}^\top \bm V_{i,j}, \quad \bm w_{i,j}^N \coloneqq  (\bm T_G)_{i,j+\frac{1}{2}}^\top \bm V_{i,j},\quad w_{i,j}^S \coloneqq  (\bm T_G)_{i,j-\frac{1}{2}}^\top \bm V_{i,j},
\end{equation}
where the matrices $\bm T_F, \bm T_G$ are defined in \eqref{def_T_F}. 
Next, we perform a second-order total variation-diminishing (TVD) reconstruction of the scaled variable $\bm w$ at the cell interfaces
\begin{equation}\label{ES2_w_tilde}
    \begin{aligned}
       \widetilde{\bm w}_{i,j}^E &\coloneqq \bm w_{i,j}^E + \frac{1}{2}\phi ( \bm \theta_{i,j}^E ) \circ \llangle \bm w \rrangle_{i+\frac{1}{2},j}, & \widetilde{\bm w}_{i,j}^W &\coloneqq \bm w_{i,j}^W - \frac{1}{2}\phi ( \bm \theta_{i,j}^W ) \circ \llangle \bm w \rrangle_{i-\frac{1}{2},j},\\
            \widetilde{\bm w}_{i,j}^N &\coloneqq \bm w_{i,j}^N + \frac{1}{2}\phi ( \bm \theta_{i,j}^N ) \circ \llangle \bm w \rrangle_{i,j+\frac{1}{2}}, & \widetilde{\bm w}_{i,j}^S &\coloneqq \bm w_{i,j}^S - \frac{1}{2}\phi ( \bm \theta_{i,j}^S ) \circ \llangle \bm w \rrangle_{i,j-\frac{1}{2}},
\end{aligned}
\end{equation}
where the jump $\llangle \cdot \rrangle$ is defined in \eqref{jump_V}, $\circ$ denotes the Hadamard (elementwise) product on vectors, and $\bm \theta_{i,j}$ represents the difference quotients, defined by
\begin{align}\label{def_theta}
        \bm \theta_{i,j}^E &\coloneqq \llangle \bm w\rrangle_{i-\frac{1}{2},j} \oslash \llangle \bm w \rrangle_{i+\frac{1}{2},j}, &  \bm \theta_{i,j}^W &\coloneqq \llangle \bm w\rrangle_{i+\frac{1}{2},j} \oslash \llangle \bm w \rrangle_{i-\frac{1}{2},j},\\
            \bm \theta_{i,j}^N &\coloneqq \llangle \bm w\rrangle_{i,j-\frac{1}{2}} \oslash \llangle \bm w \rrangle_{i,j+\frac{1}{2}}, &  \bm \theta_{i,j}^S &\coloneqq \llangle \bm w\rrangle_{i,j+\frac{1}{2}} \oslash \llangle \bm w \rrangle_{i,j-\frac{1}{2}},
\end{align}
where $\oslash$ is the Hadamard (elementwise) division between vectors. To preserve the TVD property, as discussed in \cite{dai2024energy,fjordholm2012arbitrarily}, we select the function $\phi$ as the minmod limiter, defined by
\begin{equation}
    \phi(\theta) = \left\{
\begin{aligned}
& 0,&  & \theta < 0, \\
  & \theta ,&  &0\le \theta\le 1,\\
  & 1,&  &\text{otherwise,}
\end{aligned}
\right.
\end{equation}
which operates elementwise on vector inputs. Finally, the desired reconstructions for $\bm V_{i,j}^E,\bm V_{i,j}^W,\bm V_{i,j}^N,\bm V_{i,j}^S$ can be computed by inverting the $\bm w\text{-to-}\bm V$ map, respectively,
\begin{equation}\label{ES2_def_jump_V}
    (\bm T_F)_{i+\frac{1}{2},j}^\top \bm V_{i,j}^E \coloneqq  \widetilde{\bm w}_{i,j}^E  ,\quad (\bm T_F)_{i-\frac{1}{2},j}^\top \bm V_{i,j}^W \coloneqq  \widetilde{\bm w}_{i,j}^W, \quad (\bm T_G)_{i,j+\frac{1}{2}}^\top \bm V_{i,j}^N \coloneqq  \widetilde{\bm w}_{i,j}^N  \quad (\bm T_G)_{i,j-\frac{1}{2}}^\top \bm V_{i,j}^S \coloneqq  \widetilde{\bm w}_{i,j}^S.
\end{equation}
The final ES2 scheme defined in \eqref{ES2_flux} satisfies the desired properties.

\begin{theorem}[ES2 scheme]\label{ES2_thm} The FV scheme \eqref{FV-SGSWE}, with the source term in \eqref{EC_flux} and the diffusive numerical fluxes $\mathcal{F}^{ES2}_{i+\frac{1}{2},j}, \mathcal{G}^{ES2}_{i,j+\frac{1}{2}}$ defined in \eqref{ES2_flux}, is a second-order, well-balanced, and energy stable scheme.
\end{theorem}
The second-order accuracy results from the fact that the jumps $  \llangle \bm V \rrangle_{i+\frac{1}{2},j}$ and $\llangle \bm V \rrangle_{i,j+\frac{1}{2}}$ for each $i,j$ are computed by using second-order accurate reconstructions. The well-balanced property can be shown in the same manner as in the ES1 scheme, following from the definition of $\bm V_{i,j}$ and the assumption that the stochastic lake-at-rest initial data implies that the jumps are zero. To show the energy stability property, we refer to the following lemma from \cite{fjordholm2012arbitrarily} adapted to the two-dimensional case as stated below.

\begin{lemma}[\cite{fjordholm2012arbitrarily}, Lemma 3.2]\label{lemma_ES2} For each $i,j$, if there exists positive diagonal matrices $\bm \Pi_{i+\frac{1}{2},j}\ge 0 $ and $ \bm \Pi_{i,j+\frac{1}{2}} \ge 0$, s.t., the second-order jumps satisfy,
\begin{equation}\label{ES2_proof}
\begin{split}
        & \llangle \bm V \rrangle_{i+\frac{1}{2},j} =\left( (\bm T_F)^\top_{i+\frac{1}{2},j}\right)^{-1} \bm \Pi_{i+\frac{1}{2},j}  (\bm T_F)_{i+\frac{1}{2},j}^\top \llbracket \bm V \rrbracket_{i+\frac{1}{2},j},\\
    & \llangle \bm V \rrangle_{i,j+\frac{1}{2}} =\left( (\bm T_G)^\top_{i,j+\frac{1}{2}}\right)^{-1} \bm \Pi_{i,j+\frac{1}{2}}  (\bm T_G)_{i,j+\frac{1}{2}}^\top \llbracket \bm V \rrbracket_{i,j+\frac{1}{2}},
\end{split}
\end{equation}
 then the scheme \eqref{FV-SGSWE} with fluxes $\mathcal{F}_{i+\frac{1}{2},j} = \mathcal{F}_{i+\frac{1}{2},j}^{ES2} $ and $ 
  \mathcal{G}_{i,j+\frac{1}{2}} = \mathcal{G}_{i,j+\frac{1}{2}}^{ES2}$ is an ES scheme.
\end{lemma}
The proof of this lemma, which we omit, is accomplished by multiplying the FV scheme \eqref{FV-SGSWE} by $\bm V_{i,j}^\top$ and following a similar approach as it in the proof of Lemma \ref{lemma_sufficientcond_EC}. We turn to complete the proof of \Cref{ES2_thm}.
\begin{proof}[Proof of Theorem \ref{ES2_thm}]
By \Cref{lemma_ES2}, it suffices to establish the equalities \eqref{ES2_proof} to demonstrate the ES property of the ES2 scheme. The definition of \eqref{ES2_w_tilde} implies
\begin{equation}
\begin{split}
    & \llangle \widetilde{\bm w} \rrangle_{i+\frac{1}{2},j}^l = \left( 1 - \frac{1}{2}\phi \bigl( (\bm \theta_{i+1,j}^{W})^l \bigr)   - \frac{1}{2} \bigl( (\bm \theta_{i,j}^{E})^l \bigr) \right) \llangle \bm w\rrangle_{i+\frac{1}{2},j}^l,\\
    & \llangle \widetilde{\bm w} \rrangle_{i,j+\frac{1}{2}}^l = \left( 1 - \frac{1}{2}\phi \bigl( (\bm \theta_{i,j+1}^{S})^l \bigr)   - \frac{1}{2} \bigl( (\bm \theta_{i,j}^{N})^l \bigr) \right) \llangle \bm w\rrangle_{i,j+\frac{1}{2}}^l,
\end{split}
\end{equation}
where $l$ denotes the index of the components of the corresponding vectors. The equality above, together with the definitions in \eqref{ES2_def_w}, \eqref{ES2_def_jump_V}, and the linearity of $\llangle \cdot \rrangle $ and $ \llbracket \cdot \rrbracket$, implies \eqref{ES2_proof} with the diagonal matrices $\bm \Pi_{i+\frac{1}{2},j}$ and $\bm \Pi_{i,j+\frac{1}{2}}$ defined by
\begin{align}
    (\bm \Pi_{i+\frac{1}{2},j})_{l,l} &=  1 - \frac{1}{2}\phi \bigl( (\bm \theta_{i+1,j}^{W})^l \bigr)   - \frac{1}{2} \bigl( (\bm \theta_{i,j}^{E})^l \bigr), &(\bm \Pi_{i,j+\frac{1}{2}})_{l,l} &=  1 - \frac{1}{2}\phi \bigl( (\bm \theta_{i,j+1}^{S})^l \bigr)   - \frac{1}{2} \bigl( (\bm \theta_{i,j}^{N})^l \bigr).
\end{align}

Note that the diagonal matrices $\bm \Pi_{i+\frac{1}{2},j} $ and $ \bm \Pi_{i,j+\frac{1}{2}} $ are positive semi-definite, since the minmod limiter satisfies $0 \le \phi(\theta) \le 1$. Consequently, the ES2 scheme is energy stable, completing the proof.
\end{proof}

\section{Algorithmic details and pseudocode}\label{Sec4_5}
We have presented energy conservative and energy stable schemes (EC, ES1, ES2) for the semi-discrete form \eqref{FV-SGSWE}, including the numerical construction and theoretical results. We now discuss details of the corresponding fully discrete algorithms. Complete algorithmic pseudocode is given in \cref{alg:cap}.

\subsection{Desingularization}
In the construction of numerical fluxes, the computation of $\bm u_{i,j}, \bm v_{i,j}$ requires $\big(\mathcal{P}(\bm h_{i,j})\big)^{-1}$, which is valid when $\mathcal{P}(\bm h_{i,j})$ is positive-definite. However, this matrix could be ill-conditioned in some situations. To mitigate error from ill-conditioned operations, we employ a desingularization procedure introduced in \cite{kurganov2007second}, whose stochastic variant was introduced in \cite{dai2021hyperbolicity,dai2022hyperbolicity,dai2024energy}.
Suppose $\mathcal{P}(\bm h_{i,j})$ has the eigenvalue decomposition,
\begin{equation}
    \mathcal{P}(\bm h_{i,j}) = \bm Q \bm \Pi \bm Q ^\top,
\end{equation}
where $\Pi = \text{diag}(\pi_1,\dots,\pi_K)$ is a diagonal positive matrix. Then the desingularization procedure approximates $\big(\mathcal{P}(\bm h_{i,j})\big)^{-1} \bm q_{i,j}$ by perturbing the eigenvalues $\pi_k$ to regularize the velocities:
\begin{align}
        \bm u_{i,j} &= \bm Q \widetilde{\bm \Pi}^{-1} \bm Q^\top \bm q_{i,j}^x, & \widetilde{\bm \Pi} &= \text{diag}(\widetilde{\pi}_1,\dots,\widetilde{\pi}_K), & \widetilde{\pi}_i &= \frac{\sqrt{\pi_i^4 + \max\{\pi_i^4,\epsilon^4\}}}{\sqrt{2}\pi_i},\\
         \bm v_{i,j} &= \bm Q \widetilde{\bm \Pi}^{-1} \bm Q^\top \bm q_{i,j}^y, &\widetilde{\bm \Pi} &= \text{diag}(\widetilde{\pi}_1,\dots,\widetilde{\pi}_K), & \widetilde{\pi}_i &= \frac{\sqrt{\pi_i^4 + \max\{\pi_i^4,\epsilon^4\}}}{\sqrt{2}\pi_i},
\end{align}
where $\epsilon$ is a small positive constant. To ensure the consistency of the scheme, the discharge variables must be recomputed when the desingularization procedure is activated,
\begin{align}
    \bm q_{i,j}^x &\xleftarrow{} \mathcal{P}(\bm h_{i,j})\bm u_{i,j}, & \bm q_{i,j}^y &\xleftarrow{} \mathcal{P}(\bm h_{i,j})\bm v_{i,j}.
\end{align}

\subsection{Hyperbolicity-preserving criterion}
The hyperbolicity and existence of the entropy flux pair require that $\mathcal{P}(\widehat{h})$ is positive-definite, i.e., $\mathcal{P}(\widehat{h})>0$. This corresponds to the condition $\mathcal{P}(\bm h_{i,j})>0$ holding for each cell $\mathcal{I}_{i,j}$. While this may be true at the current time, extra conditions must be imposed that this is true at the next time step. A sufficient condition, \cite[Theorem 3.2]{dai2021hyperbolicity}, is the following condition for every $n \in [N]$:
\begin{align}\label{positive_criterion}
    \widehat{h}_{i,j}(\xi_n)& \coloneqq \sum_{k=1}^K  (\bm h_{i,j})_k \phi_k(\xi_n) > 0, & \bm h_{i,j} &= \big( (\bm h_{i,j})_1,\dots,(\bm h_{i,j})_K \big)^\top,
\end{align}
where $\{ \xi_n\}_{n=1}^N$ is a nodal set in $\mathbb{R}^{d}$ for a positive-weight quadrature rule with sufficient accuracy relative to the $\xi$-polynomial space $P = \text{span}(\phi_1,\dots,\phi_K)$. We enforce \eqref{positive_criterion} by restricting the time step size in the time evolution procedure, as is done in \cite[Lemma 4.2]{dai2021hyperbolicity}. Suppose \eqref{positive_criterion} is satisfied at the current time level, and the forward Euler method is applied. To enforce \eqref{positive_criterion} in the next time level, we require
\begin{equation}\label{positive_check_dt}
  \Delta t < \lambda \coloneqq \min_{i,j} \min_{n \in [N]} \left| \frac{\widehat{h}_{i,j}(\xi_n)}{ \frac{\widehat{\mathcal{F}}^h_{i+\frac{1}{2},j}(\xi_n) - \widehat{\mathcal{F}}^h_{i-\frac{1}{2},j}(\xi_n)}{\Delta x} + \frac{\widehat{\mathcal{G}}^h_{i,j+\frac{1}{2}}(\xi_n) - \widehat{\mathcal{G}}^h_{i,j-\frac{1}{2}}(\xi_n)}{\Delta y} } \right|,
\end{equation}
where 
\begin{align}
      \widehat{\mathcal{F}}^h_{i+\frac{1}{2},j}(\xi)& \coloneqq \sum_{k=1}^K \big(\mathcal{F}^h_{i+\frac{1}{2},j}\big)_k \phi_k(\xi), & \mathcal{F}_{i+\frac{1}{2},j} &= \big( (\mathcal{F}^h_{i+\frac{1}{2},j})^\top, (\mathcal{F}^{q^x}_{i+\frac{1}{2},j})^\top, (\mathcal{F}^{q^y}_{i+\frac{1}{2},j})^\top \big)^\top \in \mathbb{R}^{3K} \\
    \widehat{\mathcal{G}}^h_{i,j+\frac{1}{2}}(\xi) &\coloneqq \sum_{k=1}^K \big(\mathcal{G}^h_{i,j+\frac{1}{2}}\big)_k \phi_k(\xi), & \mathcal{G}_{i,j+\frac{1}{2}} &= \big( (\mathcal{G}^h_{i,j+\frac{1}{2}})^\top, (\mathcal{G}^{q^x}_{i,j+\frac{1}{2}})^\top, (\mathcal{G}^{q^y}_{i,j+\frac{1}{2}})^\top \big)^\top \in \mathbb{R}^{3K}.
\end{align}
In addition to enforcing a global time step restriction across all cells for hyperbolicity, we must also ensure that 
$\Delta t$ satisfies the wave speed CFL condition to account for the hyperbolic nature of the system, cf. \cite[Equation (4.15)]{dai2021hyperbolicity}.

\subsection{Adaptive time step size}\label{adaptive_time_procedure}
The time step restriction for the hyperbolicity-preserving procedure \eqref{positive_check_dt} is derived by assuming a forward Euler time integration method. We can extend this to higher-order strong stability-preserving Runge-Kutta (SSP-RK) methods, which are convex combinations of forward Euler methods with multiple intermediate stages \cite{doi:10.1137/S003614450036757X}. We employ the adaptive time-stepping procedure in \cite{chertock2015well-adaptiveRK}, which allows Forward Euler time step restrictions to be applied to SSP-RK procedures. Computationally, the procedure computes updates of the time step restriction at each intermediate stage of the RK process.

\subsection{Numerical imposition of boundary conditions}
On a rectangular computational domain $[a,b]\times [c,d]$, let $Q_{i,j}$, $i, j \in [M]$, denote the (numerical) cell averages of any physical quantity on the mesh. Periodic boundary conditions are straightforward to implement, 
\begin{align*}
  Q_{i,0} &= Q_{i,M}, & Q_{i,1} &= Q_{i,M+1}, & Q_{0,j} &= Q_{M,j}, & Q_{1,j} &= Q_{M+1,j}, &\forall\; i,j \in [M]
\end{align*}
Periodic boundary conditions can be unphysical, and one way to mitigate finite-domain effects is to impose ``outflow'' boundary conditions that emulate physical continuation of the domain beyond the computational grid. 
A simple strategy for outflow boundary condition is to set ghost cell values through extrapolation from interior cells. The straightforward zeroth-order extrapolation uses a constant function for the extrapolation, i.e., sets as boundary conditions,
\begin{align}\label{eq:bc-outflow}
  Q_{i,0}&=Q_{i,1}, &Q_{i,M+1} &= Q_{i,M},& Q_{0,j}&=Q_{1,j},& Q_{M+1,j} &= Q_{M,j},& \forall\; i,j \in [M]
\end{align}
Higher-order extrapolation can improve accuracy, but zeroth-order extrapolation is often preferred for stability \cite{leveque2002finite}. 

\subsection{Measuring energy change through augmented energy}\label{bdry_effect}
\Cref{def_ECES} that labels schemes as energy conservative and energy stable relies on such properties holding in the \textit{interior} of the computational domain, but does not guarantee energy conservation/stability if boundary fluxes work to increase energy inside the domain. In such cases, energy will increase through no fault of the scheme but instead as a natural physical evolution of the model. Such situations can occur even for standard procedures that implement outflow boundary conditions. We discuss this in more detail below.

Under periodic boundary conditions, summing the equality/inequality in the definition of the EC or ES scheme yields,
\begin{align}
  \sum_{i,j \in [M]} \frac{d}{dt}\bm E_{i,j}(t) &= 0, &\text{ or}, & &\sum_{i,j \in [M]}\frac{d}{dt}\bm E_{i,j}(t) &\le 0,
\end{align}
respectively. 
With outflow boundary conditions, the summation of the semi-discrete formula differs from the periodic case. For example, in a rectangular computational domain with outflow boundary conditions on all four sides, summing the energy change across all cells to enforce energy conservation or stability corresponds to the respective assertions,
\begin{subequations}\label{Energychange_outflow}
  \begin{align}
      \sum_{i,j \in [M]} \frac{d}{dt}\bm E_{i,j}(t) &= \frac{1}{\Delta x} \sum_{j \in [M]} (\mathcal{H}_{\frac{1}{2},j} - \mathcal{H}_{M+\frac{1}{2},j} )+ \frac{1}{\Delta y} \sum_{i \in [M]}(\mathcal{K}_{i,\frac{1}{2}} - \mathcal{K}_{i,M+\frac{1}{2}}), & \textrm{(EC)} \\
    \sum_{i,j \in [M]} \frac{d}{dt}\bm E_{i,j}(t) &\le \frac{1}{\Delta x} \sum_{j \in [M]} (\mathcal{H}_{\frac{1}{2},j} - \mathcal{H}_{M+\frac{1}{2},j} )+ \frac{1}{\Delta y} \sum_{i \in [M]}(\mathcal{K}_{i,\frac{1}{2}} - \mathcal{K}_{i,M+\frac{1}{2}}), & \textrm{(ES)}
  \end{align}
\end{subequations}
From these expressions, we see that since the boundary fluxes are imposed by a relatively arbitrary problem setup (outflow boundary conditions, Dirichlet boundary conditions, etc.), then we cannot expect the above equality or inequality to result in non-increasing energy, regardless of the scheme developed. This is a concern that arises in practice: Consider the SWE SG system in one spatial dimension (formally, from the two-dimensional formulation one can set all $y$-derivatives to 0 and make all quantities $y$-independent), 
with the following piecewise linear bottom topography $B$, water velocity $u$, and water surface $w$:
\begin{align}\label{eq:augenergy-example}
  B(x) &= \left\{
\begin{aligned}
& 0,& & x<0.5, \\
  & 0.1x - 0.05 ,& &0.5 \le x \le 1.5,\\
  & 0.1,& &\text{otherwise,}
\end{aligned}
  \right. & 
u(x,0)& = 0.3, & w(x,0) &= 1,
\end{align}
over the computational domain is $[0,2]$. We impose the one-dimensional version of the outflow boundary conditions \eqref{eq:bc-outflow} on both sides of the domain, and we integrate up to terminal time $T=0.07$. 
\begin{figure}[htbp]
    \centering  \includegraphics[width=1\textwidth]{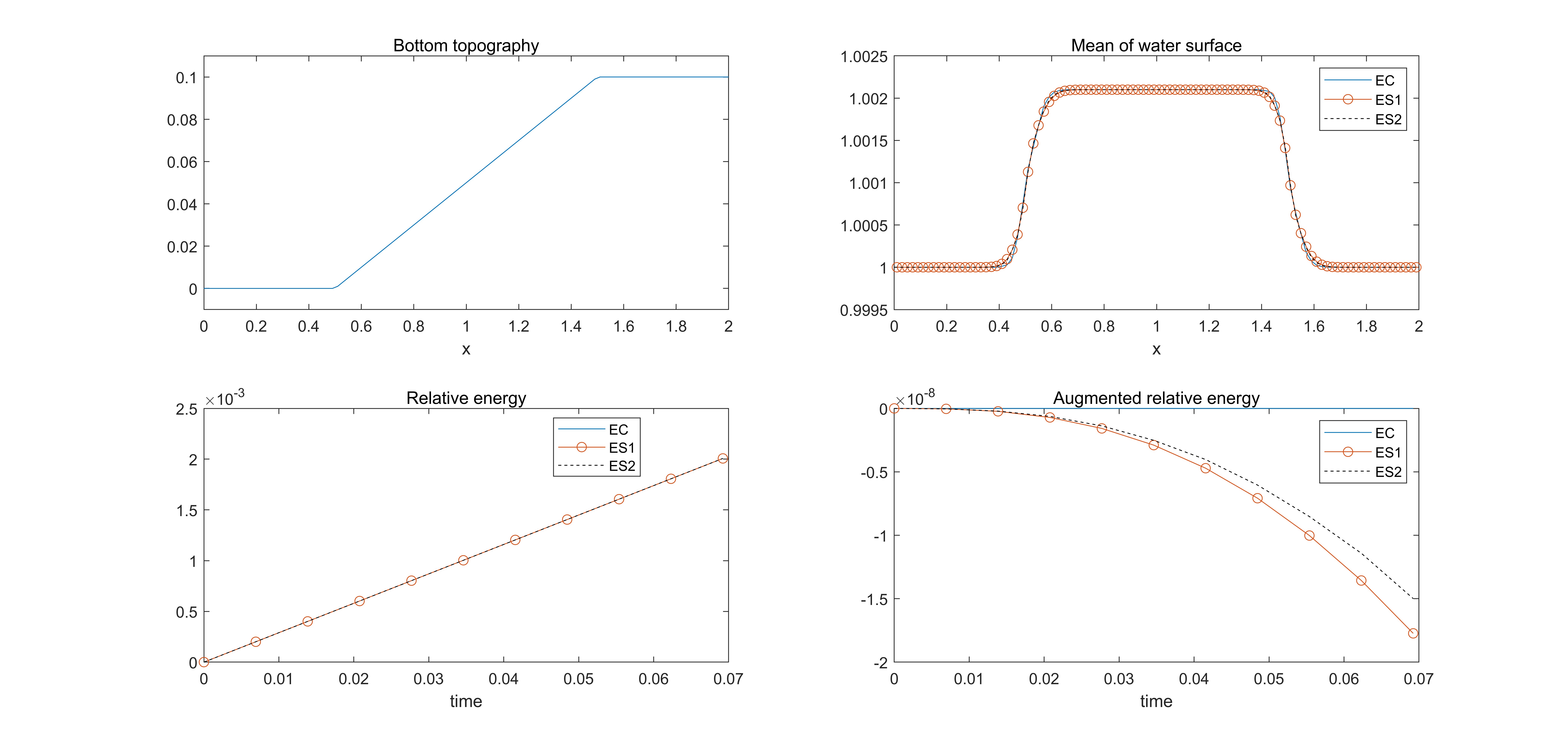}
    \centering
    \captionsetup{font={scriptsize},justification=raggedright}
    \caption{Illustration of energy versus augmented energy that incorporates energy fluxes at the boundary. Top left: Bottom topography $B$ corresponding to \eqref{eq:augenergy-example}. Top right: Mean of the water surface $w = h + B$ at terminal time $T = 0.07$. Bottom left: The standard relative energy $\frac{\bm{E}(t) - \bm{E}(0)}{\bm{E}(0)}$ increases in time for any accurate scheme due to influx of energy at the boundaries. Bottom right: Measuring the relative change in augmented energy $\frac{\widetilde{\bm{E}}(t) - \widetilde{\bm{E}}(0)}{\widetilde{\bm{E}}(0)}$ that offsets energy change due to boundary effects restores expected behavior for relative change of energy when using EC and ES schemes.}
    \label{fig_bdryeffect}
\end{figure}
The results are shown in Figure \ref{fig_bdryeffect}, and it is evident that the relative energy of the three schemes (EC, ES1, ES2) is increasing, see the bottom-left plot. I.e., this increase in relative energy is caused by the fact that this setup results in the total energy of the system increasing due to the outflow boundary conditions. 

This increase in energy for the true solution makes it difficult to evaluate the effectiveness of the EC and ES schemes. An alternative measure of energy is the \textit{augmented energy}, described below, which accounts for boundary effects on the total energy of the system. If $H_1$ and $H_M$ denote the entropy flux values at the one-dimensional boundary values, then we define the augmented energy $\widetilde{\bm{E}}$ over the computational domain to evolve as,
\begin{align}\label{derivative_augenergy_1D}
  \frac{d}{dt} \widetilde{\bm{E}}(t) &\coloneqq 
  \sum_{i=1}^M \frac{d}{dt}\bm {E}_{i}(t) -\frac{1}{\Delta x}(H_{1}-H_M), & 
  \widetilde{\bm{E}}(0) &= \sum_{i=1}^M \bm{E}_i(0).
\end{align}
I.e., the augmented energy evolves according to the sum over all cells of the energy derivative, along with the energy change associated with the entropy boundary fluxes. The equation \eqref{derivative_augenergy_1D} can be time integrated since the right-hand side can be evaluated. By evaluating the relative augmented energy for the setup \eqref{eq:augenergy-example}, shown in the bottom right plot of \Cref{fig_bdryeffect}, we observe that the EC scheme exhibits no (augmented) energy change, as expected, and the ES1 and ES2 schemes exhibit small energy decrease, with ES2 less dissipative than ES1, again as expected.

In this paper, for two-dimensional numerical examples, we generally implement outflow boundary conditions using zero-order extrapolation on the left and right boundaries and periodic boundary conditions on the upper and lower boundaries, resulting in the derivative of the augmented energy
Our two-dimensional numerical examples in the next section largely use periodic boundary conditions on the upper and lower boundaries and zeroth-order outflow boundary conditions on the left and right boundaries, corresponding to the augmented energy,
\begin{equation}\label{derivative_augenergy}
  \frac{d}{dt} \widetilde{\bm{E}}(t) 
  \coloneqq \sum_{i,j=1}^M \frac{d}{dt}\bm {E}_{i,j}(t) -\sum_{j=1}^M\frac{1}{\Delta x}(H_{1,j}-H_{M,j}).
\end{equation}


\begin{algorithm}[tbp]
\caption{Pseudocode of the fully discrete EC/ES schemes for the 2D SG SWE}\label{alg:cap}
\textbf{Input:} Scheme type: $scheme = $ EC, ES1, or ES2.\\
\textbf{Input:} Bottom topography $B$, initial data $U$ at $t=0$, terminal time $T$.\\
\textbf{Initialize:} Get PCE coefficients $\bm U_{i,j}$ with given polynomial index set $\Lambda$, set $t=0$.
\begin{algorithmic}
\While{$t<T$}
\State Compute $\bm B_{i,j}$ from $B$ for all $i,j$.
\State Compute $\bm u_{i,j}, \bm v_{i,j}$ for all $i,j$.
\State Compute $\mathcal{F}_{i+\frac{1}{2},j}^{EC}, \mathcal{G}_{i,j+\frac{1}{2}}^{EC}$, and $\bm S_{i,j}$ for all $i,j$, by \eqref{EC_flux}.
\If{$scheme$ is EC}
    \State  $\mathcal{F}_{i+\frac{1}{2},j} \gets \mathcal{F}_{i+\frac{1}{2},j}^{EC}, \mathcal{G}_{i,j+\frac{1}{2}} \gets \mathcal{G}_{i,j+\frac{1}{2}}^{EC}$, for all $i,j$.
\Else{ for all $i,j$}
\State Compute entropy variable $\bm V_{i,j}$ by \eqref{entropy_quantities_discrete}.
\State Compute $\bm Q_{i+\frac{1}{2},j}^{ES,F}, \bm Q_{i,j+\frac{1}{2}}^{ES,G}$ for all $i,j$ by \eqref{diff_matrix_ES1_F}, \eqref{diff_matrix_ES1_G}. 
\If{$scheme$ is ES1}
    \State Compute $\mathcal{F}_{i+\frac{1}{2},j} \gets \mathcal{F}_{i+\frac{1}{2},j}^{ES1}, \mathcal{G}_{i,j+\frac{1}{2}} \gets \mathcal{G}_{i,j+\frac{1}{2}}^{ES1}$ by \eqref{ES1_flux} using $\mathcal{F}_{i+\frac{1}{2},j}^{EC}, \mathcal{G}_{i,j+\frac{1}{2}}^{EC}, \bm V_{i,j}, $
    \State and $\bm Q_{i+\frac{1}{2},j}^{ES,F}, \bm Q_{i,j+\frac{1}{2}}^{ES,G}$.
    \ElsIf{ $scheme$ is ES2}
    \State Construct $\bm V_{i,j}^E, \bm V_{i,j}^W, \bm V_{i,j}^N, \bm V_{i,j}^S$ by \eqref{ES2_def_w}, \eqref{ES2_w_tilde}, and \eqref{ES2_def_jump_V}.
    \State Compute $\mathcal{F}_{i+\frac{1}{2},j} \gets \mathcal{F}_{i+\frac{1}{2},j}^{ES2}, \mathcal{G}_{i,j+\frac{1}{2}} \gets \mathcal{G}_{i,j+\frac{1}{2}}^{ES2}$ 
    \State by \eqref{ES2_flux} using $\mathcal{F}_{i+\frac{1}{2},j}^{EC}, \mathcal{G}_{i,j+\frac{1}{2}}^{EC}, \bm V_{i,j}^E, \bm V_{i,j}^W, \bm V_{i,j}^N, \bm V_{i,j}^S, $ and $\bm Q_{i+\frac{1}{2},j}^{ES,F}, \bm Q_{i,j+\frac{1}{2}}^{ES,G}$.
\EndIf
\EndIf
\State Initialize $\lambda$ and $\Delta t$ by \eqref{positive_check_dt} and CFL condition.
\State Determine $\Delta t$ using the adaptive time step size procedure in Section \ref{adaptive_time_procedure}.
\State Use a SSP RK method with $\Delta t$ determined adaptively, updating $\bm h_{i,j}, \bm q^x_{i,j}, \bm q^y_{i,j}$.
\State $t \gets t+\Delta t$
\EndWhile
\end{algorithmic}
\end{algorithm}

\section{Numerical experiments}\label{Sec5}
We will illustrate the performance of our EC and ES schemes through various examples, investigating convergence rates, energy conservation and decay properties, the well-balanced property, and the ability to handle multivariate random variables.

Let $\xi$ be a one-dimensional random variable from the Beta parametric family, 
\begin{align}
  \xi \sim \mathrm{Beta}(\beta+1, \alpha+1) \textrm{ on $[-1,1]$ }:\; \rho(\xi) &\propto (1-\xi)^{\alpha}(1+\xi)^{\beta}, & \xi &\in [-1,1],
\end{align}
Hence, 
the parameters $\alpha, \beta>-1$ can be chosen freely and control the mass concentration at $\xi=1$ and $\xi = -1$, respectively. For this density $\rho$, the associated orthonormal polynomial basis is the family of Jacobi polynomials with parameters $(\alpha,\beta)$. In particular, $\alpha=\beta = 0$ corresponds to the uniform distribution $\mathcal{U}(-1,1)$, with Legendre polynomial basis functions. We will also consider a two-dimensional random variable $\xi = (\xi^{(1)},\xi^{(2)})$ with two independent and identically distributed one-dimensional random variables $\xi^{(1)},\xi^{(2)}$, i.e., $\xi^{(1)}, \xi^{(2)} \stackrel{\mathrm{iid}}{\sim} \mathrm{Beta}(\beta+1, \alpha+1)$. In this case the density is $\rho(\xi) \coloneqq \rho(\xi^{(1)})\rho(\xi^{(2)})$, and the orthonormal basis is obtained by tensorizing one-dimensional orthonormal polynomials.

For all two-dimensional examples, we implement outflow boundary conditions using zero-order extrapolation on the left and right boundaries and periodic boundary conditions on the upper and lower boundaries. To compare the energy conservation and decay properties of our EC and ES schemes, we will measure both the relative energy and the augmented relative energy change, with the latter introduced in \cref{bdry_effect}:
\begin{align}\label{eq:rel-energies}
  \text{relative change in (original) energy} &= \frac{\bm{E}(t)-\bm{E}(0)}{\bm{E}(0)}, \\ 
  \text{relative change in augmented energy} &= \frac{\widetilde{\bm{E}}(t) - \widetilde{\bm{E}}(0)}{\widetilde{\bm{E}}(0)},
\end{align}
where $\bm{E}(t) = \sum_{i,j} \bm E_{i,j}(t)$.


Throughout this section on numerical experiments, we assume the gravitational constant $g=1$ and generally use $K=4$ terms in the PCE procedure, except in the accuracy test and in the example involving a two-dimensional random variable. For the visualization of results, we will plot the water surface $w \coloneqq h+B$ instead of the conservative variable $h$ corresponding to the water height. We generally apply the first- and second-order energy stable schemes (ES1, ES2) in the following numerical examples and further apply the second-order energy conservative scheme (EC) in the accuracy test.

\subsection{Accuracy test for the energy conservative and the energy stable schemes}\label{Sec5_1}

We begin by examining the order of accuracy of the proposed EC and ES schemes for the two-dimensional stochastic shallow water system. This test is a stochastic modification of the test originally developed in \cite{BEKP}. The initial conditions for the water surface and velocities are deterministic and defined as follows:
\begin{align}
    w(x,y,0,\xi) &= 1, & u(x,y,0,\xi) &= 0.3, & v(x,y,0,\xi) &= 0,
\end{align}
where $w$ represents the water surface, and $u$ and $v$ are the velocities in the $x$- and $y$- directions, respectively. Consider a stochastic elliptic-shaped hump bottom
\begin{equation}
    B(x,y,\xi) = 0.5 \exp(-25(x-1)^2-50(y-0.5)^2) + 0.1(\xi + 1),
\end{equation}
where 
$\xi \sim \mathcal{U}(-1,1)$. 
The computational domain is $[0,2] \times [0,1]$. 
A reference solution is computed on a $800 \times 800$ uniform rectangular grid for each scheme. For the time evolution solver, we utilize the third-order Strong Stability-Preserving (SSP) Runge-Kutta (RK3) method \cite{doi:10.1137/S003614450036757X}. 

We illustrate the order of accuracy of our schemes using the water height $h$ by computing the error between the reference solution and the numerical test solution with the $ L^1$ norm in physical space and the $ L^2$ norm in stochastic space.
\begin{equation}
\begin{split}
    \text{Error}(h_d) &= \| h_d(x,y,t,\xi)-h_{ref}(x,y,t,\xi)\|_{L^1(\mathcal{D};L_{\rho}^2(\mathbb{R}^d))},\\
    \|h(x,y,t,\xi )\|_{L^1(\mathcal{D};L_{\rho}^2(\mathbb{R}^d))} & \coloneqq \int_{\mathcal{D}}\|h(x,y,t,\xi)\|_{L_{\rho}^2}dxdy
\end{split}
\end{equation}

The contour plots of the reference solution for the mean of water surface and its standard deviation at $t=0.07$ are shown in Figure \ref{fig_ex1}.

\begin{figure}[H]
    \centering  \centerline{\includegraphics[width=1.2\textwidth]{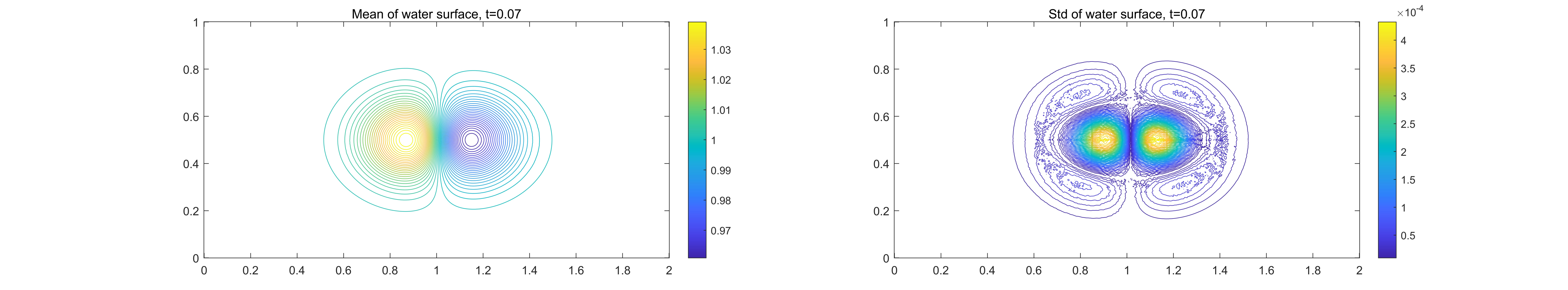}}
    \centering
    \captionsetup{font={scriptsize},justification=raggedright}
    \caption{Results for \cref{Sec5_1}: Contours for the water surface of the reference solution of ES2. Left: mean. Right: standard deviation. $T=0.07$.}
    \label{fig_ex1}
\end{figure}

We compute the orders of convergence for the numerical solutions with $K=2, T=0.07$, as summarized in Table \ref{table1}. The results indicate that the EC scheme has the smallest error and the highest order of accuracy, while the ES1 scheme has the largest error and the lowest order of accuracy. All schemes yield results consistent with theoretical expectations. We avoid using larger values of $K$ due to the occurrence of oscillations in stochastic space, which leads to instability in the higher modes of the PCE. This is a known general problem of the SG method when used for the hyperbolic conservation/balance laws problems, e.g. \cite{10.1093/imanum/drz004} and it will be part of our future research to address it.
\begin{table}[H]
    \centering
    \begin{tabular}{llllll}
    \noalign{\smallskip}
\hline
\noalign{\smallskip}
Scheme & Grid size  & Error & Order \\
\noalign{\smallskip}
\hline
\noalign{\smallskip}
\multirow{4}{4em}{ES1}& $100\times 100$
 & 2.1447e-04 &  & \\

&$200 \times 200$  &7.3671e-05 &  1.5417 &  \\

&$400 \times 400$  & 2.2557e-05&   1.7075 &  \\

\hline
\multirow{4}{4em}{ES2}& $100\times 100$
 &   1.5434e-04 &  & \\

&$200 \times 200$  &3.9852e-05 & 1.9534    &  \\

&$400 \times 400$  &1.0528e-05  &   1.9203 &  \\

\noalign{\smallskip}
\hline
\multirow{4}{4em}{EC}& $100\times 100$
 &  1.4880e-04   &  & \\

&$200 \times 200$  & 3.6890e-05&  2.0121   &  \\

&$400 \times 400$  & 8.8995e-06 & 2.0514 &  \\

\noalign{\smallskip}
\hline
\end{tabular}
  \caption{Accuracy test of \cref{Sec5_1}: Order of convergence for $K=2, T = 0.07$. Reference solution computed with grid size $800 \times 800$.}
\label{table1}
\end{table}

We investigated the EC scheme in this accuracy test because it produces smooth solutions for short time evolution up to $T = 0.07$. However, in the following numerical examples, we avoid using the EC scheme, as it introduces unphysical numerical oscillations into non-smooth solutions since energy should be dissipated across discontinuities, similar to the observations reported for the models in 1D physical space \cite{dai2024energy}.

\subsection{Gaussian-shape hump with stochastic bottom}\label{GShump_stochasticbottom}

We consider the example with a random variable applied to the position of the hump in the bottom topography, 
\begin{align}
    B(x,y,\xi) &= 0.8\exp(-5(x-0.9+0.1\xi)^2-50(y-0.5)^2), & \xi &\sim \mathcal{U}(-1,1).
\end{align}
The discontinuous initial water surface is given by
\begin{equation}
    w(x,y,0,\xi) = \left\{
\begin{aligned}
&1.01,&  &0.05<x<0.15, \\
  &1,&  &\text{otherwise},
\end{aligned}
\right.
\end{equation}
with zero initial velocities $ u(x,y,0) = v(x,y,0)=0.$

The computational domain is $[0,2] \times [0,1]$. The test is a modification of the tests from \cite{LEVEQUE1998346,BEKP,dai2022hyperbolicity}. We compute the numerical solutions of the ES schemes on a $200\times 200$ mesh grid at times $t=0.6, 0.9, 1.2, 1.5, 1.8.$ The results include contour plots representing the mean water surface, accompanied by disk glyphs whose radius is proportional to the uncertainties at the corresponding cells. According to the plots in Figure \ref{fig_ex2_solution}, the water propagates to the right initially. After interacting with the hump, the water splits and propagates in all directions, generating more wave structures. The uncertainties are concentrated near the peak of the water surface until the water reaches the hump, after which they spread out along the wave structure. The results in Figure \ref{fig_ex2_solution} and \ref{fig_ex2_energy} show that both the ES1 and ES2 schemes dissipate energy and produce similar wave structures. In addition, the ES2 scheme achieves higher resolution and with less energy dissipation than the ES1 scheme which delivers a more smeared/diffuse solution, as theoretically expected.

\begin{figure}[htbp]
    \centering  
    \includegraphics[width=1\textwidth]{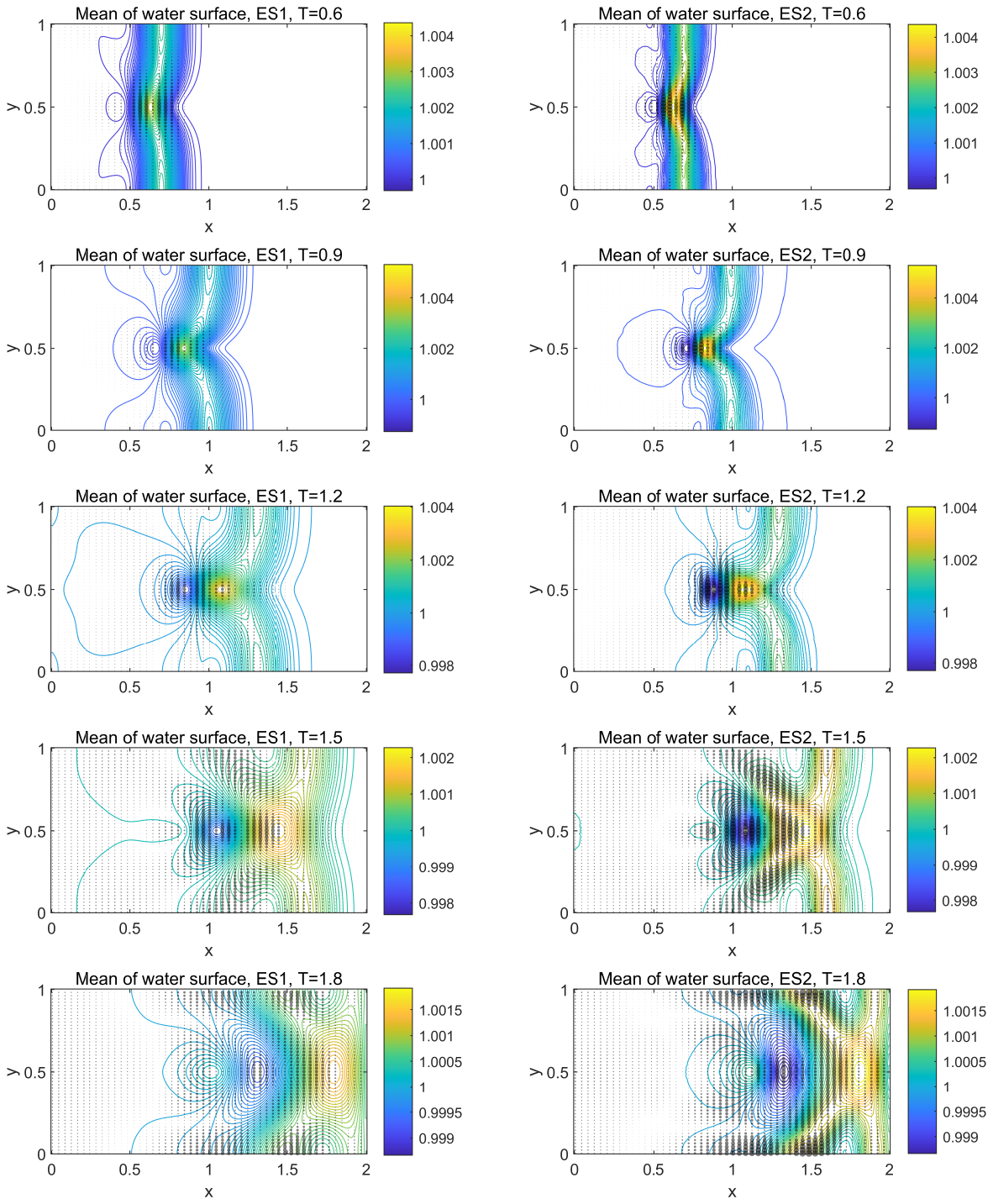}
    \centering
    \vspace{-1cm}
    \setlength{\abovecaptionskip}{-1cm}
    \captionsetup{font={scriptsize},justification=raggedright}
    \caption{Results for \cref{GShump_stochasticbottom}. Mean of water surface. Disk-glyph over mean contours, where the radii of the disks indicate the magnitude of the standard deviation. Left: ES1, the maximum standard deviation is 4.6524e-04, 7.3933e-04, 3.5168e-04, 1.3633e-04, 1.0856e-04, respectively. Right: ES2, the maximum standard deviation is 1.0398e-03, 2.1739e-03, 8.9851e-04, 3.4970e-04, 2.9730e-04, respectively. $200\times 200$, $t = 0.6, 0.9, 1.2, 1.5, 1.8$.}
    \label{fig_ex2_solution}
\end{figure}

\begin{figure}[htbp]
    \centering  \centerline{\includegraphics[width=1.1\textwidth]{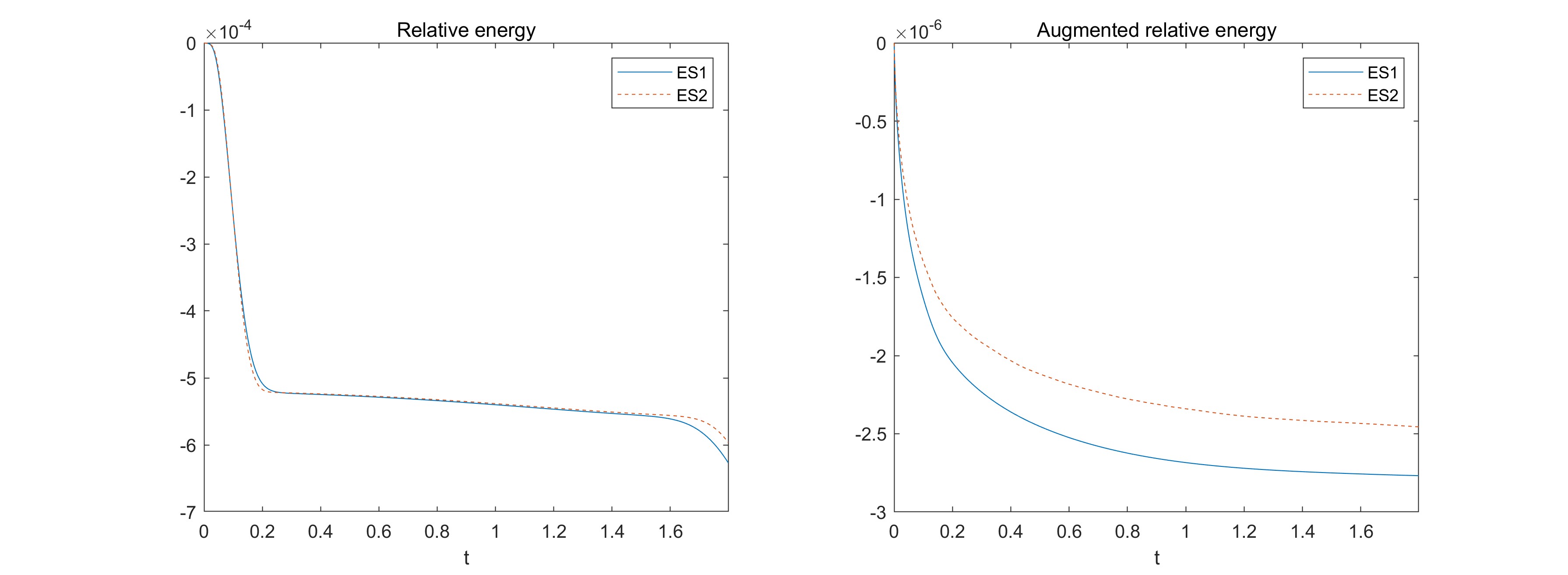}}
    \centering
    \captionsetup{font={footnotesize},justification=raggedright}
    \caption{Results for \cref{GShump_stochasticbottom}. Left: relative change in energy. Right: relative change in augmented energy.}
    \label{fig_ex2_energy}
\end{figure}

\subsection{Gaussian-shape hump with stochastic initial water surface}\label{Sec5_3}

To demonstrate the robustness of our schemes, we consider a similar example of the Gaussian-shape hump with a stochastic initial water surface, as discussed in Section \ref{GShump_stochasticbottom}. Consider the deterministic bottom topography
\begin{equation}
    B(x,y) = 0.8\exp(-5(x-0.9)^2-50(y-0.5)^2),
\end{equation}
and a stochastic initial water surface
\begin{equation}
    w(x,y,0,\xi) = \left\{
\begin{aligned}
&1+0.01(\xi+1),  &  &0.05<x<0.15, \\
  &1,  &  &\text{otherwise},
\end{aligned}
\right.
\end{equation}
where $\xi \sim \mathcal{U}(-1,1),$ and with zero initial velocities $ u(x,y,0) = v(x,y,0)=0.$ We implement the same settings as described in Section \ref{GShump_stochasticbottom} and present the results in Figures \ref{fig_ex3_solution} and \ref{fig_ex3_energy}. The plots display similar contours of the mean water surface as those in the previous example in Section \ref{GShump_stochasticbottom}. Both the ES1 and ES2 schemes demonstrate energy dissipation; however, the ES2 scheme provides better resolution and less energy dissipation compared to the ES1 scheme. Additionally, the uncertainties are more uniformly distributed across the water surface in proportion to the mean water surface, as compared to the distribution of uncertainties in Section \ref{GShump_stochasticbottom}.

\begin{figure}[htbp]
    \centering  
    \includegraphics[width=1\textwidth]{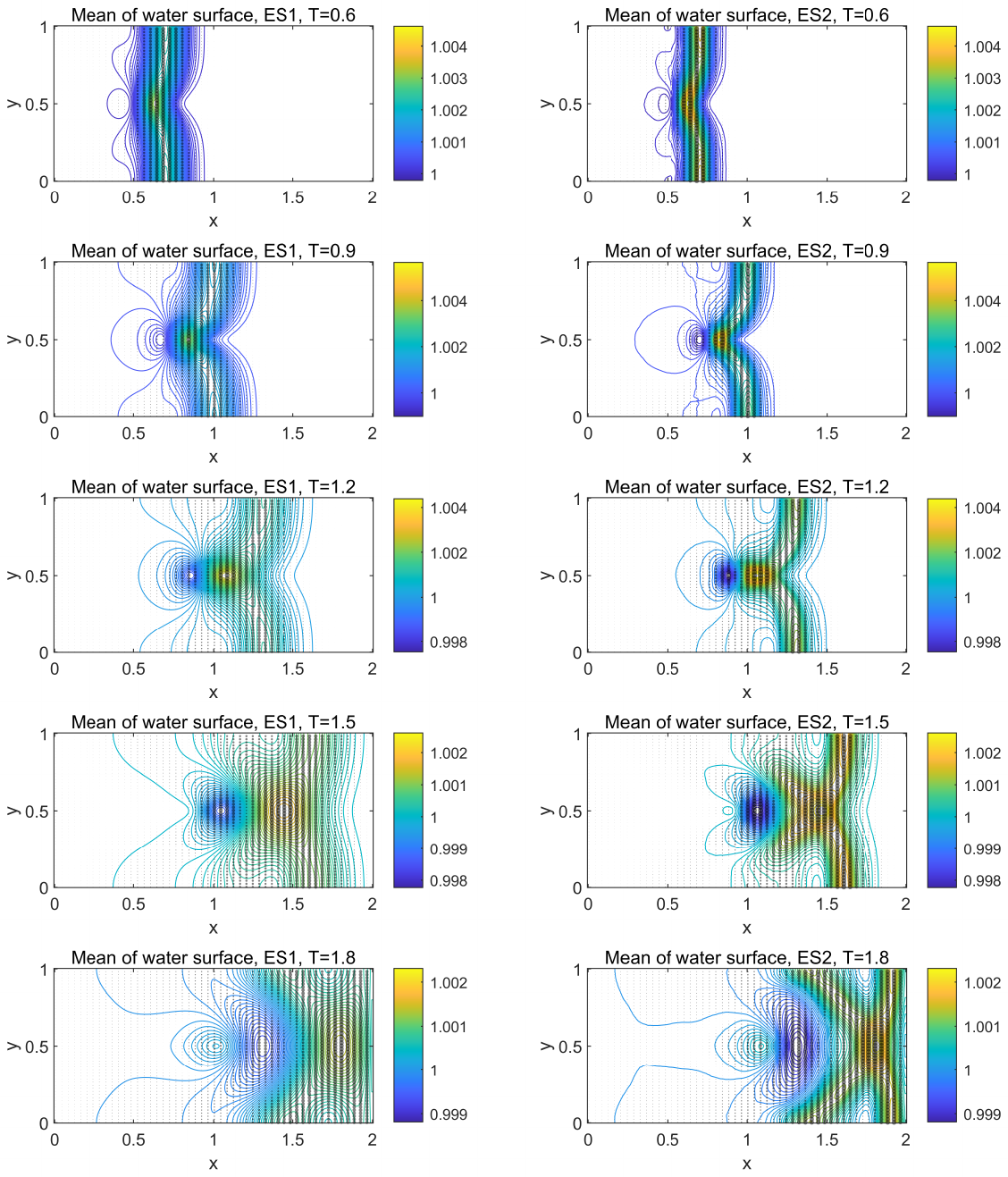}
    \centering
    \setlength{\abovecaptionskip}{-3cm}
    \captionsetup{font={scriptsize},justification=raggedright}
    \caption{Results for \cref{Sec5_3}. Mean of water surface. Disk-glyph over mean contours, where the radii of the disks indicate the magnitude of the standard deviation. Left: ES1, the maximum standard deviation is 1.8216e-03, 1.9375e-03, 1.5528e-03, 1.0162e-03, 8.2437e-04, respectively. Right: ES2, the maximum standard deviation is 2.6787e-03, 3.2736e-03, 2.5036e-03, 1.5027e-03, 1.3321e-03, respectively. $200\times 200$, $t = 0.6, 0.9, 1.2, 1.5, 1.8$.}
   \label{fig_ex3_solution}
\end{figure}

\begin{figure}[htbp]
    \centering  \centerline{\includegraphics[width=1.1\textwidth]{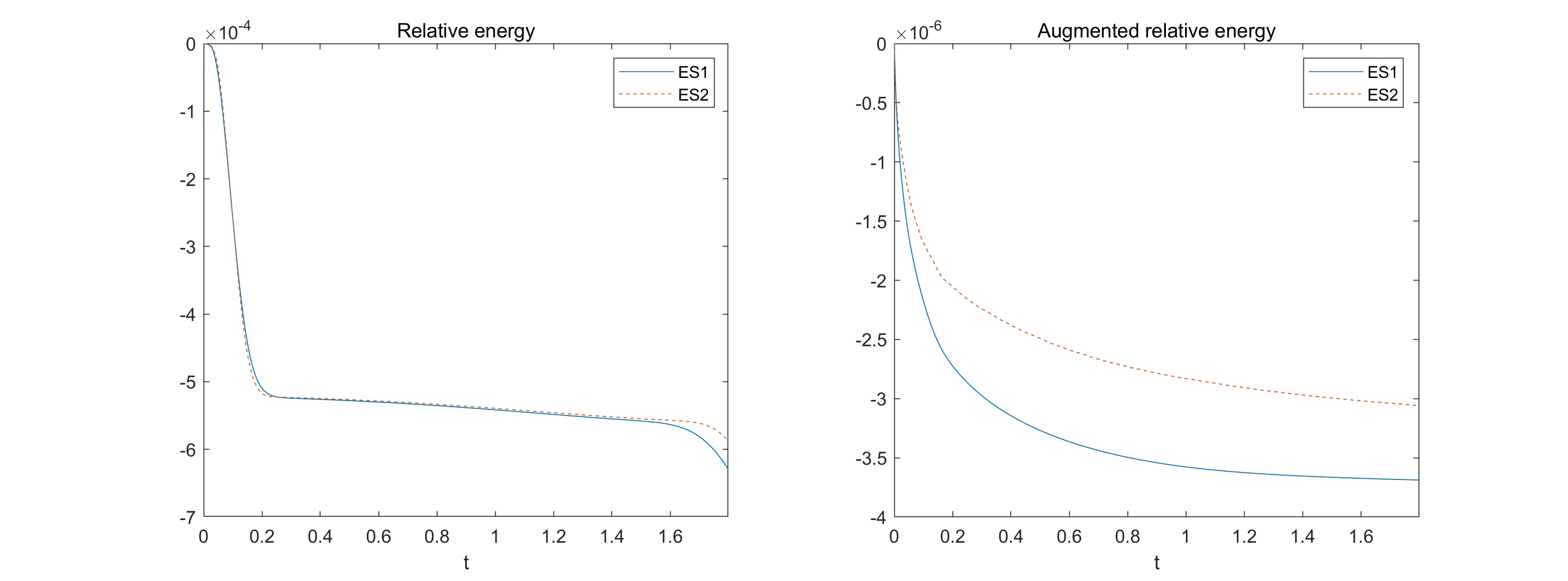}}
    \centering
    \captionsetup{font={scriptsize},justification=raggedright}
    \caption{Results for \cref{Sec5_3}. Left: relative change in energy. Right: relative change in augmented energy.}
    \label{fig_ex3_energy}
\end{figure}

\subsection{A submerged flat plateau}\label{Sec5_4}
We consider the deterministic initial water surface with a small deterministic perturbation of size $10^{-4}$ as follows:
\begin{align}
    w(x,y,0,\xi) &= \left\{
\begin{aligned}
&1.0001,&  &-0.4<x<-0.3, \\
  &1, &   &\text{otherwise,}
\end{aligned}
\right. &  u(x,y,0)&=v(x,y,0)=0,
\end{align}
and with a stochastic bottom function
\begin{align}
     B(x,y) = \left\{
\begin{aligned}
& 0.9998, &  &r\le 0.1, \\
  & 9.997(0.2-r) + 0.0001(\xi + 1) , & &0.1<r\le 0.2,\\
  & 0.0001, & &\text{otherwise,}
\end{aligned}
\right.
\end{align}
where $r \coloneqq \sqrt{x^2+y^2} + 0.0001$ and $\xi \sim \mathcal{U}(-1,1)$.

The bottom is a flat plateau close to the water surface with a small perturbation of magnitude $10^{-4}$. This topography is continuous but lacks smoothness in the physical domain. The test is a modification of the test from \cite{BEKP,dai2022hyperbolicity}. We compute the solutions of the ES schemes over the computational domain $[-0.5,0.5] \times [-0.5,0.5]$ at times $t=0.2, 0.35, 0.5, 0.65$. The initial deterministic perturbation generates two waves. The left-going wave travels out of the computational domain, while the right-going wave interacts with the plateau, leading to complex wave patterns as the interaction progresses. The results in Figure \ref{fig_ex4_solution} demonstrate that our schemes effectively capture the stochastic lake-at-rest solution, indicating that they are well-balanced. Both the ES1 and ES2 schemes dissipate energy, with the ES2 scheme providing higher resolution and less energy dissipation compared to the ES1 scheme. Note that the uncertainties of the ES1 scheme are not visible, since the ES1 standard deviation is much smaller than that of the ES2 scheme.

\begin{figure}[htbp]
    \centering  
    \includegraphics[width=1\textwidth]{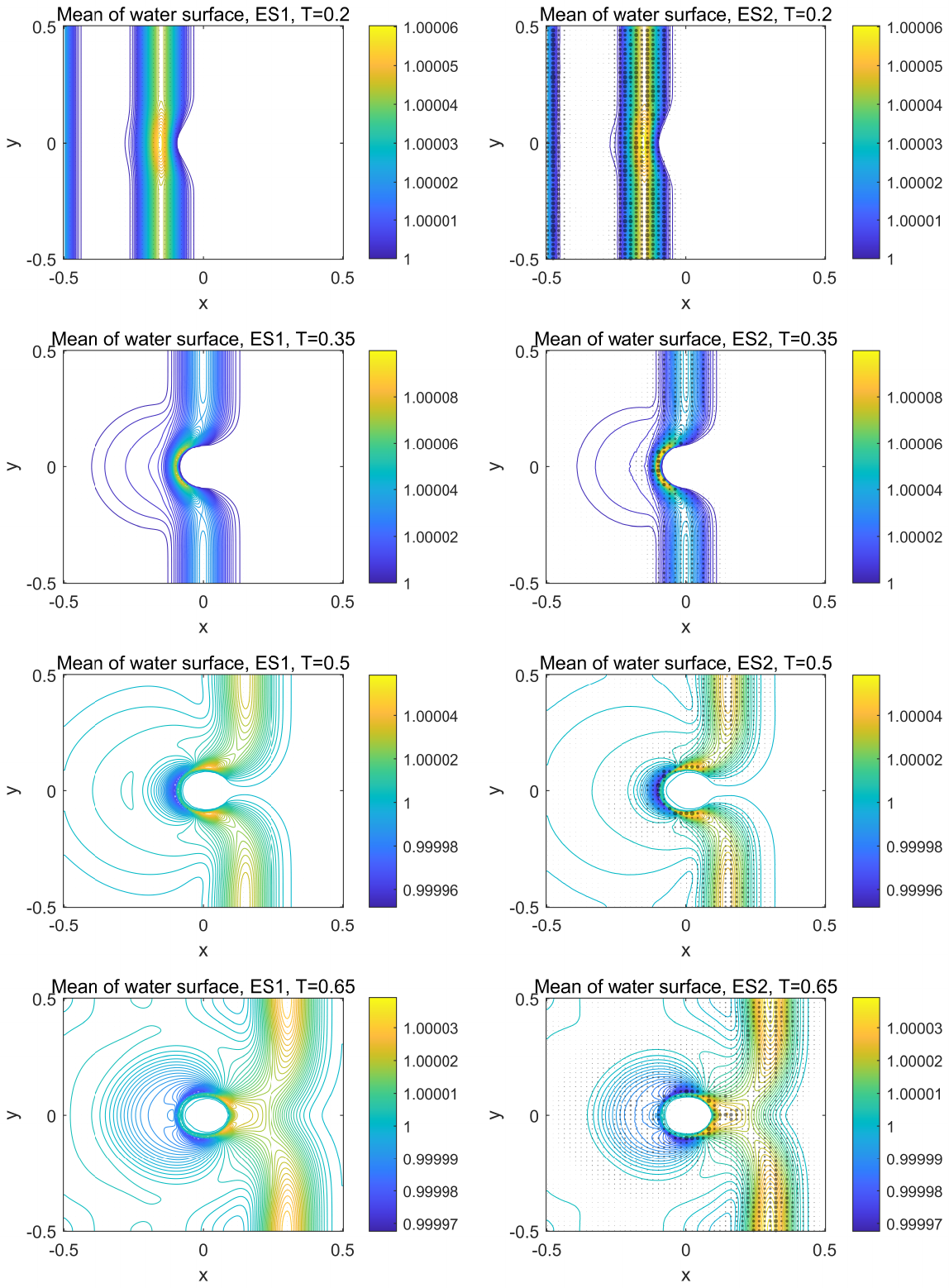}
    \centering
     \setlength{\abovecaptionskip}{-1cm}
    \captionsetup{font={scriptsize},justification=raggedright}
    \caption{Results for \cref{Sec5_4}. Mean of water surface. Disk-glyph over mean contours, where the radii of the disks indicate the magnitude of the standard deviation. Left: ES1, the maximum standard deviation is 1.7485e-08, 1.9120e-07, 1.3268e-07, 1.0392e-07, respectively. Right: ES2, the maximum standard deviation is 6.6484e-07, 1.8990e-06, 1.4625e-06,  8.1986e-07, respectively. $200\times 200$, $t = 0.2, 0.35, 0.5, 0.65$. }
    \label{fig_ex4_solution}
\end{figure}

\begin{figure}[htbp]
    \centering  \centerline{\includegraphics[width=1.1\textwidth]{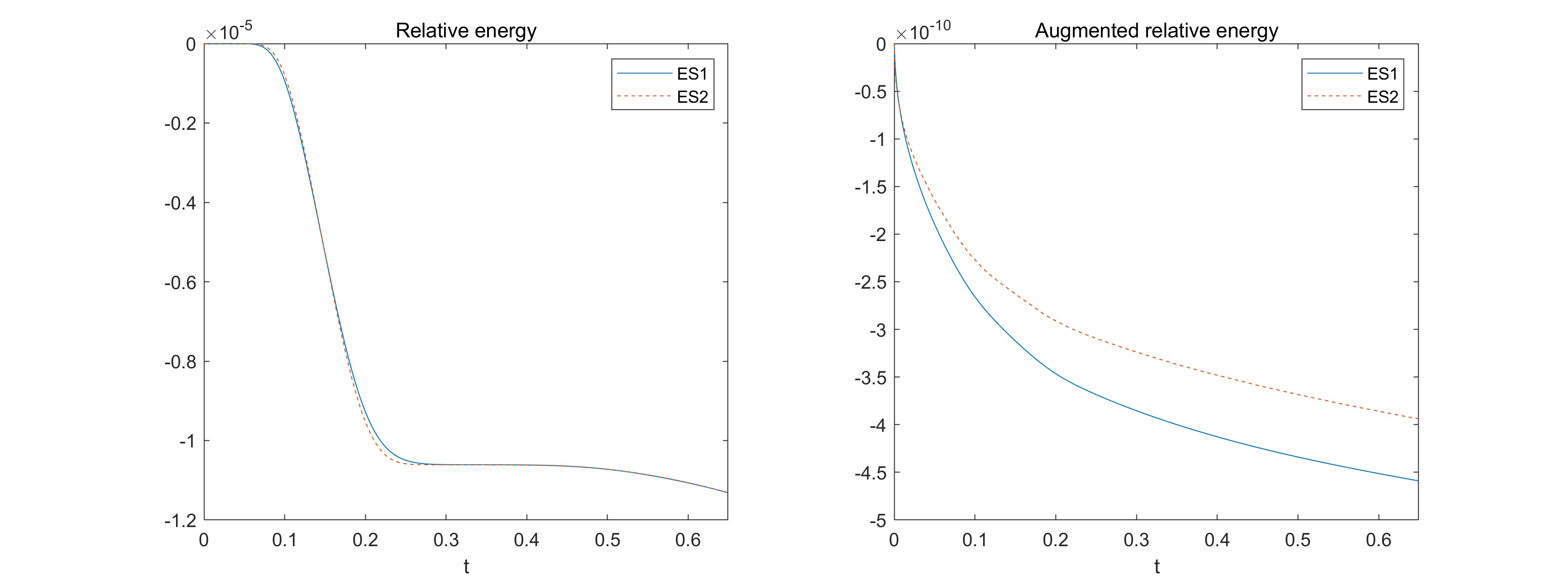}}
    \centering
    \captionsetup{font={scriptsize},justification=raggedright}
    \caption{Results of \cref{Sec5_4}. Left: relative change in energy. Right: relative change in augmented energy.}
    \label{fig_ex4_energy}
\end{figure}

The initial deterministic perturbation on the water surface is $10^{-4}$. Throughout the time evolution, the maximum mean water surface displacement from the steady state remains within this initial perturbation threshold, and the standard deviation remains small. These results demonstrate the capability of both the ES1 and ES2 schemes to accurately capture the lake-at-rest solution under small perturbations.  

For comparison, we also present a non-well-balanced ES2 scheme by modifying the numerical source term to
\begin{equation}\label{eq:S-nwb}
     \bm{S}_{i,j}  = \begin{pmatrix}
       0 \\
       -\frac{g}{2\Delta x} \mathcal{P}(\overline{\bm h}_{i,j}) \big(\llbracket \bm B\rrbracket_{i+\frac{1}{2},j} +   \llbracket \bm B\rrbracket_{i-\frac{1}{2},j}\big) \\
       -\frac{g}{2\Delta y} \mathcal{P}(\overline{\bm h}_{i,j}) \big(\llbracket \bm B\rrbracket_{i,j+\frac{1}{2}} +   \llbracket \bm B\rrbracket_{i,j-\frac{1}{2}}\big)
   \end{pmatrix},
\end{equation}
corresponding to a central differencing scheme for the derivative of bottom topography. 

We can observe from Figure \ref{fig_ex4_nowb_solution} that the wave structure produced by the non-well-balanced ES2 scheme differs significantly from that of the well-balanced schemes in Figure \ref{fig_ex4_solution}. The non-well-balanced scheme generates substantially larger perturbations compared to the well-balanced solutions. Note that even with mesh refinement, the non-well-balanced scheme cannot produce physically accurate wave structure, and the perturbations remain larger than the initial value throughout the time evolution.

\begin{figure}[htbp]
    \centering  \includegraphics[width=1\textwidth]{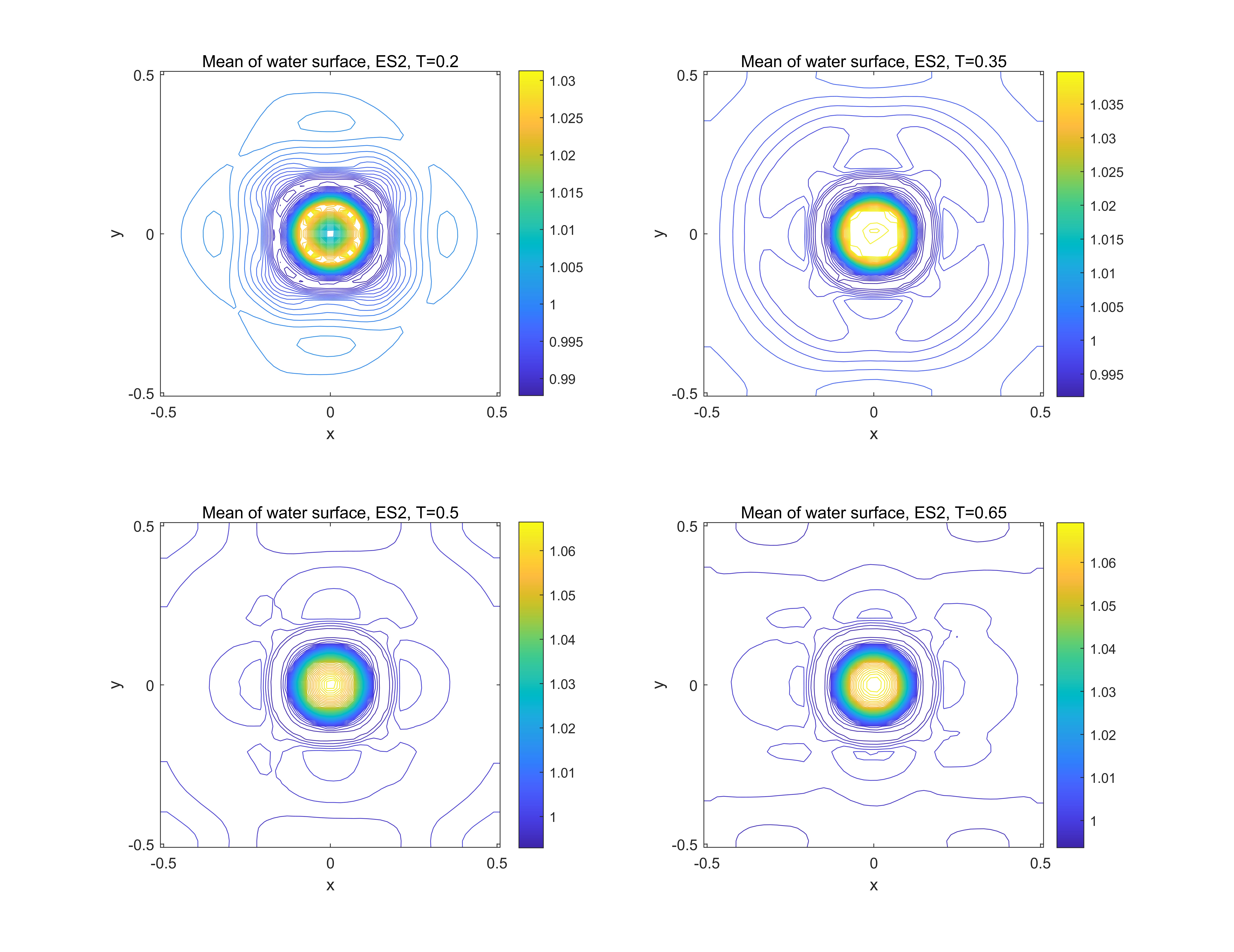}
    \centering
    \setlength{\abovecaptionskip}{-0.5cm}
    \captionsetup{font={scriptsize},justification=raggedright}
    \caption{Results for \cref{Sec5_4}. Mean of water surface. Non-well-balanced ES2 scheme. $50\times 50$, $t = 0.2, 0.35, 0.5, 0.65$.}
    \label{fig_ex4_nowb_solution}
\end{figure}

\subsection{Perturbation to lake-at-rest}\label{Sec5_5}
We consider another example of a perturbation to lake-at-rest solution with a stochastic initial water surface
\begin{align}
    w(x,y,0,\xi) &= \left\{
\begin{aligned}
&1 + 0.001(\xi+1), & |xy|\le 0.05, \\
  &1, &\text{otherwise,}
\end{aligned}
\right.& u(x,y,0)&=v(x,y,0)=0,
\end{align}
and with the deterministic bottom topography
\begin{subequations}
\begin{equation}\label{eq:lar-B1}
    B(x,y) = \left\{
\begin{aligned}
& 0.25(\cos(5\pi(xy+0.35))+1), & &-0.55<xy<-0.15, \\
  & 0.125(\cos(10\pi(xy-0.35))+1), & &0.25<xy<0.45,\\
  & 0, & &\text{otherwise.}
\end{aligned}
\right.
\end{equation}
The computational domain is $[-1,1]\times [-1,1]$ and $\xi \sim \mathcal{U}(-1,1)$. The test is a 2D modification of the test from \cite{chertock2015well,dai2024energy}.

The results presented in Figure \ref{fig_ex5_solution} indicate that the water propagates towards the four vertices of the rectangular computational domain. The wave moves slower after reaching the raised bottom. As in previous examples, both ES1 and ES2 schemes are well-balanced and produce a similar wave structure, with the ES2 scheme creating a more accurate solution with a higher resolution and less diffusion. Furthermore, uncertainties are distributed in proportion to the mean water surface in the ES schemes, as they originate from the initial water surface rather than the bottom topography. 

To investigate the impact of the bottom topography bump on the results, we slightly modify the bottom topography to
\begin{align}\label{eq:lar-B2}
  B(x,y) = \left\{ \begin{array}{ll} 
    0.45(\cos(5\pi(xy+0.35))+1),&  -0.55<xy<-0.15 \\
    0.125(\cos(10\pi(xy-0.35))+1),& 0.25<xy<0.45 \\
    0, & \textrm{otherwise},
  \end{array}\right.
\end{align}
\end{subequations}
and plot results for the ES1 and ES2 schemes in \cref{fig_ex5_solution2}.
Comparing the solutions in Figure \ref{fig_ex5_solution2} with those in Figure \ref{fig_ex5_solution} at $T = 0.6$ and $T = 0.8$, one can observe the significant influence of the bottom topography on the wave structure.

\begin{figure}[htbp]
    \centering  \includegraphics[width=1\textwidth]{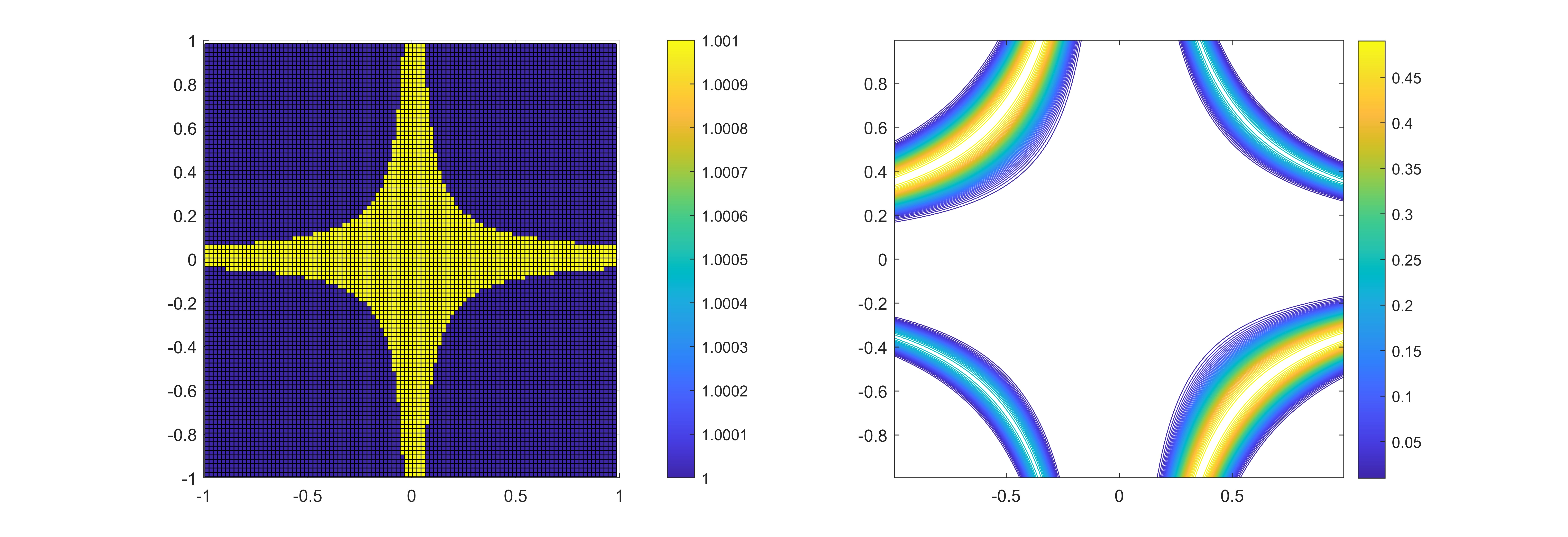}
    \centering
    \captionsetup{font={scriptsize},justification=raggedright}
    \caption{Results for \cref{Sec5_5}. Left: initial water surface. Right: bottom topography. $200\times 200$}
    \label{fig_ex5_bottom-initialdata}
\end{figure}

\begin{figure}[htbp]
    \centering  \includegraphics[width=1\textwidth]{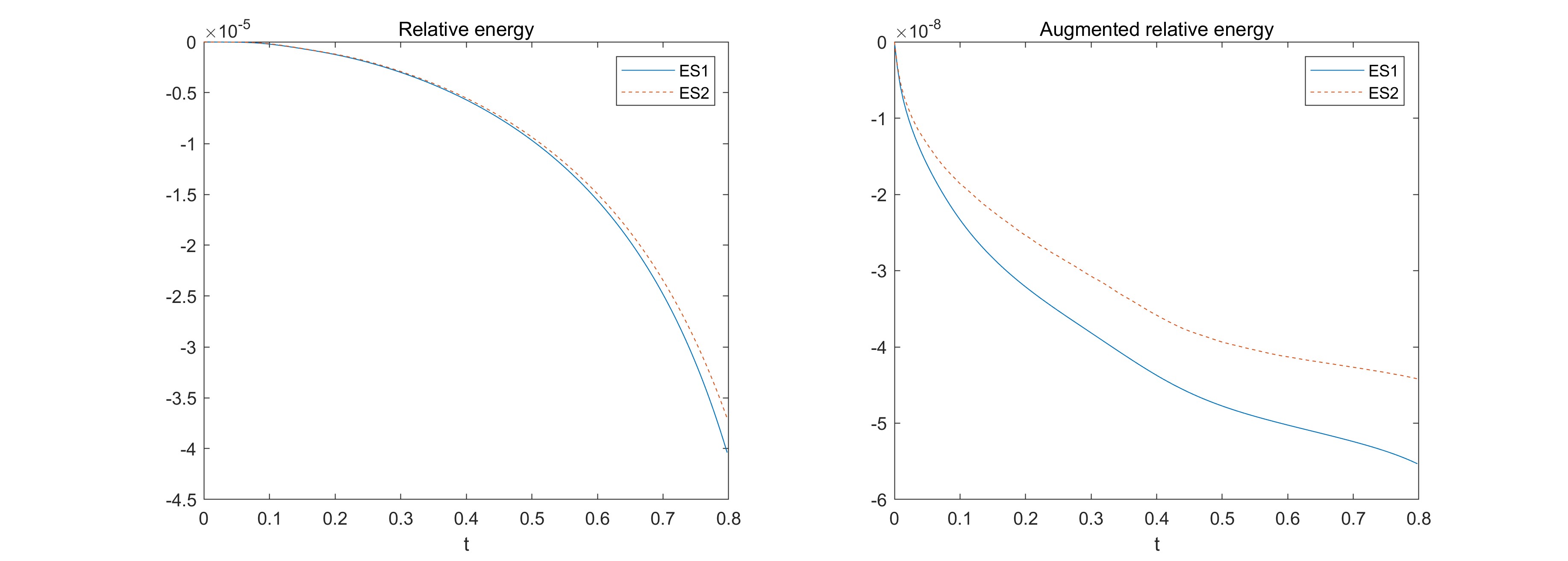}
    \centering
    \captionsetup{font={scriptsize},justification=raggedright}
    \caption{Results for \cref{Sec5_5}. Left: relative change in energy. Right: relative change in augmented energy. $200\times 200$}
    \label{fig_ex5_energy}
\end{figure}

\begin{figure}[htbp]
    \centering  \includegraphics[width=1\textwidth,height=1.35\textwidth]{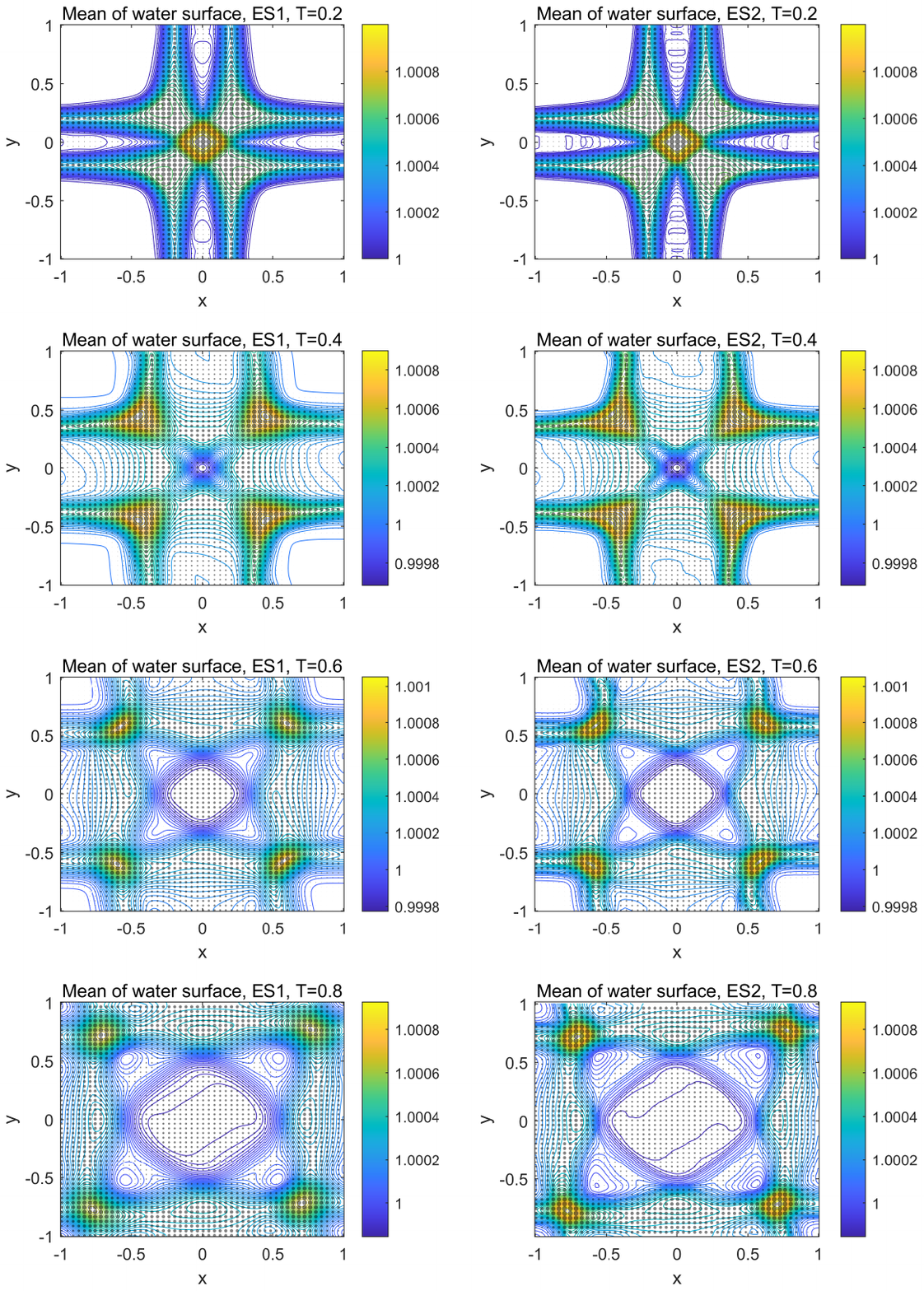}
    \centering
     \setlength{\abovecaptionskip}{-0.5cm}
    \captionsetup{font={scriptsize},justification=raggedright}
    \caption{Results for \cref{Sec5_5} with bottom topography $B$ given by \eqref{eq:lar-B1}. Contour plots for the water surface. Disk-glyph over mean contours, where the radii of the disks indicate the magnitude of the standard deviation. Left: ES1, the maximum standard deviation is  5.7522e-04, 4.6740e-04, 4.8587e-04, 3.9092e-04. Right: ES2, the maximum standard deviation is 5.7687e-04, 5.0798e-04,   5.8621e-04, 5.1600e-04. The mesh is $200\times 200$ with snapshots in time shown at $T = 0.2, 0.4, 0.6, 0.8.$}
    \label{fig_ex5_solution}
\end{figure}

\begin{figure}[htbp]
    \centering  \includegraphics[width=1\textwidth,height=1.35\textwidth]{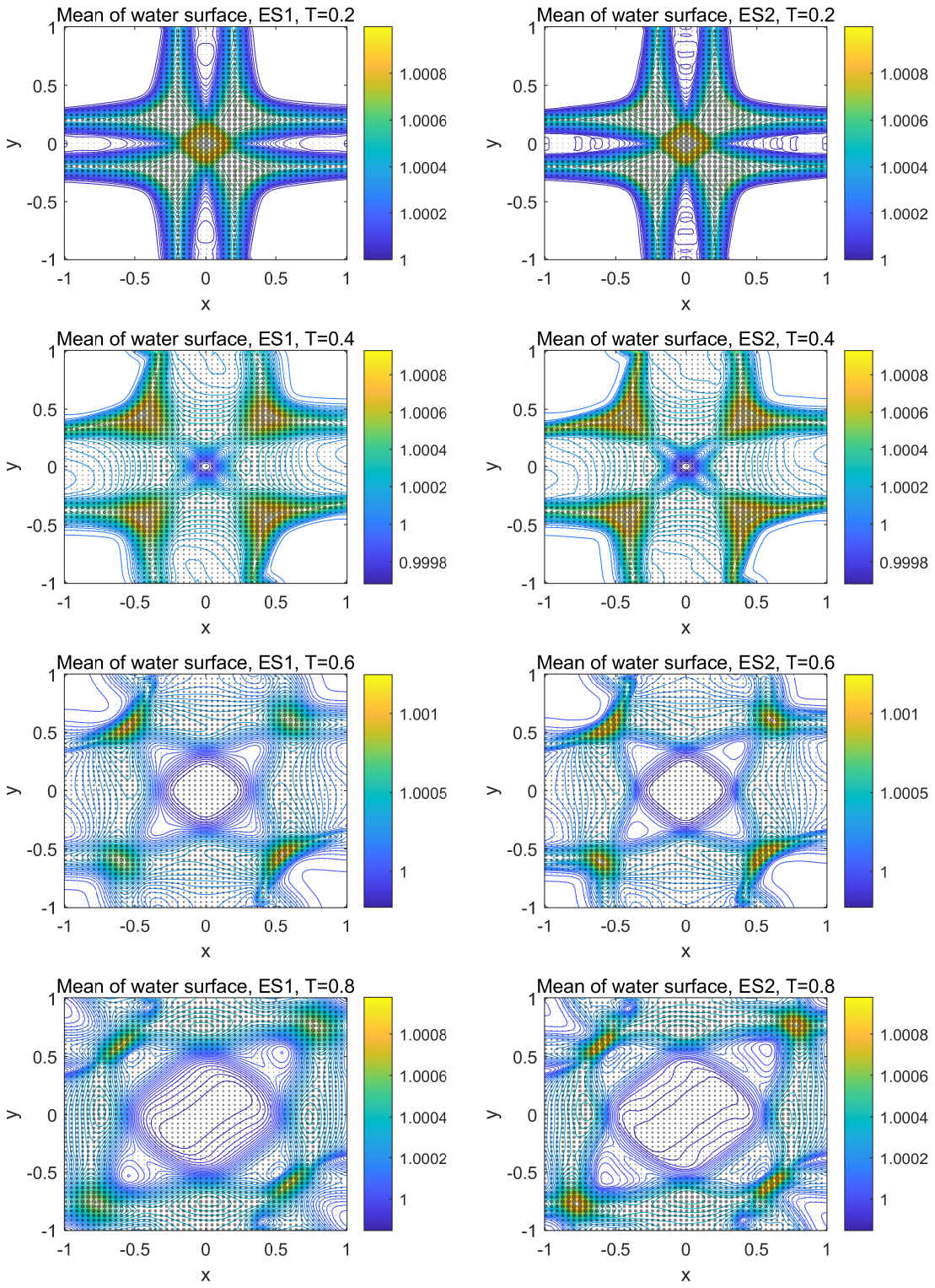}
    \centering
     \setlength{\abovecaptionskip}{-1cm}
    \captionsetup{font={scriptsize},justification=raggedright}
    \caption{Results for \cref{Sec5_5} with bottom topography $B$ given by \eqref{eq:lar-B2}. Contour plots for the water surface. Disk-glyph over mean contours, where the radii of the disks indicate the magnitude of the standard deviation. Left: ES1, the maximum standard deviation is 5.7522e-04, 4.8313e-04, 5.9729e-04, 4.0329e-04. Right: ES2, the maximum standard deviation is 5.7688e-04, 5.2388e-04, 7.0756e-04, 5.4857e-04.
    The mesh is $200\times 200$ with snapshots in time shown at $T = 0.2, 0.4, 0.6, 0.8.$}
    \label{fig_ex5_solution2}
\end{figure}

For comparison, we also simulate results using an ES2 scheme that employs the non-well-balanced source term discretization in \eqref{eq:S-nwb}.
From the results shown in Figure \ref{fig_ex5_nowb_solution}, the non-well-balanced ES2 scheme produces a different wave structure and generates a larger displacement of the water surface compared to the well-balanced ES1 and ES2 schemes. Again, note that refining the mesh does not produce the accurate wave structure, and
the perturbations during the time evolution are still larger than the initial perturbation for the non-well-balanced scheme.

\begin{figure}[htbp]
    \centering  \includegraphics[width=0.9\textwidth]{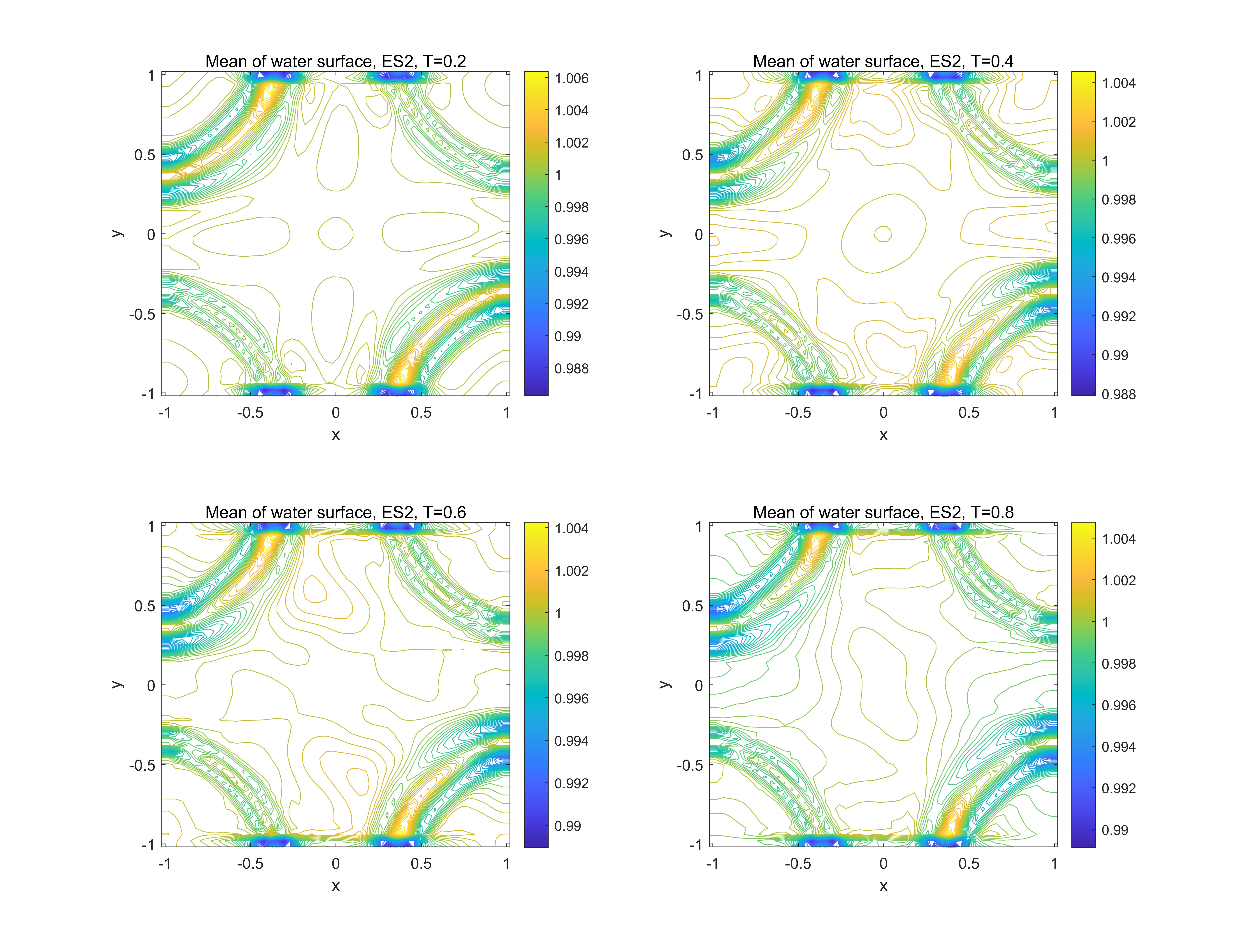}
    \centering
    \captionsetup{font={scriptsize},justification=raggedright}
    \caption{Results for \cref{Sec5_5}. Mean of water surface. Non-well-balanced ES2 scheme. $50\times 50$, $t = 0.2, 0.4, 0.6, 0.8$.}
    \label{fig_ex5_nowb_solution}
\end{figure}

\subsection{Gaussian-shape hump with two-dimensional random variable}\label{Sec5_6}
As a final test, we consider a two-dimensional in the stochastic space variant of the example used for the accuracy test provided in Section \ref{Sec5_1}, with a deterministic initial water surface
\begin{align}
    w(x,y,0,\xi) &= 1, & u(x,y,0,\xi) &= 0.3,& v(x,y,0,\xi) &= 0,
\end{align}
and a stochastic bottom
\begin{equation}
    B(x,y,\xi) = 0.5\exp(-12.5(\xi_1+1)(x-1)^2 - 25(\xi_2+1)(y-0.5)^2).
\end{equation}
Here $\xi = (\xi_1, \xi_2)$ is a two-dimensional random vector with independent components, where $\xi_1$ has a $\mathrm{Beta}(\beta+1,\alpha+1)$ distribution on $[-1,1]$ with parameters $(\alpha,\beta) = (3,1)$, and $\xi_2$ is uniformly distributed on $[-1,1]$. The uncertainties are in the width of the Gaussian-shape hump. To simulate the SG SWE system, we use $K_1 = K_2 =3$ polynomial terms for the parameters $\xi_1$ and $\xi_2$, resulting in a tensor-product space of polynomials with dimension $K = K_1K_2 = 9$. We compute solutions for $(x,y) \in [0,2]\times [0,1]$ up to time $T=0.07$. Due to the effect of the two-dimensional random variable, the solution in Figure \ref{fig_ex6_solution} has a more complex structure compared to the structure of the solution in the test in Figure \ref{fig_ex1}. The largest uncertainties are observed near the peak and bottom of the mean water surface, while the smallest uncertainties occur in the region between them. The energy plots in Figure \ref{fig_ex6_energy} demonstrate that measuring augmented energy is necessary to account for energy-related boundary effects in this example.

\begin{figure}[htbp]
    \centering  \includegraphics[width=1\textwidth]{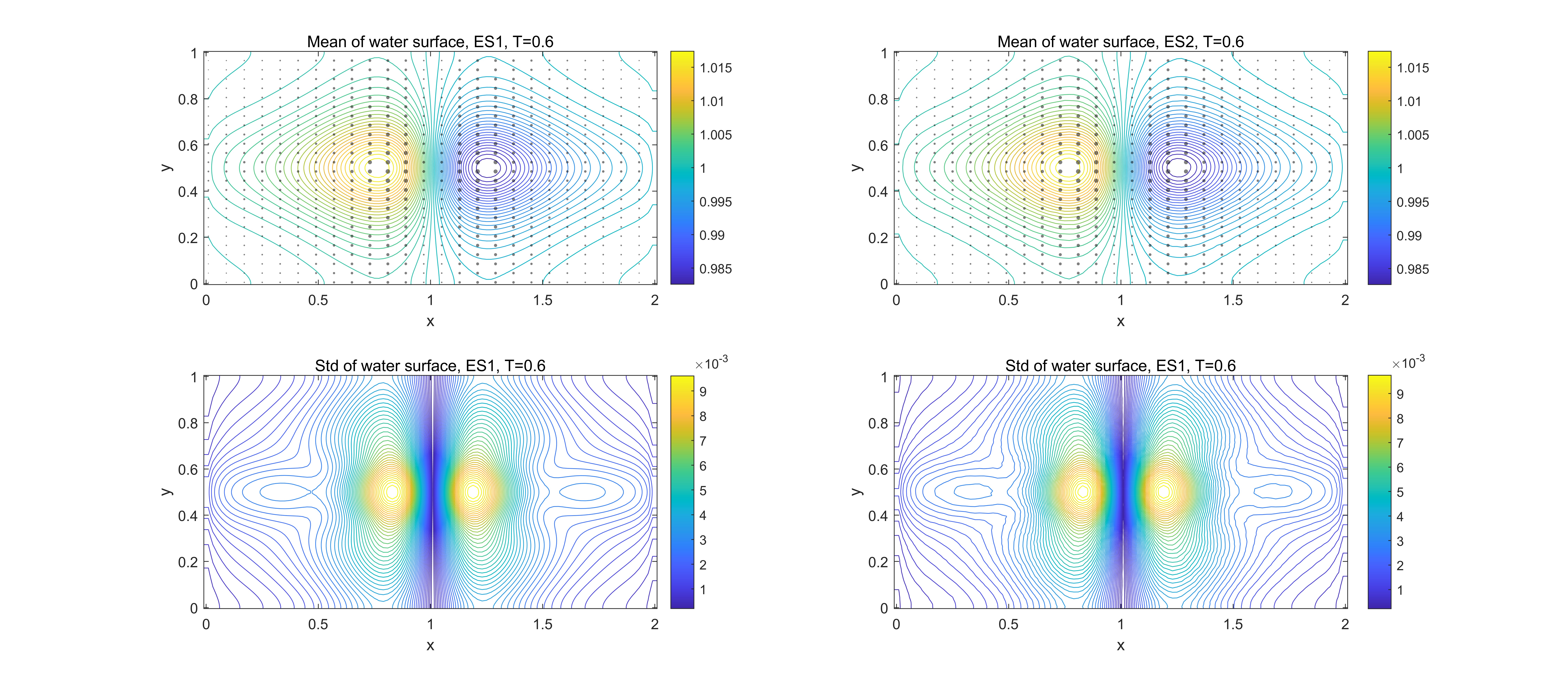}
    \centering
    \captionsetup{font={scriptsize},justification=raggedright}
    \caption{Results for \cref{Sec5_6}. Contour plots for the water surface. Disk-glyph over mean contours, where the radii of the disks indicate the magnitude of the standard deviation. Left: ES1, the maximum standard deviation is 9.8031e-03. Right: ES2, the maximum standard deviation is 9.9456e-03. $100\times 100, T = 0.07.$}
    \label{fig_ex6_solution}
\end{figure}

\begin{figure}[htbp]
    \centerline{\includegraphics[width=1\textwidth]{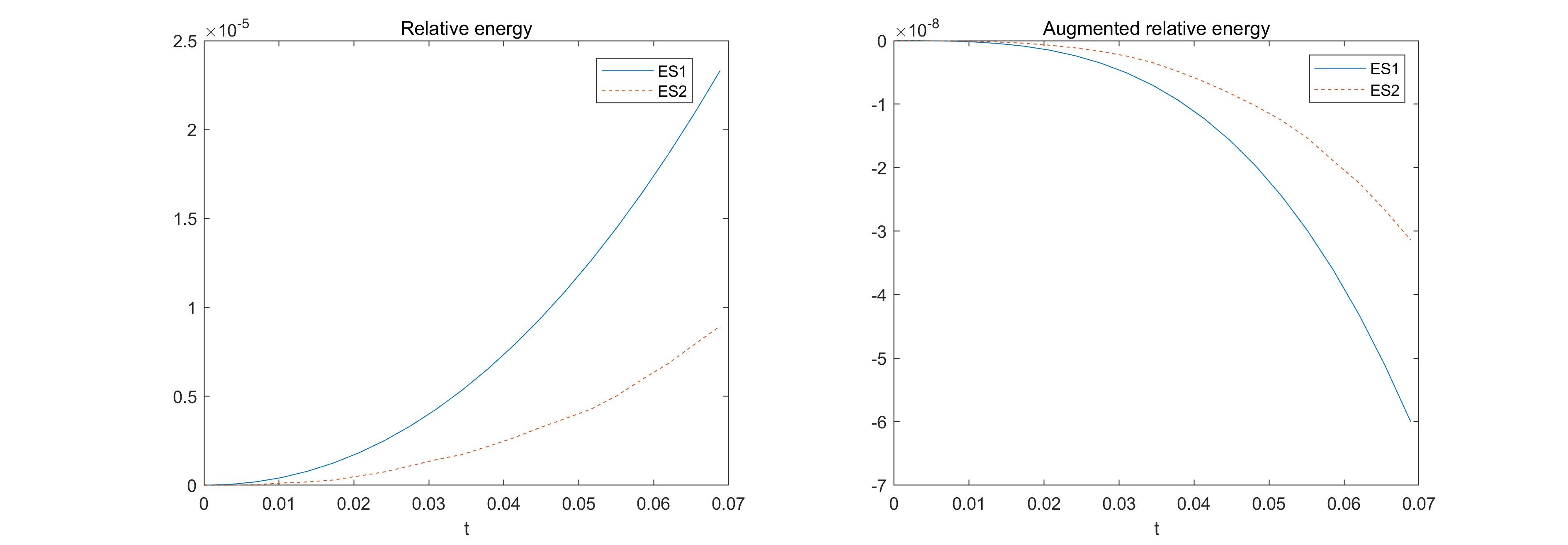}}
    \centering
    \captionsetup{font={scriptsize},justification=raggedright}
    \caption{Results for \cref{Sec5_6}. Left: relative change in energy. Right: relative change in augmented energy. $100\times 100.$}
    \label{fig_ex6_energy}
\end{figure}

\section{Conclusion}\label{Sec6}
In this work we derived an entropy flux pair for the two-dimensional hyperbolicity-preserving and positivity-preserving SG SWE system developed in \cite{dai2022hyperbolicity}. The entropy flux pair facilitates the formulation of entropy admissibility criteria to obtain the desired physical solution among non-unique weak solutions. Using this entropy pair, we constructed second-order energy conservative and first- and second-order energy stable finite volume schemes for the two-spatial-dimensional SG SWE, with the well-balanced property. We presented several numerical experiments demonstrating the efficiency and robustness of our schemes. Currently, our schemes are limited to second-order spatial accuracy; however, they can be extended to higher orders by developing higher-order energy conservative finite volume schemes, more accurate diffusion operators, or employing discontinuous Galerkin methods. Another possibility for future research is to extend the developed framework to models associated with dry/wet interfaces.

\section*{Acknowledgement}
The work of Yekaterina Epshteyn, Akil Narayan and Yinqian Yu was partially supported by NSF DMS-2207207. AN was also partially supported
by NSF DMS-1848508.

\footnotesize{\bibliographystyle{plain} 
\bibliography{reference} 

\appendix

\section{Proof of \cref{Thm_entropy_flux_pair}}\label{sec:entropy-tuple-proof}
\Cref{Thm_entropy_flux_pair} is a direct result of convexity of the function $E(\widehat{U})$, proven in \cref{ssec:E-convexity-proof}, and also of the fact that $(E,H,K)$ satisfy the companion balance law, proven in \cref{ssec:companion-balance-law-proof}. The proofs of these results require the technical result of \cref{lemma_gradient}, so we begin with that result.

\subsection{Proof of \cref{lemma_gradient}}
The proof is based on the calculations in Lemma 3.1 of \cite{dai2024energy}, where the partial derivative of the inverse of a parameterized matrix is computed. Let $\widehat{h}_l$ be an arbitrary component of $\widehat{h}$, then compute the partial derivative
\begin{equation}
    \frac{\partial \mathcal{P}^{-1}(\widehat{h})}{\partial \widehat{h}_l} = - \mathcal{P}^{-1}(\widehat{h}) \frac{\partial \mathcal{P}(\widehat{h})}{\partial \widehat{h}_l}\mathcal{P}^{-1}(\widehat{h}) = - \mathcal{P}^{-1}(\widehat{h}) \mathcal{M}_l\mathcal{P}^{-1}(\widehat{h}).
\end{equation}
Multiply $\widehat{q^x}$ on both sides, then 
\begin{equation}
    \frac{\partial \widehat{u}}{\partial \widehat{h}_l} = \frac{\partial \mathcal{P}^{-1}(\widehat{h})}{\partial \widehat{h}_l}\widehat{q^x} =  - \mathcal{P}^{-1}(\widehat{h}) \mathcal{M}_l\mathcal{P}^{-1}(\widehat{h}) \widehat{q^x} = - \mathcal{P}^{-1}(\widehat{h}) \mathcal{M}_l\widehat{u},
\end{equation}
resulting in 
\begin{equation}
    \frac{\partial \widehat{u}}{\partial \widehat{h}} = -\mathcal{P}^{-1}(\widehat{h})[\mathcal{M}_1 \widehat{u},\dots,\mathcal{M}_K \widehat{u}  ] \overset{\eqref{commutative_prop}}{=} -\mathcal{P}^{-1}(\widehat{h})\mathcal{P}(\widehat{u}).
\end{equation}
Similarly, we also get
\begin{equation}
     \frac{\partial \widehat{v}}{\partial \widehat{h}} = -\mathcal{P}^{-1}(\widehat{h})\mathcal{P}(\widehat{v}).
\end{equation}
In addition, by definition \eqref{velocity_coeff}, it is straightforward that
\begin{align}
    \frac{\partial \widehat{u}}{\partial \widehat{q^x}}& = \frac{\partial \widehat{v}}{\partial \widehat{q^y}} = \mathcal{P}^{-1}(\widehat{h}),& \quad \frac{\partial \widehat{v}}{\partial \widehat{q^x}} &= \frac{\partial \widehat{u}}{\partial \widehat{q^y}} = 0.
\end{align}
This completes the proof.

\subsection{Proof of \cref{lemma_convexity}}\label{ssec:E-convexity-proof}
    The proof is straightforward and follows from computing the Hessian to show that it is positive definite. Recall the entropy function,
     \begin{equation}
        E(\widehat{U}) = \frac{1}{2}\Big((\widehat{q^x})^\top \widehat{u} + (\widehat{q^y})^\top \widehat{v}\Big) + \frac{1}{2}g\|\widehat{h}\|^2 + g\widehat{h}^\top \widehat{B},
    \end{equation}
where we denote the terms separately,
\begin{align}
    f_1 &\coloneqq \frac{1}{2}(\widehat{q^x})^\top \widehat{u},&   f_2 &\coloneqq \frac{1}{2}(\widehat{q^y})^\top \widehat{v}, & f_3 &\coloneqq \frac{1}{2}g\|\widehat{h}\|^2 + g\widehat{h}^\top \widehat{B}.
\end{align}
By Lemma \ref{lemma_gradient}, we compute the first-order partial derivatives of $f_1$,
\begin{equation}
   \begin{split}
     &  \frac{\partial f_1}{ \partial \widehat{h}} = \frac{1}{2} (\widehat{q^x})^\top \frac{\partial \widehat{u}}{\partial \widehat{h}} \overset{\eqref{gradient_velocity}}{=} -\frac{1}{2} (\widehat{q^x})^\top \mathcal{P}^{-1}(\widehat{h})\mathcal{P}(\widehat{u}) = -\frac{1}{2} \Big( \mathcal{P}^{-1}(\widehat{h})\widehat{q^x}
 \Big)^\top \mathcal{P}(\widehat{u}) = -\frac{1}{2}\widehat{u}^\top \mathcal{P}(\widehat{u}),\\
 & \frac{\partial f_1}{\partial \widehat{q^x}} \overset{\eqref{gradient_velocity}}{=} \widehat{u}^\top, \quad \frac{\partial f_1}{\partial \widehat{q^y}} = 0,
   \end{split}
\end{equation}
and the second-order partial derivatives,
\begin{equation}
    \begin{split}
         \frac{ \partial^2 f_1}{\partial \widehat{h}^2} &= -\frac{1}{2}\frac{\partial}{\partial \widehat{h}} \big( \widehat{u}^\top  \mathcal{P}(\widehat{u})\big) \overset{\eqref{gradient_velocity}}{=} \mathcal{P}(\widehat{u}) \mathcal{P}^{-1}(\widehat{h})\mathcal{P}(\widehat{u}),\\
         \frac{\partial^2 f_1}{ \partial \widehat{h} \partial \widehat{q^x}} &\overset{\eqref{gradient_velocity}}{=} - \mathcal{P}(\widehat{u}) \mathcal{P}^{-1}(\widehat{h})\\  \frac{\partial^2 f_1}{\partial \widehat{q^x}^2} &\overset{\eqref{gradient_velocity}}{=} \mathcal{P}^{-1}(\widehat{h}),
    \end{split}
\end{equation}
with the remaining elements in the Hessian of $f_1$ equal to zero, i.e.,
\begin{equation}
    \frac{ \partial^2 f_1}{\partial \widehat{U}^2} = \begin{pmatrix}
        \mathcal{P}(\widehat{u}) \mathcal{P}^{-1}(\widehat{h})\mathcal{P}(\widehat{u}) &  - \mathcal{P}(\widehat{u}) \mathcal{P}^{-1}(\widehat{h}) & 0 \\
        -\mathcal{P}^{-1}(\widehat{h}) \mathcal{P}(\widehat{u})  &  \mathcal{P}^{-1}(\widehat{h}) & 0 \\
        0 & 0 & 0
    \end{pmatrix}.
\end{equation}
Similarly to $f_1$, we compute the Hessian of $f_2$ as follows,
\begin{equation}
    \frac{ \partial^2 f_2}{\partial \widehat{U}^2} = \begin{pmatrix}
        \mathcal{P}(\widehat{v}) \mathcal{P}^{-1}(\widehat{h})\mathcal{P}(\widehat{v})  & 0 & - \mathcal{P}(\widehat{v}) \mathcal{P}^{-1}(\widehat{h}) \\
        0  &  0 & 0 \\
        - \mathcal{P}^{-1}(\widehat{h})\mathcal{P}(\widehat{v})  & 0 & \mathcal{P}^{-1}(\widehat{h})
    \end{pmatrix}.
\end{equation}
In addition, a direct computation gives
\begin{align}
   \frac{\partial f_3}{\partial \widehat{U}} &=  \Big( g(\widehat{h}+\widehat{B})^\top, 0, 0 \Big) & \implies&  &\frac{\partial^2 f_3}{\partial \widehat{U}^2} &= \begin{pmatrix}
       gI & 0 & 0\\ 0& 0 & 0 \\ 0 &0 & 0
   \end{pmatrix}.
\end{align}
Finally, for any vector $\big( w_1^\top, w_2^\top, w_3^\top \big)^\top \in \mathbb{R}^{3K}$, we compute the quadratic form associated with the Hessian matrices as follows:
\begin{equation}
\begin{split}
    &  \big( w_1^\top, w_2^\top, w_3^\top \big) \frac{ \partial^2 f_1}{\partial \widehat{U}^2}\begin{pmatrix}
        w_1 \\w_2 \\ w_3
    \end{pmatrix}  = \big( \mathcal{P}(\widehat{u})w_1 - w_2 \big)^\top \mathcal{P}^{-1}(\widehat{h})\big( \mathcal{P}(\widehat{u})w_1 - w_2 \big) \ge 0,\\
    &  \big( w_1^\top, w_2^\top, w_3^\top \big) \frac{ \partial^2 f_2}{\partial \widehat{U}^2}\begin{pmatrix}
        w_1 \\w_2 \\ w_3
    \end{pmatrix}  = \big( \mathcal{P}(\widehat{v})w_1 - w_3 \big)^\top \mathcal{P}^{-1}(\widehat{h})\big( \mathcal{P}(\widehat{v})w_1 - w_3 \big) \ge 0,\\
    & \big( w_1^\top, w_2^\top, w_3^\top \big) \frac{ \partial^2 f_3}{\partial \widehat{U}^2}\begin{pmatrix}
        w_1 \\w_2 \\ w_3
    \end{pmatrix}  = g \|w_1\|^2 \ge 0.
\end{split}
\end{equation}
This results in that the quadratic form associated with the Hessian of $E$ being non-negative, provided that $\mathcal{P}$ is positive definite. In addition, the quadratic form vanishes if and only if $w_1 = w_2 = w_3 = \bm 0$, implying that $E$ is strictly convex in $\widehat{U}$, completing the proof.


\subsection{Proof of \cref{lemma_companion_bl}}\label{ssec:companion-balance-law-proof}
Instead of directly demonstrating \eqref{stochastic_companion_bl}, we will present the equivalent compatibility condition,
\begin{equation}\label{compatibility_condition}
    \frac{\partial E}{\partial \widehat{U}}\Big( \frac{\partial \widehat{F}}{\partial \widehat{U}} \frac{\partial \widehat{U}}{\partial x} + \frac{\partial \widehat{G}}{\partial \widehat{U}} \frac{\partial \widehat{U}}{\partial y}  - \widehat{S} \Big) = \frac{\partial H}{\partial x} + \frac{\partial K}{\partial y}.
\end{equation}
We divide each of the entropy and flux functions into two parts, respectively,
\begin{align}\label{Lemma4_eq1}
        E(\widehat{U}) &= E_1 + E_2, & E_1 &=  \frac{1}{2}\Big((\widehat{q^x})^\top \widehat{u} + (\widehat{q^y})^\top \widehat{v}\Big) + \frac{1}{2}g\|\widehat{h}\|^2,& E_2 &=  g\widehat{h}^{\top} \widehat{B},\\
        H(\widehat{U}) &= H_1 + H_2, & H_1 &= \frac{1}{2} \Big(\widehat{u}^\top \mathcal{P}(\widehat{q^x})\widehat{u} + \widehat{v}^\top \mathcal{P}(\widehat{q^x})\widehat{v} \Big) + g (\widehat{q^x})^\top\widehat{h}, & H_2 &= g (\widehat{q^x})^\top\widehat{B},\\
        K(\widehat{U}) &= K_1 + K_2, & K_1 &= \frac{1}{2}\Big(\widehat{v}^\top \mathcal{P}(\widehat{q^y})\widehat{v} + \widehat{u}^\top \mathcal{P}(\widehat{q^y})\widehat{u} \Big) + g(\widehat{q^y})^\top\widehat{h}& K_2 &= g(\widehat{q^y})^\top\widehat{B}.
\end{align}
The computation in Lemma \ref{lemma_convexity} gives,
\begin{align}\label{derivative_E1}
    \frac{\partial E_1}{\partial \widehat{U}} &= \Big( -\frac{1}{2}\widehat{u}^\top \mathcal{P}(\widehat{u})-\frac{1}{2}\widehat{v}^\top \mathcal{P}(\widehat{v}) + g\widehat{h}^\top, \widehat{u}^\top, \widehat{v}^\top \Big), & \frac{\partial E_2}{\partial \widehat{U}} &= \big( g\widehat{B}^\top, 0,  0 \big),
\end{align}
which implies
\begin{equation}
    -\frac{\partial E_1}{\partial \widehat{U} } \widehat{S} =g(\widehat{q^x})^\top \widehat{B}_x + g(\widehat{q^y})^\top \widehat{B}_y.
\end{equation}
The equations above, the flux Jacobians in \eqref{SG_Jacobians}, and the source term in \eqref{exact_source_term}, indicate
\begin{align}
       & -\frac{\partial E_1}{\partial \widehat{U}}\widehat{S} +   \frac{\partial E_2}{\partial \widehat{U}}\Big( \frac{\partial \widehat{F}}{\partial \widehat{U}} \frac{\partial \widehat{U}}{\partial x} + \frac{\partial \widehat{G}}{\partial \widehat{U}} \frac{\partial \widehat{U}}{\partial y}  - \widehat{S} \Big)\\
      = &   g(\widehat{q^x})^\top \widehat{B}_x + g(\widehat{q^y})^\top \widehat{B}_y + g\widehat{B}^\top \big( \frac{\partial\widehat{q^x}}{\partial x} + \frac{\partial\widehat{q^y}}{\partial y} \big)\\
      = & \big( g(\widehat{q^x})^\top \widehat{B}_x  + g\widehat{B}^\top \frac{\partial \widehat{q^x}}{\partial x} \big) + \big( g(\widehat{q^y})^\top \widehat{B}_y  + g\widehat{B}^\top \frac{\partial \widehat{q^y}}{\partial y} \big)\\
      = & \frac{\partial H_2}{\partial x} + \frac{\partial K_2}{\partial y}.
\end{align}
Therefore, the compatibility condition \eqref{compatibility_condition} is equivalent to the following condition,
\begin{equation}
        \frac{\partial E_1}{\partial \widehat{U}}\Big( \frac{\partial \widehat{F}}{\partial \widehat{U}} \frac{\partial \widehat{U}}{\partial x} + \frac{\partial \widehat{G}}{\partial \widehat{U}} \frac{\partial \widehat{U}}{\partial y}  \Big) = \frac{\partial H_1}{\partial x} + \frac{\partial K_1}{\partial y} = \frac{\partial H_1}{\partial \widehat{U}}\frac{\partial \widehat{U}}{\partial x} + \frac{\partial K_1}{\partial \widehat{U}}\frac{\partial \widehat{U}}{\partial y}.
\end{equation}
It suffices to show
\begin{align}
      \frac{\partial E_1}{\partial \widehat{U}}   \frac{\partial \widehat{F}}{\partial \widehat{U}} &= \frac{\partial H_1}{\partial \widehat{U}}, &  \frac{\partial E_1}{\partial \widehat{U}}   \frac{\partial \widehat{G}}{\partial \widehat{U}} &= \frac{\partial K_1}{\partial \widehat{U}},
\end{align}
which can be computed directly from \eqref{derivative_E1}, \eqref{SG_Jacobians}, \eqref{Lemma4_eq1}, and \eqref{gradient_velocity}. 


\section{Proofs for \cref{Sec4_2}}\label{app:ec-proofs}
In this section we provide proofs for \cref{lemma_LTE_EC,lemma_wellbalanced,lemma_sufficientcond_EC} in \cref{Sec4_2}, which are the crucial ingredients to proving that the choice of numerical fluxes \eqref{EC_flux} yields a second-order, well-balanced, EC scheme.

\subsection{Proof of \cref{lemma_LTE_EC}}\label{sec:LTE_EC}
    Assume $(\widehat{U},\widehat{B})$ are spatially smooth functions and $(\bm U_i, \bm B_i)$ are the exact cell averages. It suffices to compare the numerical fluxes $\mathcal{F}_{i+\frac{1}{2},j}, \mathcal{G}_{i,j+\frac{1}{2}}$ and the numerical source term $\bm S_{i,j}$ to the exact fluxes $\widehat{F}(\widehat{U})$ at $(x_{i+\frac{1}{2}},y_j)$, and $\widehat{G}(\widehat{U})$ at $(x_i,y_{j+\frac{1}{2}})$, and the exact source function $\widehat{S}(\widehat{U})$ at $(x_{i},y_j)$, respectively. We can rewrite $\widehat{F}(\widehat{U})$ and $ \widehat{G}(\widehat{U})$ in \eqref{exact_flux_term} as
\begin{equation}\label{exact_flux_rewrite}
    \begin{split}
        \widehat{F}(\widehat{U}) \overset{\eqref{velocity_coeff}}{=} & \begin{pmatrix}
            \mathcal{P}(\widehat{h})\widehat{u}\\
            \mathcal{P}(\widehat{q^x})\widehat{u} +  \frac{1}{2}g \mathcal{P}(\widehat{h})\widehat{h}\\
            \mathcal{P}(\widehat{q^x})\widehat{v} 
        \end{pmatrix} \overset{\eqref{commutative_prop},\eqref{velocity_coeff}}{=} 
        \begin{pmatrix}
            \mathcal{P}(\widehat{h})\widehat{u}\\
            \mathcal{P}(\widehat{u})\mathcal{P}(\widehat{h})\widehat{u}+  \frac{1}{2}g \mathcal{P}(\widehat{h})\widehat{h}\\
            \mathcal{P}(\widehat{v})\mathcal{P}(\widehat{h})\widehat{u}
        \end{pmatrix},\\
         \widehat{G}(\widehat{U}) \overset{\eqref{velocity_coeff}}{=} & \begin{pmatrix}
            \mathcal{P}(\widehat{h})\widehat{v}\\
            \mathcal{P}(\widehat{q^y})\widehat{u} \\
            \mathcal{P}(\widehat{q^y})\widehat{v} +  \frac{1}{2}g \mathcal{P}(\widehat{h})\widehat{h}
        \end{pmatrix} \overset{\eqref{commutative_prop},\eqref{velocity_coeff}}{=} 
        \begin{pmatrix}
            \mathcal{P}(\widehat{h})\widehat{v}\\
            \mathcal{P}(\widehat{u})\mathcal{P}(\widehat{h})\widehat{v}\\
            \mathcal{P}(\widehat{v})\mathcal{P}(\widehat{h})\widehat{v}+  \frac{1}{2}g \mathcal{P}(\widehat{h})\widehat{h}
        \end{pmatrix}.
    \end{split}
\end{equation}
Notice the following quantitative approximations in space:
\begin{subequations}
\begin{align}\label{quantitative_approximations}
  \overline{\bm U}_{i+\frac{1}{2},j} &=   \widehat{U}(x_{i+\frac{1}{2}},y_j) + \mathcal{O}(\Delta x^2),& \overline{\bm U}_{i,j+\frac{1}{2}} &= \widehat{U}(x_{i},y_{j+\frac{1}{2}}) + \mathcal{O}(\Delta y^2),\\\label{eq:taylor-ubar}
       \overline{\bm u}_{i+\frac{1}{2},j} &=   \widehat{u}(x_{i+\frac{1}{2}},y_j) + \mathcal{O}(\Delta x^2),& \overline{\bm u}_{i,j+\frac{1}{2}} &= \widehat{u}(x_{i},y_{j+\frac{1}{2}}) + \mathcal{O}(\Delta y^2),\\
        \overline{\bm v}_{i+\frac{1}{2},j} &=   \widehat{v}(x_{i+\frac{1}{2}},y_j) + \mathcal{O}(\Delta x^2),& \overline{\bm v}_{i,j+\frac{1}{2}} &= \widehat{v}(x_{i},y_{j+\frac{1}{2}}) + \mathcal{O}(\Delta y^2),\\
        \llbracket \bm U \rrbracket_{i+\frac{1}{2},j} &=  \Delta x \widehat{U}_x(x_{i+\frac{1}{2}},y_j) + \mathcal{O}(\Delta x^2), &   \llbracket \bm U \rrbracket_{i,j+\frac{1}{2}} &=  \Delta y \widehat{U}_y(x_{i},y_{j+\frac{1}{2}}) + \mathcal{O}(\Delta y^2),\\     \label{eq:taylor-Phbar}
        \mathcal{P}(\overline{\bm h}_{i+\frac{1}{2},j} ) &= \mathcal{P}( \widehat{h}(x_{i+\frac{1}{2}},y_j)) + \mathcal{O}(\Delta x^2), & \mathcal{P}(\overline{\bm h}_{i,j+\frac{1}{2}} ) &= \mathcal{P}( \widehat{h}(x_{i},y_{j+\frac{1}{2}})) + \mathcal{O}(\Delta y^2),\\
        \mathcal{P}(\overline{\bm u}_{i+\frac{1}{2},j} ) &= \mathcal{P}( \widehat{u}(x_{i+\frac{1}{2}},y_j)) + \mathcal{O}(\Delta x^2), &\mathcal{P}(\overline{\bm u}_{i,j+\frac{1}{2}} )& = \mathcal{P}( \widehat{u}(x_{i},y_{j+\frac{1}{2}})) + \mathcal{O}(\Delta y^2),\\
        \mathcal{P}(\overline{\bm v}_{i+\frac{1}{2},j} ) &= \mathcal{P}( \widehat{v}(x_{i+\frac{1}{2}},y_j)) + \mathcal{O}(\Delta x^2), &\mathcal{P}(\overline{\bm v}_{i,j+\frac{1}{2}} )& = \mathcal{P}( \widehat{v}(x_{i},y_{j+\frac{1}{2}})) + \mathcal{O}(\Delta y^2).
\end{align}
\end{subequations}
With the quantitative approximations above, it is straightforward to show that the numerical fluxes \eqref{EC_flux} are second-order approximations of the exact fluxes \eqref{exact_flux_rewrite} respectively, i.e.,
\begin{align}
    \mathcal{F}_{i+\frac{1}{2},j}^{EC} &= \widehat{F}(\widehat{U})|_{x_{i+\frac{1}{2}},y_j} + \mathcal{O}(\Delta x^2),& \mathcal{G}_{i,j+\frac{1}{2}}^{EC} &= \widehat{G}(\widehat{U})|_{x_{i},y_{j+\frac{1}{2}}} + \mathcal{O}(\Delta y^2).
\end{align}

Moreover, to show that the scheme \eqref{FV-SGSWE} has a second-order spatial local truncation error, it suffices to show that $\frac{\mathcal{F}_{i+\frac{1}{2},j}^{EC} - \mathcal{F}_{i-\frac{1}{2},j}^{EC}}{\Delta x}$ is a second-order approximation of $\widehat{F}(\widehat{U})_x$ at $(x_i,y_j)$, and $\frac{\mathcal{G}_{i,j+\frac{1}{2}}^{EC} - \mathcal{G}_{i,{j-\frac{1}{2}}}^{EC}}{\Delta y}$ is a second-order approximation of $\widehat{G}(\widehat{U})_y$ at $(x_i,y_j)$, and the numerical source term $\bm S_{i,j}$ is a second-order approximation of $\widehat{S}(\widehat{U})$ at $(x_i,y_j)$. This can be shown directly by applying the Taylor's expansion to the numerical fluxes at $(x_i,y_j)$. For example, consider the first $K$-block term corresponding to $\widehat{h}$ in $\frac{\mathcal{F}_{i+\frac{1}{2},j}^{EC} - \mathcal{F}_{i-\frac{1}{2},j}^{EC}}{\Delta x}$, assuming $\bm U_{i,j} = \widehat{U}(x_i,y_j)$,
\begin{equation}\label{LTE_estimate}
  \frac{(\mathcal{F}_{i+\frac{1}{2},j}^{EC})^h - (\mathcal{F}_{i-\frac{1}{2},j}^{EC})^h}{\Delta x}  =   \frac{\mathcal{P}(\overline{\bm h}_{i+\frac{1}{2},j})\overline{\bm u}_{i+\frac{1}{2},j} - \mathcal{P}(\overline{\bm h}_{i-\frac{1}{2},j})\overline{\bm u}_{i-\frac{1}{2},j}}{\Delta x}  \stackrel{\eqref{eq:taylor-ubar},\eqref{eq:taylor-Phbar}}{=}  \big((\mathcal{P}(\widehat{h})\widehat{u})_{x}\big)_{i,j}+ \mathcal{O}(\Delta x^2 ),
\end{equation}
where $\big((\mathcal{P}(\widehat{h})\widehat{u})_{x}\big)_{i,j}$ is the partial derivative with respect to $x$ of the exact flux at $(x_i,y_j)$. More detailed computations are as follows for one of the terms: We compute the local truncation error of $\frac{(\mathcal{F}_{i+\frac{1}{2},j}^{EC})^h - (\mathcal{F}_{i-\frac{1}{2},j}^{EC})^h}{\Delta x}$ to the exact flux $\big(\widehat{F}^h(\widehat{U})\big)_x  = \big(\mathcal{P}(\widehat{h})\widehat{u}\big)_x$. First, by \eqref{EC_flux},
\begin{equation}\label{appendix_eq1}
   \begin{split}
       &  \frac{(\mathcal{F}_{i+\frac{1}{2},j}^{EC})^h - (\mathcal{F}_{i-\frac{1}{2},j}^{EC})^h}{\Delta x} 
      =   \frac{\mathcal{P}(\overline{\bm h}_{i+\frac{1}{2},j})\overline{\bm u}_{i+\frac{1}{2},j} - \mathcal{P}(\overline{\bm h}_{i-\frac{1}{2},j})\overline{\bm u}_{i-\frac{1}{2},j}}{\Delta x} \\
      = & \frac{1}{4\Delta x}\Bigg( \Big(\mathcal{P}(\widehat{h}_{i,j}) + \mathcal{P}(\widehat{h}_{i+1,j}) \Big)(\widehat{u}_{i,j}+\widehat{u}_{i+1,j}) - \Big(\mathcal{P}(\widehat{h}_{i,j}) + \mathcal{P}(\widehat{h}_{i-1,j}) \Big)(\widehat{u}_{i,j}+\widehat{u}_{i-1,j}) \Bigg).
   \end{split}
\end{equation}
Then, we apply the Taylor's expansion of the crossing product terms in the second equality, i.e., $\mathcal{P}(\widehat{h}_{i,j})\widehat{u}_{i+1,j}$, $ \mathcal{P}(\widehat{h}_{i+1,j})\widehat{u}_{i,j}$, $ \mathcal{P}(\widehat{h}_{i,j})\widehat{u}_{i-1,j}$ and $\mathcal{P}(\widehat{h}_{i-1,j})\widehat{u}_{i,j}$ at $(x_i,y_j)$. This results in
\begin{equation}
    \begin{split}\label{appendix_eq2}
        \eqref{appendix_eq1} = & \frac{1}{4\Delta x} \Bigg( \mathcal{P}(\widehat{h}_{i,j})\widehat{u}_{i,j} + \mathcal{P}(\widehat{h}_{i,j}) \Big( \widehat{u}_{i,j}+ \Delta x (\widehat{u}_x)_{i,j} + \frac{\Delta x^2}{2}(\widehat{u}_{xx})_{i,j} ) +  \frac{\Delta x^3}{6}(\widehat{u}_{xxx})_{i,j} )  + \mathcal{O}(\Delta x^4) \Big) \\
      & + \widehat{u}_{i,j}\Big(\mathcal{P}(\widehat{h}_{i,j})+ \Delta x (\mathcal{P}(\widehat{h}_x))_{i,j} + \frac{\Delta x^2}{2}(\mathcal{P}(\widehat{h}_{xx}))_{i,j} )+ \frac{\Delta x^3}{6}(\mathcal{P}(\widehat{h}_{xxx}))_{i,j} ) +\mathcal{O}(\Delta x^4)    \Big) +  (\mathcal{P}(\widehat{h})\widehat{u})_{i+1,j} \Bigg) \\
      & - \frac{1}{4\Delta x} \Bigg( \mathcal{P}(\widehat{h}_{i,j})\widehat{u}_{i,j} + \mathcal{P}(\widehat{h}_{i,j}) \Big( \widehat{u}_{i,j}- \Delta x (\widehat{u}_x)_{i,j} + \frac{\Delta x^2}{2}(\widehat{u}_{xx})_{i,j} ) - \frac{\Delta x^3}{6}(\widehat{u}_{xxx})_{i,j} )+ \mathcal{O}(\Delta x^4) \Big) \\
      & + \widehat{u}_{i,j}\Big(\mathcal{P}(\widehat{h}_{i,j})- \Delta x (\mathcal{P}(\widehat{h}_x))_{i,j} + \frac{\Delta x^2}{2}(\mathcal{P}(\widehat{h}_{xx}))_{i,j} ) - \frac{\Delta x^3}{6}(\mathcal{P}(\widehat{h}_{xxx}))_{i,j} )+ \mathcal{O}(\Delta x^4)    \Big) +  (\mathcal{P}(\widehat{h})\widehat{u})_{i-1,j} \Bigg).
    \end{split}
\end{equation}
After combining the terms in the same order, we get
\begin{equation}\label{appendix_eq3}
\begin{split}
    \eqref{appendix_eq2} =& \frac{1}{4\Delta x} \Big(  2\Delta x \mathcal{P}(\widehat{h}_{i,j})(\widehat{u}_x)_{i,j} + 2\Delta x \mathcal{P}(\widehat{h}_x)_{i,j}(\widehat{u})_{i,j} + \frac{\Delta x^3}{3}(\widehat{u}_{xxx})_{i,j} \Big) \\
    &+ \frac{1}{4\Delta x} \Big(  (\mathcal{P}(\widehat{h})\widehat{u})_{i+1,j} -  (\mathcal{P}(\widehat{h})\widehat{u})_{i-1,j} \Big)+ \mathcal{O}(\Delta x^3 ). 
\end{split}
\end{equation}
Finally, we again apply the Taylor's expansion to $(\mathcal{P}(\widehat{h})\widehat{u})_{i+1,j} -  (\mathcal{P}(\widehat{h})\widehat{u})_{i-1,j}$ to obtain a second-order approximation of $\big(\mathcal{P}(\widehat{h})\widehat{u}\big)_x$ at $(x_i,y_j)$
\begin{equation}
    \begin{split}
        \eqref{appendix_eq3} = & \frac{1}{2}((\mathcal{P}(\widehat{h})\widehat{u})_{x})_{i,j} + \frac{\Delta x^2}{12}(\widehat{u}_{xxx})_{i,j}+ \frac{1}{4\Delta x} \Big( 2\Delta x ((\mathcal{P}(\widehat{h})\widehat{u})_{x})_{i,j} + \frac{\Delta x^3}{3}\big((\mathcal{P}(\widehat{h})\widehat{u})_{xxx}\big)_{i,j} \Big) + \mathcal{O}(\Delta x^3 )\\
        = & ((\mathcal{P}(\widehat{h})\widehat{u})_{x})_{i,j}+ \frac{\Delta x^2}{12} \Big( (\widehat{u}_{xxx})_{i,j}+ \big((\mathcal{P}(\widehat{h})\widehat{u})_{xxx}\big)_{i,j}  \Big)+ \mathcal{O}(\Delta x^3 ).
    \end{split}
\end{equation}
Similarly, other terms in the numerical fluxes can be evaluated by the Taylor's expansion to verify the second-order local truncation error of the energy conservative scheme.

  In this way, one can directly compute the local truncation error to show that the scheme $\eqref{FV-SGSWE}$, with the numerical fluxes and source term \eqref{EC_flux}, is second-order spatially accurate.

\subsection{Proof of \cref{lemma_wellbalanced}}\label{sec:wellbalanced}
Suppose we are given the initial data
    \begin{align}\label{well_balanced_IC}
        \bm u_{i,j} = \bm v_{i,j} &\equiv \bm 0, & \bm h_{i,j} +\bm B_{i,j} &= \text{const vector}, &\forall i,j.
    \end{align}
The time-independent bottom topography leads to $\frac{d}{dt}\bm B_{i,j} = 0, \forall i,j$. To prove the well-balanced property, it suffices to show 
\begin{align}\label{sufficient_well_balanced}
    \frac{d}{dt} \bm h_{i,j} &\equiv 0, & \frac{d}{dt}\bm q^x_{i,j} &\equiv 0, &  \frac{d}{dt}\bm q^y_{i,j} &\equiv 0 ,& \forall i,j.
\end{align}
By the straightforward substitution and the fact that the discretization of the velocities are both zero vectors, due to \eqref{sufficient_well_balanced},
\begin{equation}
\begin{split}
     \frac{d}{dt}\bm h_{i,j} =& -\frac{1}{\Delta x}\left( \mathcal{P}(\overline{\bm h}_{i+\frac{1}{2},j})\overline{\bm u}_{i+\frac{1}{2},j} -  \mathcal{P}(\overline{\bm h}_{i-\frac{1}{2},j})\overline{\bm u}_{i-\frac{1}{2},j} \right)\\
     & - \frac{1}{\Delta y}\left( \mathcal{P}(\overline{\bm h}_{i,j+\frac{1}{2}})\overline{\bm v}_{i,j+\frac{1}{2}} -  \mathcal{P}(\overline{\bm h}_{i,j-\frac{1}{2}})\overline{\bm v}_{i,j-\frac{1}{2}} \right)\\
     =& \bm 0.
\end{split}
\end{equation}
Before investigating the time derivative of the discharge, we first introduce useful identities:
\begin{equation}\label{identity1}
\begin{split}
     \left( \overline{\mathcal{P}(\bm h) \bm h}  \right)_{i+\frac{1}{2},j} -  \left( \overline{\mathcal{P}(\bm h) \bm h}  \right)_{i-\frac{1}{2},j} & \overset{\eqref{average_jump}}{=} \frac{1}{2}\left( \llbracket \mathcal{P}(\bm h)\bm h\rrbracket_{i+\frac{1}{2},j} + \llbracket \mathcal{P}(\bm h)\bm h\rrbracket_{i-\frac{1}{2},j} \right)  \\
     & \overset{\eqref{identity0}}{=} \mathcal{P}(\overline{\bm h}_{i+\frac{1}{2},j})\llbracket \bm h\rrbracket_{i+\frac{1}{2},j} + \mathcal{P}(\overline{\bm h}_{i-\frac{1}{2},j})\llbracket \bm h\rrbracket_{i-\frac{1}{2},j},\\
 \left( \overline{\mathcal{P}(\bm h) \bm h}  \right)_{i,j+\frac{1}{2}} -  \left( \overline{\mathcal{P}(\bm h) \bm h}  \right)_{i,j-\frac{1}{2}} & \overset{\eqref{average_jump}}{=} \frac{1}{2}\left( \llbracket \mathcal{P}(\bm h)\bm h\rrbracket_{i,j+\frac{1}{2}} + \llbracket \mathcal{P}(\bm h)\bm h\rrbracket_{i,j-\frac{1}{2}} \right)  \\
     & \overset{\eqref{identity0}}{=} \mathcal{P}(\overline{\bm h}_{i,j+\frac{1}{2}})\llbracket \bm h\rrbracket_{i,j+\frac{1}{2}} + \mathcal{P}(\overline{\bm h}_{i,j-\frac{1}{2}})\llbracket \bm h\rrbracket_{i,j-\frac{1}{2}}.
\end{split}
\end{equation}
Now by straightforward substitution, using identities \eqref{identity1} and the initial data, we obtain the following result:
\begin{equation}
    \begin{split}
        \frac{d}{dt}\bm q^x_{i,j} \overset{\eqref{EC_flux}}{=} & - \frac{g}{2\Delta x}\left( \left( \overline{\mathcal{P}(\bm h) \bm h}  \right)_{i+\frac{1}{2},j} -  \left( \overline{\mathcal{P}(\bm h) \bm h}  \right)_{i-\frac{1}{2},j}  \right)\\
        & - \frac{g}{2\Delta x}\left(  \mathcal{P}(\overline{\bm h}_{i+\frac{1}{2},j})\llbracket \bm B\rrbracket_{i+\frac{1}{2},j} + \mathcal{P}(\overline{\bm h}_{i-\frac{1}{2},j})\llbracket \bm B\rrbracket_{i-\frac{1}{2},j} \right)\\
        \overset{\eqref{identity1}}{=}&  - \frac{g}{2\Delta x}\left(  \mathcal{P}(\overline{\bm h}_{i+\frac{1}{2},j})\llbracket \bm h + \bm B\rrbracket_{i+\frac{1}{2},j} + \mathcal{P}(\overline{\bm h}_{i-\frac{1}{2},j})\llbracket \bm h +\bm B\rrbracket_{i-\frac{1}{2},j} \right) \overset{\eqref{well_balanced_IC}}{=} \bm 0,\\
         \frac{d}{dt}\bm q^y_{i,j} \overset{\eqref{EC_flux}}{=} & - \frac{g}{2\Delta y}\left( \left( \overline{\mathcal{P}(\bm h) \bm h}  \right)_{i,j+\frac{1}{2}} -  \left( \overline{\mathcal{P}(\bm h) \bm h}  \right)_{i,j-\frac{1}{2}}  \right)\\
        & - \frac{g}{2\Delta y}\left(  \mathcal{P}(\overline{\bm h}_{i,j+\frac{1}{2}})\llbracket \bm B\rrbracket_{i,j+\frac{1}{2}} + \mathcal{P}(\overline{\bm h}_{i,j-\frac{1}{2}})\llbracket \bm B\rrbracket_{i,j-\frac{1}{2}} \right)\\
        \overset{\eqref{identity1}}{=}&  - \frac{g}{2\Delta y}\left(  \mathcal{P}(\overline{\bm h}_{i,j+\frac{1}{2}})\llbracket \bm h + \bm B\rrbracket_{i,j+\frac{1}{2}} + \mathcal{P}(\overline{\bm h}_{i,j-\frac{1}{2}})\llbracket \bm h +\bm B\rrbracket_{i,j-\frac{1}{2}} \right) \overset{\eqref{well_balanced_IC}}{=} \bm 0,\\
    \end{split}
\end{equation}
which completes the proof.

\subsection{Proof of \cref{lemma_sufficientcond_EC}}\label{sec:sufficien_condition_EC}

  By multiplying both sides of \eqref{FV-SGSWE} by $\bm V_{i,j}^\top$ and using the definition $\bm V_{i,j} \coloneqq (\frac{\partial \bm E_{i,j}}{\partial \bm U_{i,j}})^\top$, we obtain 
  \begin{equation}\label{discrete_energy_dev}
      \frac{d}{dt}\bm E_{i,j} = -\frac{1}{\Delta x}\left( \bm V_{i,j}^\top \mathcal{F}_{i+\frac{1}{2},j} - \bm V_{i,j}^\top \mathcal{F}_{i-\frac{1}{2},j} \right) - \frac{1}{\Delta y} \left( \bm V_{i,j}^\top \mathcal{G}_{i,j+\frac{1}{2}} - \bm V_{i,j}^\top \mathcal{G}_{i,j-\frac{1}{2}} \right) + \bm V_{i,j}^\top \bm S_{i,j}.
  \end{equation}
We estimate these terms separately. The first term on the right-hand side can be expanded as follows:
\begin{equation}\label{VF_plus}
    \begin{split}
        \bm V_{i,j}^\top \mathcal{F}_{i+\frac{1}{2},j} \overset{\eqref{identity_ave_jump}}{=}   &\overline{\bm V }_{i+\frac{1}{2},j}^\top \mathcal{F}_{i+\frac{1}{2},j} - \frac{1}{2}\llbracket \bm V \rrbracket_{i+\frac{1}{2},j}^\top \mathcal{F}_{i+\frac{1}{2},j}\\
           \overset{\eqref{sufficient_condition_EC}\eqref{energy_flux}}{ = } & \mathcal{H}_{i+\frac{1}{2},j} + \overline{\bm \Psi}_{i+\frac{1}{2},j} + \frac{g}{4}\llbracket \bm B \rrbracket_{i+\frac{1}{2},j}^\top\mathcal{P}(\overline{h}_{i+\frac{1}{2},j})\llbracket \bm u \rrbracket_{i+\frac{1}{2},j} \\
           & -\frac{1}{2}\llbracket \bm \Psi \rrbracket_{i+\frac{1}{2},j} - \frac{g}{2}\llbracket\bm B \rrbracket_{i+\frac{1}{2},j}^\top \mathcal{P}(\overline{\bm h}_{i+\frac{1}{2},j})\overline{\bm u}_{i+\frac{1}{2},j} \\
           \overset{\eqref{identity_ave_jump}}{=} & \mathcal{H}_{i+\frac{1}{2},j} + \bm \Psi_{i,j} - \frac{g}{2}\llbracket\bm B \rrbracket_{i+\frac{1}{2},j}^\top \mathcal{P}(\overline{\bm h}_{i+\frac{1}{2},j}){\bm u}_{i,j}.
    \end{split}
\end{equation}
Using an analogous computation, we can also obtain the second term,
\begin{equation}\label{VF_minus}
    \bm V_{i,j}^\top \mathcal{F}_{i-\frac{1}{2},j} =  \mathcal{H}_{i-\frac{1}{2},j} + \bm \Psi_{i,j} + \frac{g}{2}\llbracket\bm B \rrbracket_{i-\frac{1}{2},j}^\top \mathcal{P}(\overline{\bm h}_{i-\frac{1}{2},j}){\bm u}_{i,j}.
\end{equation}
Additionally, in a similar manner, and with the corresponding quantities in \eqref{sufficient_condition_EC} and \eqref{energy_flux}, we derive
\begin{equation}\label{VG}
    \begin{split}
        \bm V_{i,j}^\top \mathcal{G}_{i,j+\frac{1}{2}} & = \mathcal{K}_{i,j+\frac{1}{2}} + \bm \Phi_{i,j} - \frac{g}{2}\llbracket \bm B \rrbracket_{i,j+\frac{1}{2}}^\top \mathcal{P}(\overline{\bm h}_{i,j+\frac{1}{2}})\bm v_{i,j},\\
        \bm V_{i,j}^\top \mathcal{G}_{i,j-\frac{1}{2}} & = \mathcal{K}_{i,j-\frac{1}{2}} + \bm \Phi_{i,j} + \frac{g}{2}\llbracket \bm B \rrbracket_{i,j-\frac{1}{2}}^\top \mathcal{P}(\overline{\bm h}_{i,j-\frac{1}{2}})\bm v_{i,j}.\\
    \end{split}
\end{equation}
Finally, a direct computation shows
\begin{equation}\label{VS}
    \begin{split}
        \bm V_{i,j}^\top \bm S_{i,j} = & -\frac{g  \bm u_{i,j}^\top }{2\Delta x} \left( \mathcal{P}(\overline{\bm h}_{i+\frac{1}{2},j}) \llbracket \bm B \rrbracket_{i+\frac{1}{2},j} + \mathcal{P}(\overline{\bm h}_{i-\frac{1}{2},j}) \llbracket \bm B \rrbracket_{i-\frac{1}{2},j}\right) \\
         & -\frac{g  \bm v_{i,j}^\top }{2\Delta y} \left( \mathcal{P}(\overline{\bm h}_{i,j+\frac{1}{2}}) \llbracket \bm B \rrbracket_{i,j+\frac{1}{2}} + \mathcal{P}(\overline{\bm h}_{i,j-\frac{1}{2}}) \llbracket \bm B \rrbracket_{i,j-\frac{1}{2}}\right)
    \end{split}
\end{equation}
The formula \eqref{discrete_energy_dev}, together with the combination of expressions \eqref{VF_plus}, \eqref{VF_minus}, \eqref{VG}, and \eqref{VS}, shows that the scheme \eqref{FV-SGSWE} is energy conservative, with energy fluxes defined in \eqref{energy_flux}.

\end{document}